\documentclass{elsarticle}
\usepackage[top=1.5 in, bottom=1.5 in, left=1.5 in, right=1.5 in]{geometry}

\usepackage{amsmath}
\usepackage{amssymb}

\usepackage{tikz}

\usepackage{caption}
\usepackage{subcaption,comment}

\usepackage{booktabs}

\usepackage{algorithm}
\usepackage{algorithmicx}
\usepackage{algpseudocode}

\usepackage{longtable}

\usepackage{multirow}

\usepackage{adjustbox}  

\usepackage{natbib, setspace} 
 \bibpunct[, ]{(}{)}{,}{a}{}{,}%

\newtheorem{theorem}{Theorem}
\newtheorem{lemma}[theorem]{Lemma}

\newenvironment{proof}[1][Proof]{\textbf{#1:}\\}{\hfill\rule{1.5mm}{1.5mm}\\}
\newenvironment{omitted_proof}[1][Proof]{\textbf{#1:} }{\hfill\rule{1.5mm}{1.5mm}}

\begin{document}

\begin{frontmatter}

\title{Heuristics with Performance Guarantees for the Minimum Number of Matches Problem in Heat Recovery Network Design}
\author{Dimitrios Letsios\fnref{label0}}
\author{Georgia Kouyialis\fnref{label0}}
\author{Ruth Misener\fnref{label0}\corref{cor1}}

\address[label0]{Department of Computing; Imperial College London; South Kensington SW7 2AZ; UK}
\cortext[cor1]{\texttt{\{d.letsios, g.kouyialis14, r.misener\}@imperial.ac.uk}; Tel: +44 (0) 20759 48315}

\begin{abstract}
Heat exchanger network synthesis exploits excess heat by integrating process hot and cold streams and improves energy efficiency by reducing utility usage.
Determining {provably good solutions to} the minimum number of matches is {a} bottleneck of designing a heat recovery network using the sequential method. This subproblem is an $\mathcal{NP}$-hard mixed-integer linear program exhibiting combinatorial explosion in the possible hot and cold stream configurations. We explore this challenging optimization problem from a graph theoretic perspective and correlate it with other special optimization problems such as cost flow network and packing problems. In the case of a single temperature interval, we develop a new optimization formulation without problematic big-M parameters. 
We develop heuristic methods with performance guarantees using three approaches: (i) relaxation rounding, (ii) water filling, and (iii) greedy packing. Numerical results from a collection of {51 instances} substantiate the strength of the methods.
\end{abstract}


\end{frontmatter}



\onehalfspacing

\noindent
\emph{This manuscript is dedicated, with deepest respect, to the memory of Professor C.\ A.\ Floudas. Professor Floudas showed that, given many provably-strong solutions to the minimum number of matches problem, he could design effective heat recovery networks. So the diverse solutions generated by this manuscript directly improve Professor Floudas' method for automatically generating heat exchanger network configurations.}

\section{Introduction}
\label{sec:introduction}

Heat exchanger network synthesis (HENS) 
minimizes cost and improves energy recovery in chemical processes \citep{biegleretal:1997, smith:2000, eliaetal:2010, balibanetal:2012}.  
HENS exploits excess heat by integrating process hot and cold streams and improves energy efficiency by reducing utility usage \citep{floudas:1987, naess:1988, furman:2002, escobar-et-trierweiler:2013}. 
\citet{floudas:2012} review the critical role of heat integration for energy systems producing liquid transportation fuels \citep{niziolek:2015}.
Other important applications of HENS include: refrigeration systems \citep{shelton:1986}, batch semi-continuous processes \citep{zhao:1998,castro:2015} and water utilization systems \citep{bagakewicz:2002}.



Heat exchanger network design is a mixed-integer nonlinear optimization (MINLP) problem \citep{yee:1990, Ciric:1991, papalexandri-pistikopoulos:1994,hasan-etal:2010}. \cite{mistry:2016} recently showed that expressions incorporating logarithmic mean temperature difference, i.e.\ the nonlinear nature of heat exchange, may be reformulated to decrease the number of nonconvex nonlinear terms in the optimization problem. But HENS remains a difficult MINLP with many nonconvex nonlinearities. One way to generate good HENS solutions is to use the so-called \emph{sequential method} \citep{furman:2002}.
The sequential method decomposes the original HENS MINLP into three tasks: (i) minimizing utility cost, (ii) minimizing the number of matches, and (iii) minimizing the investment cost. 
The method optimizes the three mathematical models sequentially with: (i) a linear program (LP) \citep{cerdaetwesterberg:1983, papoulias:1983}, (ii) a mixed-integer linear program (MILP) \citep{cerda:1983, papoulias:1983}, and (iii) a nonlinear program (NLP) \citep{floudas:1986}.
The sequential method may not return the global solution of the original MINLP, but solutions generated with the sequential method are practically useful.

This paper investigates the \emph{minimum number of matches} problem \citep{floudas}, the computational bottleneck of the sequential method. 
The minimum number of matches problem is a strongly $\mathcal{NP}$-hard MILP \citep{furman:2001}.
Mathematical symmetry in the problem structure combinatorially increases the possible stream configurations and deteriorates the performance of exact, tree-based algorithms \citep{kouyialis:2016}. 

Because state-of-the-art approaches cannot solve the minimum number of matches problem to global optimality for moderately-sized instances \citep{chen:2015}, engineers develop experience-motivated heuristics \citep{hindmarsh:1983,cerdaetwesterberg:1983}. {\cite{hindmarsh:1983} highlight the importance of generating good solutions quickly: a design engineer may want to actively interact with a good minimum number of matches solution and consider changing the utility usage as a result of the MILP outcome.}
%
\citet{furman:2004} propose a collection of approximation algorithms, i.e.\ heuristics with performance guarantees, for the minimum number of matches problem by exploiting the LP relaxation of an MILP formulation.
\citet{furman:2004} present a unified worst-case analysis of their algorithms' performance guarantees and show a non-constant approximation ratio scaling with the number of temperature intervals.
They also prove a constant performance guarantee for the single temperature interval problem.

The standard MILP formulations for the minimum number of matches contain big-M constraints, i.e.\ the on/off switches associated with weak continuous relaxations of MILP.
Both optimization-based heuristics and exact state-of-the-art methods for solving minimum number of matches problem are highly affected by the big-M parameter.
Trivial methods for computing the big-M parameters are typically adopted, but \citet{gundersen:1997} propose a tighter way of computing the big-M parameters.

This manuscript develops new heuristics and provably efficient approximation algorithms for the minimum number of matches problem. These methods have guaranteed solution quality and efficient run-time bounds. 
In the sequential method, many possible stream configurations are required to evaluate the minimum overall cost \citep{floudas}, so a complementary contribution of this work is a heuristic methodology for producing multiple solutions efficiently.
We classify the heuristics based on their algorithmic nature into three categories: (i) relaxation rounding, (ii) water filling, and (iii) greedy packing.

{
The relaxation rounding heuristics we consider are (i) Fractional LP Rounding (FLPR), (ii) Lagrangian Relaxation Rounding (LRR), and (iii) Covering Relaxation Rounding (CRR).
The water-filling heuristics are (i) Water-Filling Greedy (WFG), and (ii) Water-Filling MILP (WFM). 
Finally, the greedy packing heuristics are (i) Largest Heat Match LP-based (LHM-LP), (ii) Largest Heat Match Greedy (LHM), (iii) Largest Fraction Match (LFM), and (iv) Shortest Stream (SS).
Major ingredients of these heuristics are adaptations of single temperature interval algorithms and maximum heat computations with match restrictions.
We propose (i) a novel MILP formulation, and (ii) an improved greedy approximation algorithm for the single temperature interval problem.
Furthermore, we present (i) a greedy algorithm computing maximum heat between two streams and their corresponding big-M parameter, (ii) an LP computing the maximum heat in a single temperature interval using a subset of matches, and (iii) an extended maximum heat LP using a subset of matches on multiple temperature intervals.
}

The manuscript proceeds as follows: 
Section \ref{sec:preliminaries} formally defines the minimum number of matches problem and discusses mathematical models.
Section \ref{sec:heuristics_performance_guarantees} discusses computational complexity and introduces a new $\mathcal{NP}$-hardness reduction of the minimum number of matches problem from bin packing. 
Section \ref{sec:single_temperature_interval} 
focusses on the single temperature interval problem.
Section \ref{sec:max_heat} explores computing the maximum heat exchanged between the streams with match restrictions.
Sections \ref{sec:relaxation_rounding} - \ref{sec:greedy_packing} present our heuristics for the minimum number of matches problem based on: (i) relaxation rounding, (ii) water filling, and (iii) greedy packing, respectively, as well as new theoretical performance guarantees.
Section \ref{sec:results} evaluates experimentally the heuristics and discusses numerical results.
Sections \ref{sec:discussion} and \ref{sec:conclusion} discuss the manuscript contributions and conclude the paper.

\section{Minimum Number of Matches for Heat Exchanger Network Synthesis}
\label{sec:preliminaries}
This section defines the minimum number of matches problem and presents the standard transportation and transshipment MILP models. 
Table \ref{tbl:notation} contains the notation.

\singlespacing
\begin{longtable}{l l l}
\caption{Nomenclature}\\
\toprule
Name & Description \\
\midrule
{\bf Cardinalities} & \\
$n$ & Number of hot streams \\
$m$ & Number of cold streams \\
$k$ & Number of temperature intervals\\
$v$ & Number of matches (objective value) \\
\midrule
{\bf Indices} \\
$i\in H$ & Hot stream\\
$j\in C$ & Cold stream\\
$s,t,u \in T$ & Temperature interval\\
$b\in B$ & Bin (single temperature interval problem) \\
\midrule
{\bf Sets} & & \\
$H$, $C$ & Hot, cold streams \\
$T$ & Temperature intervals \\
$M$ & Set of matches (subset of $H\times C$) \\
$C_i(M), H_j(M)$ & Cold, hot streams matched with $i\in H$, $j\in C$ in $M$ \\
$B$ & Bins (single temperature interval problem) \\
$A(M)$ & Set of valid quadruples $(i,s,j,t)$ with respect to a set $M$ of matches \\
$A_u(M)$ & Set of quadruples $(i,s,j,t)\in A(M)$ with $s\leq u<t$ \\
$V^H(M)$ & Set of pairs $(i,s)\in H\times T$ appearing in $A(M)$ (transportation vertices) \\
$V^C(M)$ & Set of pairs $(j,t)\in C\times T$ appearing in $A(M)$ (transportation vertices) \\
$V_{i,s}^C(M)$ & Set of pairs $(j,t)\in V^C(M)$ such that $(i,s,j,t)$ belongs to $A(M)$ \\
$V_{j,t}^H(M)$ & Set of pairs $(i,s)\in V^H(M)$ such that $(i,s,j,t)$ belongs to $A(M)$ \\
\midrule
{\bf Parameters} & \\
$h_i$ & Total heat supplied by hot stream $i$ ($h_i=\sum_{s\in T}\sigma_{i,s}$) \\
$h_{\max}$ & Maximum heat among all hot streams ($h_{\max}=\max_{i\in H}\{h_i\}$) \\
$c_j$ & Total heat demanded by cold stream $j$ ($c_j=\sum_{t\in T}\delta_{j,t}$) \\
$\sigma_{i,s}$ & Heat supply of hot stream $i$ in interval $s$ \\
$\delta_{j,t}$ & Heat demand of cold stream $j$ in interval $t$ \\
$\vec{\sigma}, \vec{\delta}$ & Vectors of all heat supplies, demands \\ 
$\vec{\sigma}_t, \vec{\delta}_t$ & Vectors of all heat supplies, demands in temperature interval $t$ \\ 
$R_{t}$ & Residual heat exiting temperature interval $t$\\
$U_{i,j}$ & Upper bound (big-M parameter) on the heat exchanged via match $(i,j)$ \\
$\lambda_{i,j}$ & Fractional cost approximation of match $(i,j)$ (Lagrangian relaxation) \\
$\vec{\lambda}$ & Vector of all fractional cost approximations $\lambda_{i,j}$ \\
\midrule
{\bf Variables} & \\
$y_{i,j}$ & Binary variable indicating whether $i$ and $j$ are matched \\
$q_{i,j,t}$ & Heat of hot stream $i$ received by cold stream $j$ in interval $t$ \\
$q_{i,s,j,t}$ & Heat exported by hot stream $i$ in $s$ and received by cold stream $j$ in $t$ \\
$\vec{y}, \vec{q}$ & Vectors of binary, continuous variables \\
$r_{i,s}$ & Heat residual of heat of hot stream $i$ exiting $s$ \\
$x_b$ & Binary variable indicating whether bin $b$ is used \\
$w_{i,b}$ & Binary variable indicating whether hot stream $i$ is placed in bin $b$ \\
$z_{j,b}$ & Binary variable indicating whether cold stream $j$ is placed in bin $b$ \\
\midrule
{\bf Other} & \\
$N$ & Minimum cost flow network  \\
$G$ & Solution graph (single temperature interval problem) \\
$\phi(M)$ & Filling ratio of a set $M$ of matches \\
$\vec{y}^f, \vec{q}^f$ & Optimal fractional solution \\
$\alpha_i, \beta_j$ & Number of matches of hot stream $i$, cold stream $j$ \\
$L_{i,j}$ & Heat exchanged from hot stream $i$ to cold stream $j$ \\
$I$ & Instance of the problem \\
$r$ & Remaining heat of an algorithm \\
\bottomrule
\label{tbl:notation}
\end{longtable}
\onehalfspacing

\subsection{Problem Definition}

{

Heat exchanger network design involves a set $HS$ of \emph{hot process streams} to be cooled and a set $CS$ of \emph{cold process streams} to be heated.
Each hot stream $i$ posits an initial temperature $T_{\text{in},i}^{HS}$ and a target temperature $T_{\text{out},i}^{HS}$ ($<T_{\text{in},i}^{HS}$).
Analogously, each cold stream $j$ has an initial temperature $T_{\text{in},j}^{CS}$ and a target temperature $T_{\text{out},j}^{CS}$ ($>T_{\text{in},j}^{CS}$).
Every hot stream $i$ and cold stream $j$ are associated flow rate heat capacities $FCp_i$ and $FCp_j$, respectively.
Minimum heat recovery approach temperature $\Delta T_{\min}$ relates the hot and cold stream temperature axes.
A hot utility $i$ in a set $HU$ and a cold utility $j$ in a set $CU$ may be purchased at a cost, e.g.\ with unitary costs $\kappa_i^{HU}$ and $\kappa_j^{CU}$.
Like the streams, the utilities have inlet and outlet temperatures $T_{\text{in},i}^{HU}$, $T_{\text{out},i}^{HU},T_{\text{in},j}^{CU}$ and $T_{\text{out},j}^{CU}$.
The first step in a sequential approach to HENS minimizes the utility cost and thereby specifies the heat each utility introduces in the network.
The next step minimizes the number of matches.
\ref{App:Minimum_Utility_Cost} discusses the transition from the minimizing utility cost to minimizing the number of matches.
After this transition, each utility may, without loss of generality, be treated as a stream.
}

The minimum number of matches problem posits a set of \emph{hot process streams} to be cooled and a set of \emph{cold process streams} to be heated.
Each stream is associated with an initial and a target temperature. 
This set of temperatures defines a collection of \emph{temperature intervals}. 
Each hot stream exports (or supplies) heat in each temperature interval between its initial and target temperatures.
Similarly, each cold stream receives (or demands) heat in each temperature interval between its initial and target temperatures.
\ref{App:Minimum_Utility_Cost} formally defines the temperature range partitioning.
Heat may flow from a hot to a cold stream in the same or a lower temperature interval, but not in a higher one.
In each temperature interval, the \emph{residual heat} descends to lower temperature intervals.
A zero heat residual is a \emph{pinch point}.
A pinch point restricts the maximum energy integration and divides the network into subnetworks. 

A problem instance consists of a set $H=\{1,2,\ldots,n\}$ of hot streams, a set $C=\{1,2,\ldots,m\}$ of cold streams, and a set $T=\{1,2,\ldots,k\}$ of temperature intervals. 
Hot stream $i\in H$ has heat supply $\sigma_{i,s}$ in temperature interval $s\in T$ and cold stream $j\in C$ has heat demand $\delta_{j,t}$ in temperature interval $t\in T$.
Heat conservation is satisfied, i.e.\ $\sum_{i\in H}\sum_{s\in T}\sigma_{i,s} = \sum_{j\in C}\sum_{t\in T}\delta_{j,t}$.
We denote by $h_i=\sum_{s\in T}\sigma_{i,s}$ the total heat supply of hot stream $i\in H$ and by $c_j=\sum_{t\in T}\delta_{j,t}$ the total heat demand of cold stream $j\in C$.

A feasible solution specifies a way to transfer the hot streams' heat supply to the cold streams, i.e.\ an amount $q_{i,s,j,t}$ of heat exchanged between hot stream $i\in H$ in temperature interval $s\in T$ and cold stream $j\in C$ in temperature interval $t\in T$.
Heat may only flow to the same or a lower temperature interval, i.e.\ $q_{i,s,j,t}=0$, for each $i\in H$, $j\in C$ and $s,t\in T$ such that $s>t$.
A hot stream $i\in H$ and a cold stream $j\in C$ are \emph{matched}, if there is a positive amount of heat exchanged between them, i.e.\ $\sum_{s,t\in T}q_{i,s,j,t}>0$.
The objective is to find a feasible solution minimizing the number of matches $(i,j)$.

\subsection{Mathematical Models}
The transportation and transshipment models formulate the minimum number of matches as a mixed-integer linear program (MILP).

\paragraph{Transportation Model \citep{cerda:1983}} 
As illustrated in Figure \ref{Fig:transportation}, the transportation model represents heat as a commodity transported from supply nodes to destination nodes.
For each hot stream $i\in H$, there is a set of supply nodes, one for each temperature interval $s\in T$ with $\sigma_{i,s}>0$.
For each cold stream $j\in C$, there is a set of demand nodes, one for each temperature interval $t\in T$ with $\delta_{j,t}>0$.
There is an arc between the supply node $(i,s)$ and the destination node $(j,t)$ if $s\leq t$, for each $i\in H$, $j\in C$ and $s,t\in T$.

In the MILP formulation, variable $q_{i,s,j,t}$ specifies the heat transferred from hot stream $i\in H$ in temperature interval $s\in T$ to cold stream $j\in C$ in temperature interval $t\in T$. 
Binary variable $y_{i,j}$ if whether streams $i\in H$ and $j\in C$ are matched or not. 
Parameter $U_{i,j}$ is a big-M parameter bounding the amount of heat exchanged between every pair of hot stream $i\in H$ and cold stream $j\in C$, e.g.\ $U_{i,j}=\min\{h_i,c_j\}$.
The problem is formulated:
{\allowdisplaybreaks
\begin{align}
\text{min} & \sum_{i \in H}\sum_{j \in C} y_{i,j} \label{TransportationMIP_Eq:ObjMinMatches} \\ 
& \sum_{j\in C}\sum_{t\in T} q_{i,s,j,t} = \sigma_{i,s} & i\in H, s\in T \label{TransportationMIP_Eq:HotStreamConservation}\\
& \sum_{i\in H}\sum_{s\in T} q_{i,s,j,t} = \delta_{j,t} & j\in C, t\in T \label{TransportationMIP_Eq:ColdStreamConservation}\\
& \sum_{s,t\in T} q_{i,s,j,t}\leq U_{i,j}\cdot y_{i,j} & i\in H, j\in C \label{TransportationMIP_Eq:BigM_Constraint}\\
& q_{i,s,j,t} = 0 & i\in H, j\in C, s,t\in T: s> t \label{TransportationMIP_Eq:ThermoConstraint} \\
& y_{i,j} \in \{0,\,1\},q_{i,s,j,t}\geq 0 & i\in H, j\in C,\; s,t\in T \label{TransportationMIP_Eq:IntegralityConstraints}
\end{align}
}
Expression (\ref{TransportationMIP_Eq:ObjMinMatches}), the objective function, minimizes the number of matches.
Equations (\ref{TransportationMIP_Eq:HotStreamConservation}) and (\ref{TransportationMIP_Eq:ColdStreamConservation}) ensure heat conservation.
Equations (\ref{TransportationMIP_Eq:BigM_Constraint}) enforce a match between a hot and a cold stream if they exchange a positive amount of heat.
Equations (\ref{TransportationMIP_Eq:BigM_Constraint}) are \emph{big-M constraints}.
Equations (\ref{TransportationMIP_Eq:ThermoConstraint}) ensure that no heat flows to a hotter temperature.

{

The transportation model may be reduced by removing redundant variables and constraints.
Specifically, a mathematically-equivalent \emph{reduced transportation MILP model} removes: (i) all variables $q_{i,s,j,t}$ with $s> t$ and (ii) Equations (\ref{TransportationMIP_Eq:ThermoConstraint}).
But modern commercial MILP solvers may detect redundant variables constrained to a fixed value and exploit this information to their benefit. Table \ref{Table:Transportation_Models_Comparison} shows that the aggregate performance of CPLEX and Gurobi is unaffected by the redundant constraints and variables.

}

\begin{figure*}[t!]
\centering

\begin{subfigure}[t]{0.45\textwidth}
\centering
\includegraphics{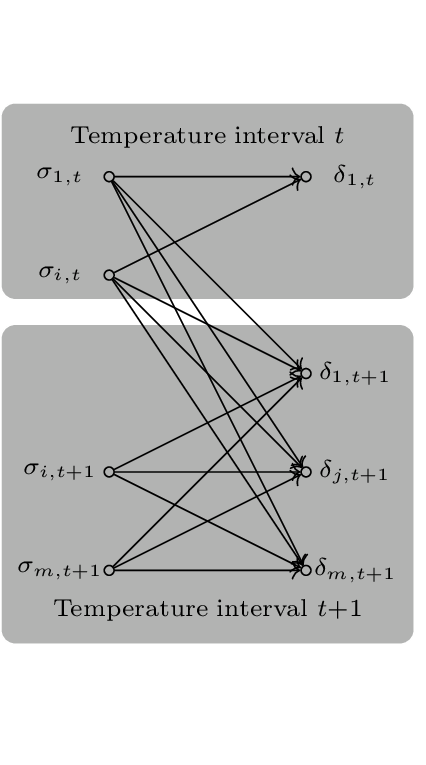}
\vspace*{-1cm}
\caption{ Transportation Model}
\label{Fig:transportation}
\end{subfigure}
\begin{subfigure}[t]{0.45\textwidth}
\centering
\includegraphics{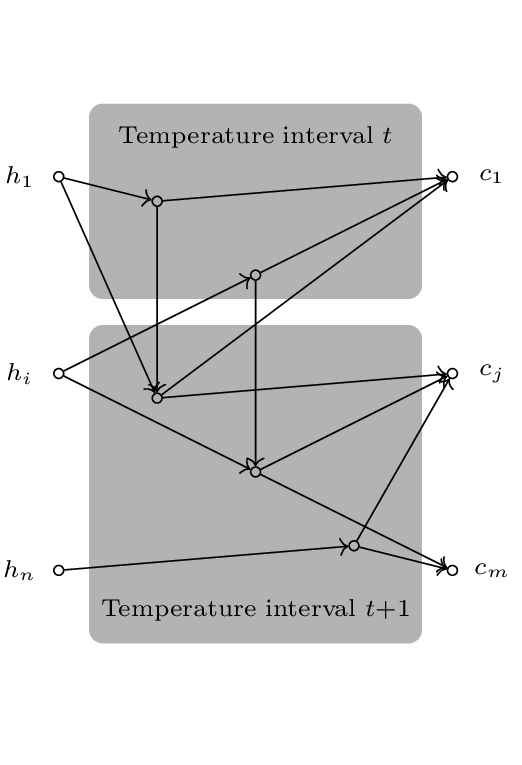}
\vspace*{-1cm}
\caption{Transshipment Model}
\label{Fig:transshipment}
\end{subfigure}
\vspace*{-0.5cm}
\caption{
In the transportation model \citep{cerda:1983}, each hot stream $i$ supplies $\sigma_{i,t}$ units of heat in temperature interval $t$ which can be received, in the same or a lower temperature interval, by a cold stream $j$ which demands $\delta_{j,t}$ units of heat in $t$. 
In the transshipment model \citep{papoulias:1983}, there are also intermediate nodes transferring residual heat to a lower temperature interval.
This figure is adapted from \citet{furman:2004}.
}
\end{figure*}

\paragraph{Transshipment Model \citep{papoulias:1983}} 
As illustrated in Figure \ref{Fig:transshipment}, the transshipment formulation transfers heat from hot streams to cold streams via intermediate transshipment nodes.
In each temperature interval, the heat entering a transshipment node either transfers to a cold stream in the same temperature interval or it descends to the transshipment node of the subsequent temperature interval as residual heat.

Binary variable $y_{i,j}$ is 1 if hot stream $i\in H$ is matched with cold stream $j\in C$ and 0 otherwise.
Variable $q_{i,j,t}$ represents the heat received by cold stream $j\in C$ in temperature interval $t\in T$ originally exported by hot stream $i\in H$.
Variable $r_{i,s}$ represents the residual heat of hot stream $i\in H$ that descends from temperature interval $s$ to temperature interval $s+1$. 
Parameter $U_{i,j}$ is a big-M parameter bounding the heat exchanged between hot stream $i\in H$ and cold stream $j\in C$, e.g.\ $U_{i,j}=\min\{h_i,c_j\}$.
The problem is formulated:
{\allowdisplaybreaks
\begin{align}
\text{min} & \sum_{i\in H}\sum_{j \in C} y_{i,j} \label{TransshipmentMIP_Eq:ObjMinMatches}  \\ 
& \sum_{j \in C} q_{i,j,s} + r_{i,s} = \sigma_{i,s} + r_{i,s-1} & i\in H, s\in T \label{TransshipmentMIP_Eq:HotStreamConservation} \\
& r_{i,k} = 0 & i\in H \label{TransshipmentMIP_Eq:HeatConservation} \\
& \sum_{i\in H} q_{i,j,t}  = \delta_{j,t} &  j\in C, t \in T \label{TransshipmentMIP_Eq:ColdStreamConservation} \\ 
& \sum_{t\in T} q_{i,j,t}\leq U_{i,j}\cdot y_{i,j} & i\in H, j\in C \label{TransshipmentMIP_Eq:BigM_Constraint} \\
& y_{i,j}\in \{0,1\},\; q_{i,j,t}, r_{i,s}\geq 0 & i\in H, j\in C, s,t\in T
\end{align}
}
Expression (\ref{TransshipmentMIP_Eq:ObjMinMatches}) minimizes the number of matches. 
Equations (\ref{TransshipmentMIP_Eq:HotStreamConservation})-(\ref{TransshipmentMIP_Eq:ColdStreamConservation}) enforce heat conservation.
Equation (\ref{TransshipmentMIP_Eq:BigM_Constraint}) allows positive heat exchange between hot stream $i\in H$ and cold stream $j\in C$ only if $(i,j)$ are matched.


\section{Heuristics with Performance Guarantees}
\label{sec:heuristics_performance_guarantees}

\subsection{Computational Complexity}
\label{sec:computational_complexity}

We briefly introduce $\mathcal{NP}$-completeness and basic computational complexity classes \citep{arora:2009,papadimitriou:1994}.
A \emph{polynomial algorithm} produces a solution for a computational problem with a running time polynomial to the size of the problem instance.
There exist problems which admit a polynomial-time algorithm and others which do not.
There is also the class of \emph{$\mathcal{NP}$-complete problems} for which we do not know whether they admit a polynomial algorithm or not.
The question of whether $\mathcal{NP}$-complete problems admit a polynomial algorithm is known as the $\mathcal{P}=\mathcal{NP}$ question.
In general, it is conjectured that $\mathcal{P}\neq \mathcal{NP}$, i.e.\ $\mathcal{NP}$-complete problems are not solvable in polynomial time.
An optimization problem is \emph{$\mathcal{NP}$-hard} if its decision version is $\mathcal{NP}$-complete.
A computational problem is \emph{strongly $\mathcal{NP}$-hard} if it remains $\mathcal{NP}$-hard when all parameters are bounded by a polynomial to the size of the instance.


The minimum number of matches problem is known to be strongly $\mathcal{NP}$-hard, even in the special case of a single temperature interval.
\citet{furman:2004} propose an $\mathcal{NP}$-hardness reduction from the well-known 3-Partition problem, i.e.\ they show that the minimum number of matches problem has difficulty equivalent to the 3-Partition problem.
\ref{App:NP_harndess} presents an alternative $\mathcal{NP}$-hardness reduction from the bin packing problem. This alternative setting of the minimum number of matches problem gives new insight into the packing nature of the problem.
A major contribution of this paper is to design efficient, greedy heuristics motivated by packing.

\begin{theorem}
\label{thm:NP_hardness}
There exists an $\mathcal{NP}$-hardness reduction from bin packing to the minimum number of matches problem with a single temperature interval.
\end{theorem}
\begin{omitted_proof}
See \ref{App:NP_harndess}.
\end{omitted_proof}


\subsection{Approximation Algorithms}
\label{sec:approximation_algorithms}

A heuristic with a performance guarantee is usually called an \emph{approximation algorithm} \citep{vazirani:2001,williamson:2011}.
Unless $\mathcal{P}=\mathcal{NP}$, there is no polynomial algorithm solving an $\mathcal{NP}$-hard problem.
An approximation algorithm is a polynomial algorithm producing a near-optimal solution to an optimization problem. 
Formally, consider an an optimization problem, without loss of generality minimization, and a polynomial Algorithm $A$ for solving it (not necessarily to global optimality).
For each problem instance $I$, let $C_A(I)$ and $C_{OPT}(I)$ be the algorithm's objective value and the optimal objective value, respectively.
Algorithm $A$ is $\rho$-approximate if, for every problem instance $I$, it holds that:
\begin{equation*}
C_A(I)\leq \rho\cdot C_{OPT}(I).
\end{equation*}
That is, a $\rho$-approximation algorithm computes, in polynomial time, a solution with an objective value at most $\rho$ times the optimal objective value.
The value $\rho$ is the \emph{approximation ratio} of Algorithm $A$. 
To prove a $\rho$-approximation ratio, we proceed as depicted in Figure \ref{Fig:ApproximationAlgorithm}.
For each problem instance, we compute analytically a lower bound $C_{LB}(I)$ of the optimal objective value, i.e.\ $C_{LB}(I) \leq C_{OPT}(I)$, and we show that the algorithm's objective value is at most $\rho$ times the lower bound, i.e.\ $C_A(I)\leq \rho \cdot C_{LB}(I)$.
The ratio of a $\rho$-approximation algorithm is \emph{tight} if the algorithm is not $\rho-\epsilon$ approximate for any $\epsilon>0$. 
An algorithm is $O(f(n))$-approximate and $\Omega(f(n))$-approximate, where $f(n)$ is a function of an input parameter $n$, if the algorithm does not have an approximation ratio asymptotically higher and lower, respectively, than $f(n)$.


\begin{figure}
\begin{center}
\includegraphics{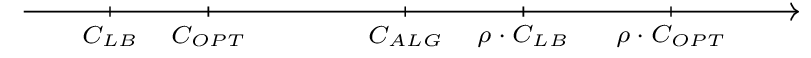}
\caption{Analysis of an Approximation Algorithm}
\label{Fig:ApproximationAlgorithm}
\end{center}
\end{figure}

Approximation algorithms have been developed for two problem classes relevant to process systems engineering: heat exchanger networks \citep{furman:2004} and pooling \citep{dey2015analysis}. Table \ref{Table:Heuristics_Performance_Guarantees} lists performance guarantees for the minimum number of matches problem; most are new to this manuscript.

\begin{table}[t]
\small
\begin{adjustbox}{center} 
\begin{tabular}{ | l l c c c | }
\hline
\textbf{Heuristic} & \textbf{Abbrev.} & \textbf{Section} & \textbf{Performance Guarantee} & \textbf{Running Time} \\ 
\hline
\multicolumn{5}{|l|} {\textbf{Single Temperature Interval Problem}} \\ 
Simple Greedy & SG & \ref{Sec:SingleTemperatureIntervalProblem-Approximation} & 2$^{\dagger}$ (tight) & $O(nm)$ \\
Improved Greedy & IG & \ref{Sec:SingleTemperatureIntervalProblem-Approximation} &  1.5 (tight) & $O(nm)$ \\
\hline
\multicolumn{5}{|l|} {\textbf{Relaxation Rounding Heuristics}} \\ 
Fractional LP Rounding & FLPR & \ref{Subsection:FLPR} & $O(k)^{\dagger}$, $O(U_{\max})$, $\Omega(n)$ & 1 LP \\
Lagrangian Relaxation Rounding & LRR & \ref{Subsection:LRR} & & 2 LPs \\
Covering Relaxation Rounding & CRR & \ref{Subsection:CRR} & & $O(nm)$ ILPs \\
\hline
\multicolumn{5}{|l|} {\textbf{Water Filling Heuristics}} \\ 
Water Filling MILP & WFM & \ref{sec:water_filling} \& \ref{Sec:SingleTemperatureIntervalProblem-MILP} & $O(k)^{\dagger}$, $\Omega(k)$ & $O(k)$ MILPs \\ 
Water Filling Greedy & WFG & \ref{sec:water_filling} \& \ref{Sec:SingleTemperatureIntervalProblem-Approximation} & $O(k)^{\dagger}$, $\Omega(k)$ & $O(n m k)$ \\
\hline
\multicolumn{5}{|l|} {\textbf{Greedy Packing Heuristics}} \\ 
Largest Heat Match LP-based & LHM-LP & \ref{Subsection:Largest_Heat_Match} & $O(\log n + \log (h_{\max}/\epsilon))$ & $O(n^2m^2)$ LPs \\
Largest Heat Match Greedy & LHM & \ref{Subsection:Largest_Heat_Match} & & $O(n^2 m^2 k)$ \\
Largest Fraction Match & LFM & \ref{Subsection:LFM} & & $O(n^2 m^2 k)$  \\
Shortest Stream & SS & \ref{Subsection:SS} & & $O(n m k)$  \\
\hline
\end{tabular}
\end{adjustbox} 
\caption{Performance guarantees for the minimum number of matches problem. The performance guarantees marked ${\dagger}$ are from \cite{furman:2004}; all others are new to this manuscript.}
\label{Table:Heuristics_Performance_Guarantees}
\end{table}



\section{Single Temperature Interval Problem}
\label{sec:single_temperature_interval}

This section proposes efficient algorithms for the single temperature interval problem.
Using graph theoretic properties, we obtain: (i) a novel, efficiently solvable MILP formulation without big-M constraints and (ii) an improved 3/2-approximation algorithm. 
{Of course, the single temperature interval problem is not immediately applicable to the minimum number of matches problem with multiple temperature intervals. But designing efficient approximation algorithms for the single temperature interval is the first, essential step before considering multiple temperature intervals. Additionally, the water filling heuristics introduced in Section \ref{sec:water_filling} repeatedly solve the single temperature interval problem.}

In the single temperature interval problem, a feasible solution can be represented as a bipartite graph $G=(H\cup C, M)$ in which there is a node for each hot stream $i\in H$, a node for each cold stream $j\in C$ and the set $M\subseteq H\times C$ specifies the matches. 
\ref{App:Single_Temperature_Interval} shows the existence of an optimal solution whose graph $G$ does not contain any cycle. 
A connected graph without cycles is a \emph{tree}, so $G$ is a forest consisting of trees.
\ref{App:Single_Temperature_Interval} also shows that the number $v$ of edges in $G$, i.e.\ the number of matches, is related to the number $\ell$ of trees with the equality $v=n+m-\ell$.
Since $n$ and $m$ are input parameters, minimizing the number of matches in a single temperature interval is equivalent to finding a solution whose graph consists of a maximal number $\ell$ of trees.

\subsection{Novel MILP Formulation}
\label{Sec:SingleTemperatureIntervalProblem-MILP}

We propose a novel MILP formulation for the single temperature interval problem.
In an optimal solution without cycles, there can be at most $\min\{n,m\}$ trees.
From a packing perspective, we assume that there are $\min\{n,m\}$ available bins and each stream is placed into exactly one bin.
If a bin is non-empty, then its content corresponds to a tree of the graph.
The objective is to find a feasible solution with a maximum number of bins. 

To formulate the problem as an MILP, we define the set $B=\{1,2,\ldots,\min\{n,m\}\}$ of available bins.
Binary variable $x_b$ is 0 if bin $b\in B$ is empty and 1, otherwise. 
A binary variable $w_{i,b}$ indicates whether hot stream $i\in H$ is placed into bin $b\in B$.
Similarly, a binary variable $z_{j,b}$ specifies whether cold stream $j\in C$ is placed into bin $b\in B$. 
Then, the minimum number of matches problem can be formulated:
{\allowdisplaybreaks
\begin{align}
\text{max} & \sum_{b\in B} x_b \label{Eq:SingleMIP_MaxBins}  \\
& x_b \leq \sum_{i\in H} w_{i,b} & b\in B \label{Eq:SingleMIP_HotBinUsage} \\ 
& x_b \leq \sum_{j\in C} z_{j,b} & b\in B \label{Eq:SingleMIP_ColdBinUsage} \\ 
& \sum_{b\in B} w_{i,b} = 1 & i\in H \label{Eq:SingleMIP_HotAssignment} \\
& \sum_{b\in B} z_{j,b} = 1 & j\in C \label{Eq:SingleMIP_ColdAssignment} \\
& \sum_{i\in H} w_{i,b} \cdot h_i = \sum_{j\in C} z_{j,b} \cdot c_j & b\in B \label{Eq:SingleMIP_BinHeatConservation} \\
& x_b, w_{i,b}, z_{j,b}\in\{0,1\} & b\in B, i\in H, j\in C \label{Eq:SingleMIP_Integrality}
\end{align}
}
Expression (\ref{Eq:SingleMIP_MaxBins}), the objective function, maximizes the number of bins. 
Equations (\ref{Eq:SingleMIP_HotBinUsage}) and (\ref{Eq:SingleMIP_ColdBinUsage}) ensure that a bin is used if there is at least one stream in it.
Equations (\ref{Eq:SingleMIP_HotAssignment}) and (\ref{Eq:SingleMIP_ColdAssignment}) enforce that each stream is assigned to exactly one bin.
Finally, Eqs.\ (\ref{Eq:SingleMIP_BinHeatConservation}) ensure the heat conservation of each bin. 
Note that, unlike the transportation and transshipment models, Eqs.\ (\ref{Eq:SingleMIP_MaxBins})-(\ref{Eq:SingleMIP_BinHeatConservation}) do not use a big-M parameter. {
\ref{App:Water_Filling} formulates the single temperature interval problem \emph{without} heat conservation. Eqs.\ (\ref{Eq:SingleMIPwithoutConservation_MaxBins})-(\ref{Eq:SingleMIPwithoutConservation_Integrality}) are similar to Eqs.\ (\ref{Eq:SingleMIP_MaxBins})-(\ref{Eq:SingleMIP_Integrality}) except (i) they drop constraints (\ref{Eq:SingleMIP_HotBinUsage}) and (ii) equalities (\ref{Eq:SingleMIP_HotAssignment}) \& (\ref{Eq:SingleMIP_BinHeatConservation}) become inequalities (\ref{Eq:SingleMIPwithoutConservation_HotAssignment}) \& (\ref{Eq:SingleMIPwithoutConservation_BinHeatConservation}).
}

\subsection{Improved Approximation Algorithm}
\label{Sec:SingleTemperatureIntervalProblem-Approximation}

\citet{furman:2004} propose a greedy 2-approxi\-mation algorithm for the minimum number of matches problem in a single temperature interval.
We show that their analysis is tight. We also propose an improved, tight $1.5$-approximation algorithm by prioritizing matches with equal heat loads and exploiting graph theoretic properties.

The simple greedy (SG) algorithm considers the hot and the cold streams in non-increasing heat load order \citep{furman:2004}.
Initially, the first hot stream is matched to the first cold stream and an amount $\min\{h_1, c_1\}$ of heat is transferred between them.
Without loss of generality $h_1 > c_1$, which implies that an amount $h_1 - c_1$ of heat load remains to be transferred from $h_1$ to the remaining cold streams.
Subsequently, the algorithm matches $h_1$ to $c_2$, by transferring $\min\{h_1 - c_1, c_2\}$ heat. 
The same procedure repeats with the other streams until all remaining heat load is transferred.



\begin{algorithm}[t] \nonumber
\caption[Simple Greedy (SG)]{Simple Greedy (SG), developed by \cite{furman:2004}, is applicable to one temperature interval only.}
\begin{algorithmic}[1]
\State Sort the streams so that $h_1\geq h_2\geq \ldots\geq h_n$ and $c_1\geq c_2\geq \ldots \geq c_m$.
\State Set $i = 1$ and $j = 1$.
\While {there is remaining heat load to be transferred}
\State Transfer $q_{i,j}=\min\{h_i, c_j\}$ 
\State Set $h_i = h_i - q_{i,j}$ and $c_j = c_j - q_{i,j}$
\State \textbf{if} $h_i = 0$, \textbf{then} set $i = i+1$
\State \textbf{if} $c_j = 0$, \textbf{then} set $j = j+1$
\EndWhile
\end{algorithmic}
\label{Alg:SimpleGreedy}
\end{algorithm}



\citet{furman:2004} show that Algorithm SG is 2-approximate for one temperature interval.
Our new result in Theorem \ref{thm:greedy} shows that this ratio is tight.

\begin{theorem}
\label{thm:greedy}
Algorithm SG achieves an approximation ratio of 2 for the single temperature interval problem and it is tight.
\end{theorem}
\begin{omitted_proof}
See \ref{App:Single_Temperature_Interval}.
\end{omitted_proof}

\medskip

Algorithm IG improves Algorithm SG by: (i) matching the pairs of hot and cold streams with equal heat loads and (ii) using the acyclic property in the graph representation of an optimal solution. {
In practice, hot and cold process streams are unlikely to have equal supplies and demands of heat, so discussing equal heat loads is largely a thought experiment. But the updated analysis allows us to claim an improved performance bound on Algorithm SG. Additionally, the notion of matching roughly equivalent supplies and demands inspires the Section \ref{Subsection:LFM} \emph{Largest Fraction Match First} heuristic.
}


\begin{algorithm}[t] \nonumber
\caption[Improved Greedy (IG)]{Improved Greedy (IG) is applicable to one temperature interval only.}
\label{alg:impgreedy}
\begin{algorithmic}[1]
\For {each pair of hot stream $i$ and cold stream $j$ s.t. $h_i=c_j$}
\State Transfer $h_i$ amount of heat load (also equal to $c_j$) between them and remove them.
\EndFor
\State Run Algorithm SG with respect to the remaining streams. 
\end{algorithmic}
\label{Alg:ImprovedGreedy}
\end{algorithm}


\medskip

\begin{theorem}\label{thm:impgreedy}
Algorithm IG achieves an approximation ratio of 1.5 for the single temperature interval problem and it is tight.
\end{theorem}
\begin{omitted_proof}
See \ref{App:Single_Temperature_Interval}.
\end{omitted_proof}

\medskip



\section{Maximum Heat Computations with Match Restrictions}
\label{sec:max_heat}
This section discusses computing the maximum heat that can be feasibly exchanged in a minimum number of matches instance. Section \ref{Sec:MaxHeat_2Streams} discusses the specific instance of two streams and thereby reduces the value of big-M parameter $U_{i,j}$. 
Sections \ref{Sec:MaxHeat_SingleInterval} \& \ref{Sec:MaxHeat_MultipleIntervals} generalize Section \ref{Sec:MaxHeat_2Streams} from 2 streams to any number of the candidate matches. Section \ref{Sec:MaxHeat_SingleInterval} is limited to a restricted subset of matches in a single temperature interval. Section \ref{Sec:MaxHeat_MultipleIntervals} calculates the maximum heat that can be feasibly exchanged for the most general case of multiple temperature intervals.
These maximum heat computations are an essential ingredient of our heuristic methods and aim in using a match in the most profitable way. They also answer the feasibility of the minimum number of matches problem.

\subsection{Two Streams and Big-M Parameter Computation}
\label{Sec:MaxHeat_2Streams}

A common way of computing the big-M parameters is setting $U_{i,j}=\min\{h_i,c_j\}$ for each $i\in H$ and $j\in C$. \citet{gundersen:1997} propose a better method for calculating the big-M parameter.
Our novel Greedy Algorithm MHG (Maximum Heat Greedy) obtains tighter $U_{i,j}$ bounds than either the trivial bounds or the \citet{gundersen:1997} bounds by exploiting the transshipment model structure.

Given hot stream $i$ and cold stream $j$, Algorithm MHG computes the maximum amount of heat that can be feasibly exchanged between $i$ and $j$ in any feasible solution.
Algorithm MHG is tight in the sense that there is always a feasible solution where streams $i$ and $j$ exchange exactly $U_{i,j}$ units of heat.
Note that, in addition to $U_{i,j}$, the algorithm computes a value $q_{i,s,j,t}$ of the heat exchanged between each hot stream $i\in H$ in temperature interval $s\in T$ and each cold stream $j\in C$ in temperature interval $t\in T$, so that $\sum_{s,t\in T}q_{i,s,j,t}=U_{i,j}$.
These $q_{i,s,j,t}$ values are required by greedy packing heuristics in Section \ref{sec:greedy_packing}.

Algorithm \ref{Alg:MaximumHeat} is a pseudocode of Algorithm MHG.
The correctness, i.e.\ the maximality of the heat exchanged between $i$ and $j$, is a corollary of the well known maximum flow - minimum cut theorem.
Initially, the procedure transfers the maximum amount of heat across the same temperature interval; $q_{i,u,s,u}=\min\{\sigma_{i,u},\delta_{j,u}\}$ for each $u\in T$.
The remaining heat is transferred greedily in a top down manner, with respect to the temperature intervals, by accounting heat residual capacities.
For each temperature interval $u\in T$, the heat residual capacity $R_u=\sum_{i=1}^n\sum_{s=1}^u\sigma_{i,s} - \sum_{j=1}^m\sum_{t=1}^u\delta_{j,t}$ imposes an upper bound on the amount of heat that may descend from temperature intervals $1,2,\ldots,u$ to temperature intervals $u+1,u+2,\ldots,k$.

\begin{algorithm}[t] \nonumber
\caption{Maximum Heat Greedy (MHG)}
\textbf{Input:} Hot stream $i\in H$ and cold stream $j\in C$ 
\begin{algorithmic}[1]
\State $\vec{q} \leftarrow \vec{0}$
\For {{$u=1,2,\ldots,k-1$}}
\State {$R_u=\sum_{i=1}^n\sum_{s=1}^u\sigma_{i,s} - \sum_{j=1}^m\sum_{t=1}^u\delta_{j,t}$}
\EndFor
\For {$u=1,2,\ldots,k$}
\State $q_{i,u,j,u}\leftarrow \min\{\sigma_{i,u},\delta_{j,u}\}$
\State $\sigma_{i,u}\leftarrow \sigma_{i,u}-q_{i,u,j,u}$
\State $\delta_{j,u}\leftarrow \delta_{j,u}-q_{i,u,j,u}$
\EndFor
\For {$s=1,2,\ldots,k-1$}
\For {$t=s+1,s+2,\ldots,k$}
\State $q_{i,s,j,t}=\min\{\sigma_{i,s},\delta_{j,t},\min_{s\leq u\leq t-1}\{R_u\}\}$
\State $\sigma_{i,s}\leftarrow \sigma_{i,s}-q_{i,s,j,t}$
\State $\delta_{j,t}\leftarrow \delta_{j,t}-q_{i,s,j,t}$
\For {$u=s,s+1,s+2,\ldots,t-1$}
\State $R_u\leftarrow R_u-q_{i,s,j,t}$
\EndFor
\EndFor
\EndFor
\State Return $\vec{q}$
\end{algorithmic}
\label{Alg:MaximumHeat}
\end{algorithm}

\subsection{Single Temperature Interval}
\label{Sec:MaxHeat_SingleInterval}

Given an instance of the single temperature interval problem and a subset $M$ of matches, the maximum amount of heat that can be feasibly exchanged between the streams using only the matches in $M$ can be computed by solving \ref{EquationSet:SingleMaxHeatLP_initial_stattement}.
{Like the single temperature interval algorithms of Section \ref{sec:single_temperature_interval}, \ref{EquationSet:SingleMaxHeatLP_initial_stattement} is not directly applicable to a minimum number of matches problem with multiple temperature intervals. But \ref{EquationSet:SingleMaxHeatLP_initial_stattement} is an important part of our water filling heuristics.}
For simplicity, \ref{EquationSet:SingleMaxHeatLP_initial_stattement} drops
temperature interval indices for variables $q_{i,j}$.
{\allowdisplaybreaks
\begin{align}
\tag{MaxHeatLP}
\label{EquationSet:SingleMaxHeatLP_initial_stattement}
\begin{aligned}
\text{max} & \sum_{(i,j)\in M} q_{i,j} \\
& \sum_{j\in C} q_{i,j} \leq h_i & i\in H \\
& \sum_{i\in H} q_{i,j} \leq c_j & j\in C \\
& q_{i,j} \geq 0 & i\in H, j\in C
\end{aligned}
\end{align}
}

\subsection{Multiple Temperature Intervals}
\label{Sec:MaxHeat_MultipleIntervals}

Maximizing the heat exchanged through a subset of matches across multiple temperature intervals can solved with an LP that generalizes \ref{EquationSet:SingleMaxHeatLP_initial_stattement}. 
The generalized LP must satisfy the additional requirement that, after removing a maximum heat exchange, the remaining instance is feasible.
Feasibility is achieved using residual capacity constraints which are essential for the efficiency of greedy packing heuristics (see Section \ref{Subsection:MonotonicGreedyHeuristics}).

Given a set $M$ of matches, let $A(M)$ be the set of quadruples $(i,s,j,t)$ such that a positive amount of heat can be feasibly transferred via the transportation arc with endpoints the nodes $(i,s)$ and $(j,t)$.
The set $A(M)$ does not contain any quadruple $(i,s,j,t)$ with: (i) $s>t$, (ii) $\sigma_{i,s}=0$, (iii) $\delta_{j,t}=0$, or (iv) $(i,j)\not\in M$. 
Let $V^H(M)$ and $V^C(M)$ be the set of transportation vertices $(i,s)$ and $(j,t)$, respectively, that appear in $A(M)$.
Similarly, given two fixed vertices $(i,s)\in V^H(M)$ and $(j,t)\in V^C(M)$, we define the sets $V_{i,s}^C(M)$ and $V_{j,t}^H(M)$ of their respective neighbors in $A(M)$.

Consider a temperature interval $u\in T$.
We define by $A_u(M)\subseteq A(M)$ the subset of quadruples with $s\leq u< t$, for $u\in T$. 
The total heat transferred via the arcs in $A_u(M)$ must be upper bounded by $R_u=\sum_{i=1}^n\sum_{s=1}^u\sigma_{i,s} - \sum_{j=1}^m\sum_{t=1}^u\delta_{j,t}$.
Furthermore, $A(M)$ eliminates any quadruple $(i,s,j,t)$ with $R_u=0$, for some $s\leq u<t$.  
Finally, we denote by $T(M)$ the subset of temperature intervals affected by the matches in $M$, i.e.\ if $u\in T(M)$, then there exists a quadruple $(i,s,j,t)\in A(M)$, with $s\leq u<t$.
The procedure $MHLP(M)$ is based on solving the following LP:
{\allowdisplaybreaks
\begin{align}
\max & \sum_{(i,s,j,t) \in A(M)} q_{i,s,j,t} \label{Eq:MaxHeatLP_Objective} \\ 
& \sum_{(j,t)\in V_{i,s}^C(M)} q_{i,s,j,t} \leq \sigma_{i,s} & (i,s)\in V^H(M) \label{Eq:MaxHeatLP_HotStreamSupply}\\
& \sum_{(i,s)\in V_{j,t}^H(M)} q_{i,s,j,t} \leq \delta_{j,t} & (j,t)\in V^C(M) \label{Eq:MaxHeatLP_ColdStreamDemand}\\
& \sum_{(i,s,j,t)\in A_u(M)} q_{i,s,j,t}\leq R_u & u\in T(M) \label{Eq:MaxHeatLP_ResidualCapacity}\\
& q_{i,s,j,t}\geq 0 & (i,s,j,t) \in A(M) \label{Eq:MaxHeatLP_Positiveness}
\end{align}
}
Expression (\ref{Eq:MaxHeatLP_Objective}) maximizes the total exchanged heat by using only the matches in $M$.
Constraints (\ref{Eq:MaxHeatLP_HotStreamSupply}) and (\ref{Eq:MaxHeatLP_ColdStreamDemand}) ensure that each stream uses only part of its available heat.
Constraints (\ref{Eq:MaxHeatLP_ResidualCapacity}) enforce the heat residual capacities.


\section{Relaxation Rounding Heuristics}
\label{sec:relaxation_rounding}
This section investigates relaxation rounding heuristics for the minimum number of matches problem. 
These heuristics begin by optimizing an efficiently-solvable relaxation of the original MILP. 
The efficiently-solvable relaxation allows violation of certain constraints, so that the optimal solution(s) is (are) typically infeasible in the original MILP.
The resulting infeasible solutions are subsequently rounded to feasible solutions for the original MILP.
We consider 3 types of relaxations.
Section \ref{Subsection:FLPR} relaxes the integrality constraints and proposes fractional LP rounding.
Section \ref{Subsection:LRR} relaxes the big-M constraints, i.e.\ Eqs.\ (\ref{TransportationMIP_Eq:BigM_Constraint}), 
and uses Lagrangian relaxation rounding.
Section \ref{Subsection:CRR} relaxes the heat conservation equations, i.e.\ Eqs.\ (\ref{TransportationMIP_Eq:HotStreamConservation})-(\ref{TransportationMIP_Eq:ColdStreamConservation}), 
and takes an approach based on covering relaxations.
{Figure \ref{Figure:RR_Ingredients} shows the main components of relaxation rounding heuristics.}


\begin{figure}
\begin{center}
\includegraphics{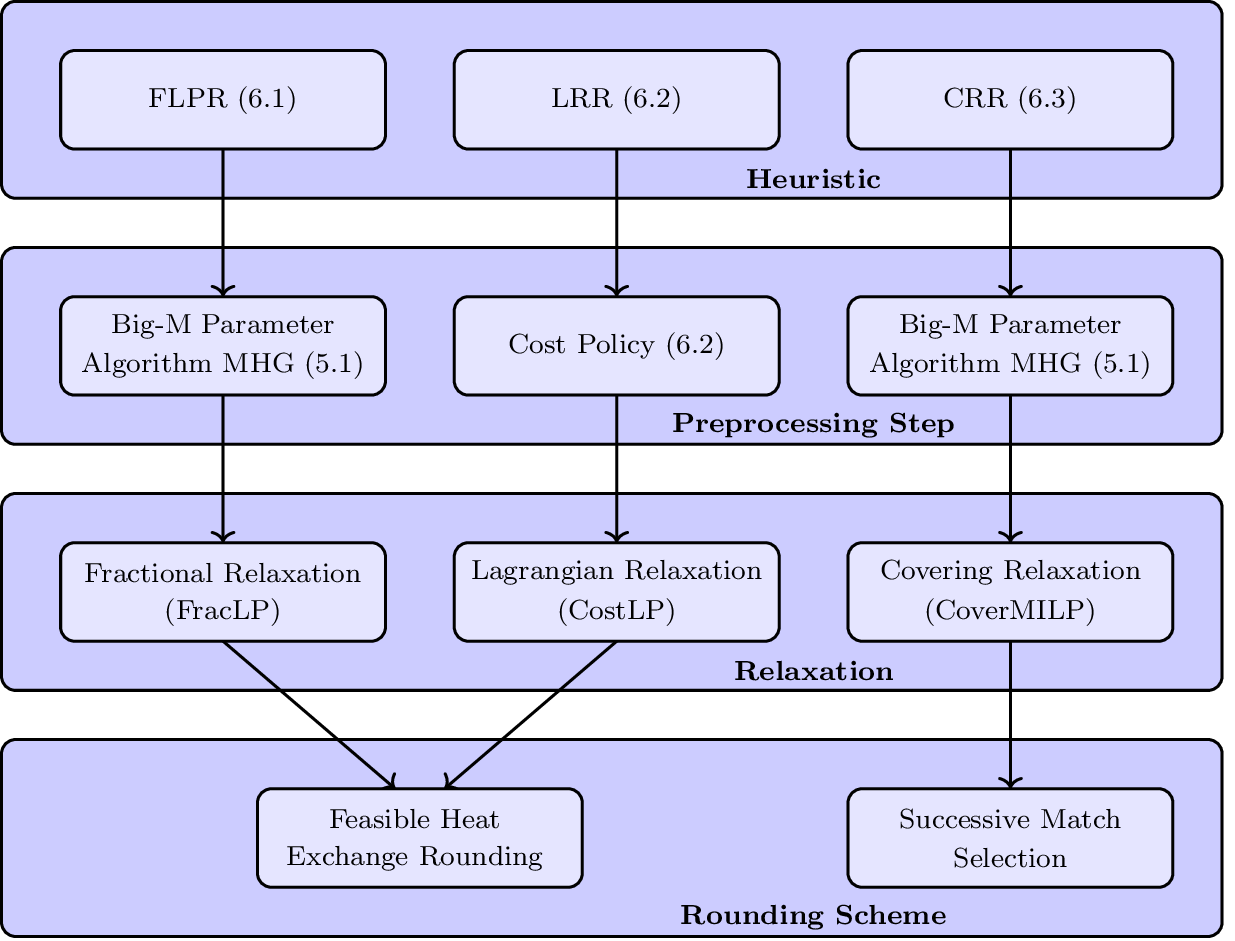}
\caption{
The main components of relaxation rounding heuristics are (i) a preprocessing step, (ii) a relaxation, and (iii) a rounding scheme.
The preprocessing step constructs the relaxation.
Fractional relaxation and covering relaxation require big-M parameter computations, while Lagrangian relaxation minimum cost LP requires cost calculations.
FLPR and LRR compute a feasible heat exchange between all streams, i.e.\ values to variables $q_{i,s,j,t}$, by solving their respective relaxations and round the relaxed solutions according to Algorithm \ref{Alg:FLPR}.
Heuristic CRR adds matches incrementally until it ends up with a feasible solution.
Feasibility is determined using the maximum heat LP in Section \ref{Sec:MaxHeat_MultipleIntervals}.}
\label{Figure:RR_Ingredients}
\end{center}
\end{figure}

\subsection{Fractional LP Rounding}
\label{Subsection:FLPR}


The LP rounding heuristic, originally proposed by \citet{furman:2004},
transforms an optimal fractional solution for the transportation MILP to a feasible integral solution.
We show that the fractional LP can be solved efficiently via network flow techniques. 
We observe that, in the worst case, the heuristic produces a weak solution if it starts with an arbitrary optimal solution of the fractional LP. 
We derive a novel performance guarantee showing that the heuristic is efficient when 
the heat of each chosen match $(i,j)$ is close to big-M parameter $U_{i,j}$, in the optimal fractional solution.

Consider the fractional LP obtained by replacing the integrality constraints $y_{i,j}\in\{0,1\}$ of the transportation MILP, i.e.\ Eqs.\ (\ref{TransportationMIP_Eq:ObjMinMatches})-(\ref{TransportationMIP_Eq:IntegralityConstraints}), with the constraints $0\leq y_{i,j}\leq 1$, for each $i\in H$ and $j\in C$:
%
\begin{align}
\tag{FracLP}
\label{EquationSet:FracLP}
\begin{aligned}
\min & \sum_{i \in H}\sum_{j \in C} y_{i,j} & \\ 
& \sum_{j\in C}\sum_{t\in T} q_{i,s,j,t} = \sigma_{i,s} & i\in H, s\in T \\ 
& \sum_{i\in H}\sum_{s\in T} q_{i,s,j,t} = \delta_{j,t} & j\in C, t\in T  \\ 
& \sum_{s,t\in T} q_{i,s,j,t}\leq U_{i,j}\cdot y_{i,j} & i\in H, j\in C  \\ 
& q_{i,s,j,t} = 0 & \;\; i\in H, j\in C, s,t\in T: s\leq t \\ 
& 0\leq y_{i,j}\leq 1,\; q_{i,s,j,t}\geq 0 & i\in H, j\in C,\; s,t\in T
\end{aligned}
\end{align}


\ref{EquationSet:FracLP} can be solved via minimum cost flow methods.
Figure \ref{Fig:min_cost_flow_network} illustrates a network $N$, i.e.\ a minimum cost flow problem instance, such that finding a minimum cost flow in $N$ is equivalent to optimizing the fractional LP.
Network $N$ is a layered graph with six layers of nodes: (i) a source node $S$, (ii) a node for each hot stream $i\in H$, (iii) a node for each pair $(i,s)$ of hot stream $i\in H$ and temperature interval $s\in T$, (iv) a node for each pair $(j,t)$ for each cold stream $j\in C$ and temperature interval $t\in T$, (v) a node for each cold stream $j\in C$, and (vi) a destination node $D$.
We add: (i) the arc $(S,i)$ with capacity $h_i$ for each $i\in H$, (ii) the arc $(i,(i,s))$ with capacity $\sigma_{i,s}$ for each $i\in H$ and $s\in T$, (iii) the arc $((i,s),(j,t))$ with infinite capacity for each $i\in H$, $j\in C$ and $s,t\in T$, (iv) the arc $((j,t),j)$ with capacity $\delta_{j,t}$ for each $j\in H$ and $t\in T$, and (v) the arc $(j,D)$ with capacity $c_j$ for each $j\in C$.  
Each arc $((i,s),(j,t))$ has cost $1/U_{i,j}$ for $i\in H$, $j\in C$ and $s,t\in T$. 
Every other arc has zero cost.
Any flow of cost $\sum_i h_i$ on network $N$ is equivalent to a feasible solution for \ref{EquationSet:FracLP} with the same cost and vice versa.

\begin{figure}[t]

\begin{center}
\includegraphics{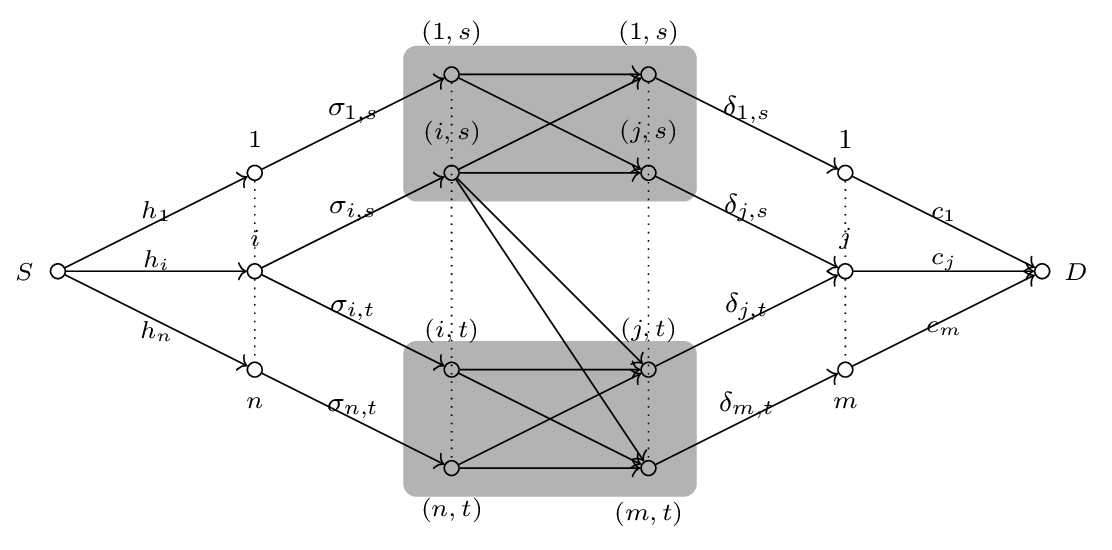}
\caption{Minimum cost network flow formulation of \ref{EquationSet:FracLP}.
The heat is modeled as flow transferred from a source node $S$ to a destination node $D$.
All finite capacities are labelled above the corresponding arcs.
The cost is incurred in each arc between node $(i,s)\in H\times T$ and node $(j,t)\in C\times T$ under the condition that heat flows to the same or a lower temperature interval.
}
\label{Fig:min_cost_flow_network}
\end{center}

\end{figure}

\cite{furman:2004} observe that any feasible solution of \ref{EquationSet:FracLP} can be rounded to a feasible solution of the original problem via Algorithm \ref{Alg:FLPR}, a simple greedy procedure that we call FLPR. 
Given a problem instance $I$, the procedure $FractionalLP(I)$ computes an optimal solution of \ref{EquationSet:FracLP}.
We denote by $(\vec{y}^f,\vec{q}^f)$ the optimal fractional solution.

\begin{algorithm}[t] \nonumber
\caption{Fractional LP Rounding (FLPR) \citep{furman:2004}}
\begin{algorithmic}[1]
\State $(\vec{y}^f,\vec{q}^f)\leftarrow FractionalLP(I)$
\Comment{{(\ref{EquationSet:FracLP}) solving, Section \ref{Subsection:FLPR}}}
\State $\vec{q} \leftarrow \vec{q}^f$
\For {each $i\in H$ and $j\in C$}
\If {$\sum_{s,t\in T} q_{i,s,j,t}>0$}
\State $y_{i,j}\leftarrow 1$
\Else
\State $y_{i,j}\leftarrow 0$
\EndIf
\EndFor
\State Return $(\vec{y},\vec{q})$
\end{algorithmic}
\label{Alg:FLPR}
\end{algorithm}

An inherent drawback of the \citet{furman:2004} approach is the existence of optimal fractional solutions with unnecessary matches.
Theorem \ref{Thm:FLPR_negative} shows that Algorithm FLPR performance is bad in the worst case, even for instances with a single temperature interval. 
The proof, given in \ref{App:Relaxation_Rounding}, can be extended so that unnecessary matches occur across multiple temperature intervals.
\begin{theorem}
\label{Thm:FLPR_negative}
Algorithm FLPR is $\Omega(n)$-approximate.
\end{theorem}
\begin{omitted_proof}
See \ref{App:Relaxation_Rounding}.
\end{omitted_proof}

\medskip

Consider an optimal fractional solution to \ref{EquationSet:FracLP} and suppose that $M\subseteq H\times C$ is the set of pairs of streams exchanging a positive amount of heat.
For each $(i,j)\in M$, denote by $L_{i,j}$ the heat exchanged between hot stream $i$ and cold stream $j$.
We define: 
\begin{equation*}
\phi(M) = \min_{(i,j)\in M} \left\{ \frac{L_{i,j}}{U_{i,j}} \right\}
\end{equation*}
as the \emph{filling ratio}, which corresponds to the minimum portion of an upper bound $U_{i,j}$ filled with the heat $L_{i,j}$, for some match $(i,j)$.
Given an optimal fractional solution with filling ratio $\phi(M)$, Theorem \ref{Thm:FLPR_positive} obtains a $1/\phi(M)$-approximation ratio for FLPR.

\begin{theorem}
\label{Thm:FLPR_positive}
Given an optimal fractional solution with a set $M$ of matches and filling ratio $\phi(M)$, FLPR produces a $\left(1/\phi(M) \right)$-approximate integral solution.
\end{theorem}
\begin{omitted_proof}
See \ref{App:Relaxation_Rounding}.
\end{omitted_proof}

\medskip

In the case where all heat supplies and demands are integers, the integrality of the minimum cost flow polytope and Theorem \ref{Thm:FLPR_positive} imply that FLPR is $U_{\max}$-approximate, where $U_{\max}=\max_{(i,j)\in H\times C}\{U_{i,j}\}$ is the biggest big-M parameter. 
{A corollary of the $L_{i,j} / U_{i,j}$ ratio is that a fractional solution transferring heat $L_{i,j}$ close to capacity $U_{i,j}$ corresponds to a good integral solution. For example, if the optimal fractional solution satisfies $L_{i,j}>0.5 \cdot U_{i,j}$, for every used match $(i,j)$ such that $L_{i,j} \neq 0$, then FLPR gives a 2-approximate integral solution. 
Finally, branch-and-cut repeatedly solves the fractional problem, so our new bound proves the big-M parameter's relevance for exact methods.}
Because performance guarantee of FLPR scales with the big-M parameters $U_{i,j}$, we improve the heuristic performance by computing a small big-M parameter $U_{i,j}$ using Algorithm MHG in Section \ref{Sec:MaxHeat_2Streams}.

\subsection{Lagrangian Relaxation Rounding}
\label{Subsection:LRR}

\cite{furman:2004} design efficient heuristics for the minimum number of matches problem by applying the method of Lagrangian relaxation and relaxing the big-M constraints.
This approach generalizes Algorithm FLPR by approximating the fractional cost of every possible match $(i,j)\in H\times C$ and solving an appropriate LP using these costs.
We present the LP and revisit different ways of approximating the fractional match costs.

In a feasible solution, the fractional cost $\lambda_{i,j}$ of a match $(i,j)$ is the cost incurred per unit of heat transferred via $(i,j)$.
In particular, 
\begin{equation*}
\lambda_{i,j}=
\left\{
	\begin{array}{ll}
		1/L_{i,j},  & \mbox{if $L_{i,j}>0$, and} \\
		0, & \mbox{if $L_{i,j}=0$}
	\end{array}
\right.
\end{equation*}
where $L_{i,j}$ is the heat exchanged via $(i,j)$.
Then, the number of matches can be expressed as $\sum_{i,s,j,t}\lambda_{i,j}\cdot q_{i,s,j,t}$.
\citet{furman:2004} propose a collection of heuristics computing a single cost value for each match $(i,j)$ and constructing a minimum cost solution.
This solution is rounded to a feasible integral solution equivalently to FLPR.

Given a cost vector $\vec{\lambda}$ of the matches, a minimum cost solution is obtained by solving:
\begin{align}
\tag{CostLP}
\label{EquationSet:CostLP}
\begin{aligned}
\text{min} & \sum_{i \in H} \sum_{j \in C} \sum_{s,t\in T} \lambda_{i,j} \cdot q_{i,s,j,t} \\ 
& \sum_{j\in C}\sum_{t\in T} q_{i,s,j,t} = \sigma_{i,s} & i\in H, s\in T \\
& \sum_{i\in H}\sum_{s\in T} q_{i,s,j,t} = \delta_{j,t} & j\in C, t\in T \\
& q_{i,s,j,t}\geq 0 & i\in H, j\in C,\; s,t\in T 
\end{aligned}
\end{align}

A challenge in Lagrangian relaxation rounding is computing a cost $\lambda_{i,j}$ for each hot stream $i\in H$ and cold stream $j\in C$.
We revisit and generalize policies for selecting costs.

\paragraph{Cost Policy 1 (Maximum Heat)} 
Matches that exchange large amounts of heat incur low fractional cost.
This observation motivates selecting $\lambda_{i,j}= 1/U_{i,j}$, for each $(i,j)\in H\times C$, where $U_{i,j}$ is an upper bound on the heat that can be feasibly exchanged between $i$ and $j$.
In this case, Lagrangian relaxation rounding is equivalent to FLPR (Algorithm \ref{Alg:FLPR}).

\paragraph{Cost Policy 2 (Bounds on the Number of Matches)}
This cost selection policy uses lower bounds $\alpha_i$ and $\beta_j$ on the number of matches of hot stream $i\in H$ and cold stream $j\in C$, respectively, in an optimal solution.
Given such lower bounds, at least $\alpha_i$ cost is incurred for the $h_i$ heat units of $i$ and at least $\beta_j$ cost is incurred for the $c_j$ units of $j$.
On average, each heat unit of $i$ is exchanged with cost at least ${\alpha_i}/{h_i}$ and each heat unit of $j$ is exchanged with cost at least ${\beta_j}/{c_j}$.
So, the fractional cost of each match $(i,j)\in H\times C$ can be approximated by setting $\lambda_{i,j}={\alpha_i}/{h_i}$, $\lambda_{i,j}={\beta_j}/{c_j}$ or $\lambda_{i,j}=\frac{1}{2}(\frac{\alpha_i}{h_i} + \frac{\beta_j}{c_j})$.

\citet{furman:2004} use lower bounds $\alpha_i=1$ and $\beta_j=1$, for each $i\in H$ and $j\in C$.
We show that, for any choice of lower bounds $\alpha_i$ and $\beta_j$, this cost policy for selecting $\lambda_{i,j}$ is not effective. Even when $\alpha_i$ and $\beta_j$ are tighter than 1, all feasible solutions of \ref{EquationSet:CostLP} attain the same cost. 
Consider any feasible solution $(\vec{y},\vec{q})$ and the fractional cost $\lambda_{i,j}=\alpha_i / h_i$ for each $(i,j)\in H\times C$.
Then the cost of $(\vec{y},\vec{q})$ in \ref{EquationSet:CostLP} is:
\begin{equation*}
\sum_{i\in H} \sum_{j\in C} \sum_{s,t\in T} \lambda_{i,j} \cdot q_{i,s,j,t} = 
\sum_{i\in H} \sum_{j\in C} \sum_{s,t\in T} \frac{\alpha_i}{h_i} \cdot q_{i,s,j,t} = 
\sum_{i\in H} \alpha_i.
\end{equation*}
Since every feasible solution in (\ref{EquationSet:CostLP}) has cost $\sum_{i\in H} \alpha_i$, Lagrangian relaxation rounding returns an arbitrary solution.
Similarly, if $\lambda_{i,j}={\beta_j}/{c_j}$ for $(i,j)\in H\times C$, every feasible solution has cost $\sum_{j\in C}\beta_j$.
If $\lambda_{i,j}=\frac{1}{2}(\frac{\alpha_i}{h_i} + \frac{\beta_j}{c_j})$, all feasible solutions have the same cost $1/2 \cdot (\sum_{i\in H} \alpha_i + \sum_{j\in C}\beta_j)$.


\paragraph{Cost Policy 3 (Existing Solution)}
This method of computing costs uses an existing solution.
The main idea is to use the actual fractional costs for the solution's matches and a non-zero cost for every unmatched streams pair. 
A minimum cost solution with respect to these costs may improve the initial solution.
Suppose that $M$ is the set of matches in the initial solution and let $L_{i,j}$ be the heat exchanged via $(i,j) \in M$.
Furthermore, let $U_{i,j}$ be an upper bound on the heat exchanged between $i$ and $j$ in any feasible solution.
Then, a possible selection of costs is $\lambda_{i,j}= 1/L_{i,j}$ if $(i,j)\in M$, and $\lambda_{i,j}=1/U_{i,j}$ otherwise.


\subsection{Covering Relaxation Rounding}
\label{Subsection:CRR}

This section proposes a novel covering relaxation rounding heuristic for the minimum number of matches problem.
The efficiency of Algorithm FLPR depends on lower bounding the unitary cost of the heat transferred via each match. 
The goal of the covering relaxation is to use these costs and lower bound the number of matches in a stream-to-stream to basis by relaxing heat conservation.
The heuristic constructs a feasible integral solution by solving successively instances of the covering relaxation.

Consider a feasible MILP solution and suppose that $M$ is the set of matches.
For each hot stream $i\in H$ and cold stream $j\in C$, denote by $C_i(M)$ and $H_j(M)$ the subsets of cold and hot streams matched with $i$ and $j$, respectively, in $M$.
Moreover, let $U_{i,j}$ be an upper bound on the heat that can be feasibly exchanged between $i\in H$ and $j\in C$.
Since the solution is feasible, it must be true that $\sum_{j\in C_i(M)}U_{i,j}\geq h_i$ and $\sum_{i\in H_j(M)}U_{i,j}\geq c_j$.
These inequalities are necessary, though not sufficient, feasibility conditions. 
By minimizing the number of matches while ensuring these conditions, we obtain a covering relaxation:
\begin{align}
\tag{CoverMILP}
\label{EquationSet:CoverMILP}
\begin{aligned}
\text{min} & \sum_{i \in H}\sum_{j \in C} y_{i,j} \\ 
& \sum_{j\in C} y_{i,j}\cdot U_{i,j} \geq h_i & i\in H \\
& \sum_{i\in H} y_{i,j}\cdot U_{i,j} \geq c_j & j\in C \\
& y_{i,j}\in\{0,1\} & i\in H, j\in C
\end{aligned}
\end{align}
In certain cases, the matches of an optimal solution to \ref{EquationSet:CoverMILP} overlap well with the matches in a near-optimal solution for the original problem. 
Our new Covering Relaxation Rounding (CRR) heuristic for the minimum number of matches problem successively solves instances of the covering relaxation \ref{EquationSet:CoverMILP}.
The heuristic chooses new matches iteratively until it terminates with a feasible set $M$ of matches. 
In the first iteration, Algorithm CRR constructs a feasible solution for the covering relaxation and adds the chosen matches in $M$.
Then, Algorithm CRR computes the maximum heat that can be feasibly exchanged using the matches in $M$ and stores the computed heat exchanges in $\vec{q}$.
In the second iteration, the heuristic performs same steps with respect to the smaller updated instance $(\vec{\sigma}',\vec{\delta}')$, where $\sigma_{i,s}'=\sigma_{i,s}-\sum_{j,t}q_{i,s,j,t}$ and $\delta_{j,t}'=\delta_{j,t}-\sum_{i,s}q_{i,s,j,t}$.
The heuristic terminates when all heat is exchanged.

Algorithm \ref{Alg:CRR} is a pseudocode of heuristic CRR.
Procedure $CoveringRelaxation(\vec{\sigma},\vec{\delta})$ produces an optimal subset of matches for the instance of the covering relaxation in which the heat supplies and demands are specified by the vectors $\vec{\sigma}$ and $\vec{\delta}$, respectively.
Procedure $MHLP(\vec{\sigma},\vec{\delta},M)$ (LP-based Maximum Heat) computes the maximum amount of heat that can be feasibly exchanged by using only the matches in $M$ and is based on solving the LP in Section \ref{Sec:MaxHeat_MultipleIntervals}.

\begin{algorithm}[t] \nonumber
\caption{Covering Relaxation Rounding (CRR)}
\begin{algorithmic}[1]
\State $M\leftarrow \emptyset$
\State $\vec{q} \leftarrow \vec{0}$
\State $r\leftarrow \sum_{i\in H}h_i$
\While {$r>0$}
\State For each $i\in H$ and $s\in T$, set $\sigma_{i,s}' \leftarrow \sigma_{i,s} - \sum_{j\in C} \sum_{t\in T} q_{i,s,j,t}$
\State For each $j\in C$ and $t\in T$, set $\delta_{j,t}' \leftarrow \delta_{j,t} - \sum_{i\in H} \sum_{s\in T} q_{i,s,j,t}$
\State $M'\leftarrow CoveringRelaxation(\vec{\sigma}',\vec{\delta}')$
\Comment{{(\ref{EquationSet:CoverMILP}) solving, Section \ref{Subsection:CRR}}}
\State $M\leftarrow M\cup M'$
\State $\vec{q}\leftarrow MHLP(\vec{\sigma},\vec{\delta},M')$
{\Comment{Equations (\ref{Eq:MaxHeatLP_Objective}) - (\ref{Eq:MaxHeatLP_Positiveness}) LP solving, Section \ref{Sec:MaxHeat_MultipleIntervals}}}
\State $r\leftarrow \sum_{i\in H}h_i-\sum_{i\in H}\sum_{j\in C}\sum_{s,t\in T} q_{i,s,j,t}$
\EndWhile
\end{algorithmic}
\label{Alg:CRR}
\end{algorithm}


\section{Water Filling Heuristics}
\label{sec:water_filling}

This section introduces \emph{water filling heuristics} for the minimum number of matches problem.
These heuristics produce a solution iteratively by exchanging the heat in each temperature interval, in a top down manner. 
The water filling heuristics use, in each iteration,
an efficient algorithm for the single temperature interval problem 
(see Section \ref{sec:single_temperature_interval}).


Figure \ref{Fig:Water_Filling} shows the main idea of a \emph{water filling heuristic} for the minimum number of matches problem with multiple temperature intervals.
The problem is solved iteratively in a top-down manner, from the highest to the lowest temperature interval.
Each iteration produces a solution for one temperature interval.
The main components of a water filling heuristic are: (i) a maximum heat procedure which reuses matches from previous iterations and (ii) an efficient single temperature interval algorithm. 

\begin{figure*}[t!]
    \centering

    \begin{subfigure}[t]{0.5\textwidth}
    \centering
	\includegraphics{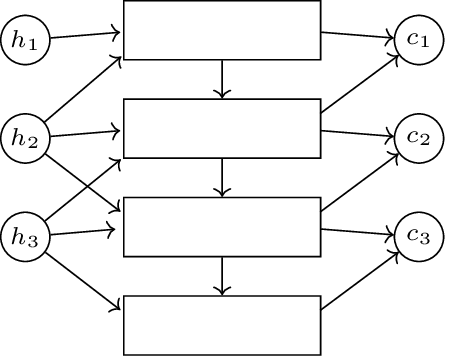}
	\caption{Top Down Temperature Interval Structure}
	\label{Fig:Water_Filling_Temperature_Intervals}
	\end{subfigure}
	\begin{subfigure}[t]{0.45\textwidth}
	\centering
	\includegraphics{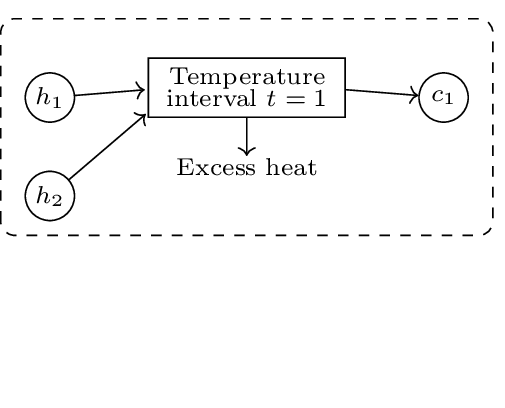}
	\caption{Excess Heat Descending}
	\label{Fig:Water_Filling_Excess_Heat}
	\end{subfigure}
\caption{A water filling heuristic computes a solution by exploiting the top down temperature interval structure and moving from the higher to the lower temperature interval.
In each temperature interval $t$, the heuristic isolates the streams with positive heat at $t$, it matches them and descends the excess heat to the next interval which is sequentially solved. 
\label{Fig:Water_Filling}}

\end{figure*}

Given a set $M$ of matches and an instance $(\vec{\sigma_t},\vec{\delta}_t)$ of the problem in the single temperature interval $t$, the procedure $MHS(\vec{\sigma}_t,\vec{\delta}_t,M)$ (Maximum Heat for Single temperature interval) computes the maximum heat that can be exchanged between the streams in $t$ using only the matches in $M$.
At a given temperature interval $t$, the $MHS$ procedure solves the LP in Section \ref{Sec:MaxHeat_SingleInterval}.
The procedure $SingleTemperatureInterval(\vec{\sigma}_t,\vec{\delta}_t)$ produces an efficient solution for the single temperature interval problem with a minimum number of matches and total heat to satisfy one cold stream. $SingleTemperatureInterval(\vec{\sigma}_t,\vec{\delta}_t)$ either: (i) solves the MILP exactly (Water Filling MILP-based or WFM) or (ii) applies the improved greedy approximation Algorithm IG in Section \ref{sec:single_temperature_interval} (Water Filling Greedy or WFG). 
Both water filling heuristics solve instances of the single temperature interval problem in which there is no heat conservation, i.e.\ the heat supplied by the hot streams is greater or equal than the heat demanded by the cold streams. 
The exact WFM uses the MILP proposed in Eqs.\ (\ref{Eq:SingleMIPwithoutConservation_MaxBins}) -  (\ref{Eq:SingleMIPwithoutConservation_Integrality}) of \ref{App:Water_Filling}.
The greedy heuristic WFG adapts Algorithm IG by terminating when the entire heat demanded by the cold streams has been transferred.
After addressing the single temperature interval, the excess heat descends to the next temperature interval.
Algorithm \ref{Alg:Water_Filling} represents our water filling approach in pseudocode.
{Figure \ref{Figure:Water_Filling_Ingredients} shows the main components of water filling heuristics.}

\begin{figure}
\begin{center}
\includegraphics{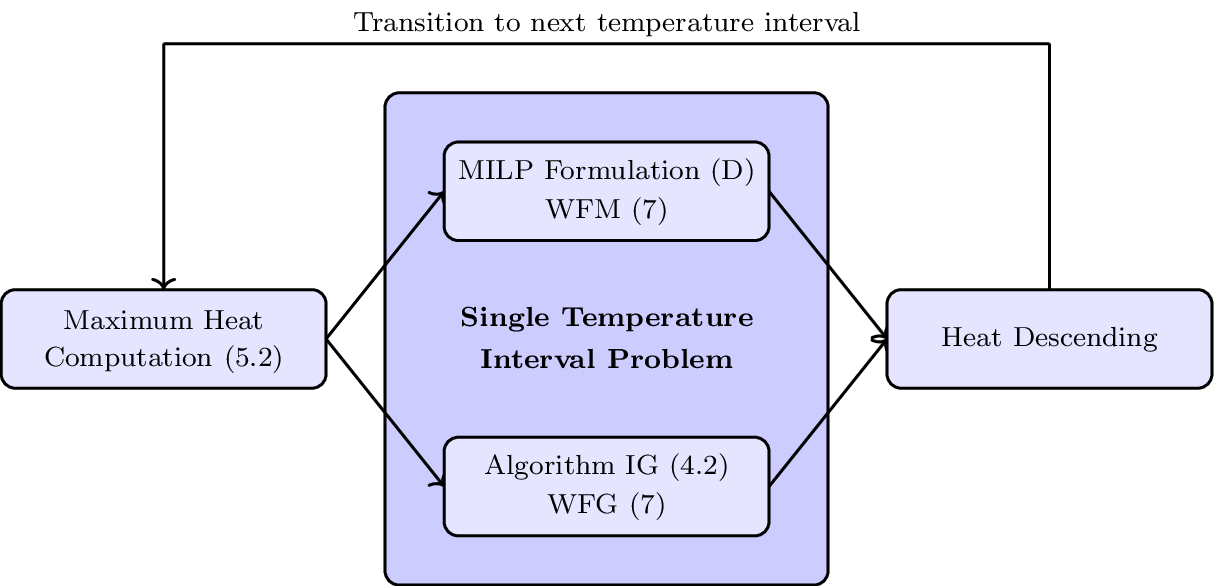}
\caption{
Water filling heuristics solve the temperature intervals serially in a top-down manner and keep composition feasible. 
The main components are (i) a maximum heat computation re-using higher temperature interval matches, (ii) a single temperature interval problem algorithm, and (iii) excess heat descending between consecutive temperature intervals. 
Heuristic WFM uses the \ref{App:Water_Filling} MILP formulation for solving the single temperature interval problem, while heuristic WFG uses the Section \ref{Sec:SingleTemperatureIntervalProblem-Approximation} Algorithm IG.}
\label{Figure:Water_Filling_Ingredients}
\end{center}
\end{figure}

\begin{algorithm}[t] \nonumber
\caption{Water Filling (WF)}\label{alg:water_filling}
\begin{algorithmic}[1]
\State $M \leftarrow \emptyset$
\State $\vec{q} \leftarrow \vec{0}$ 
\For {$t=1,2,\ldots,k$}
\If {$t\neq 1$}
\State $\vec{q}\;' \leftarrow MHS(\vec{\sigma}_t,\vec{\delta}_t,M)$
{\Comment{(\ref{EquationSet:SingleMaxHeatLP_initial_stattement}) solve, Section \ref{Sec:MaxHeat_SingleInterval}}}
\State $\vec{q} \leftarrow \vec{q} + \vec{q}\;'$
\State For each $i\in H$, set $\sigma_{i,t} \leftarrow \sigma_{i,t} - \sum_{j\in C} \sum_{t\in T} q_{i,j,t}'$
\State For each $j\in C$, set $\delta_{j,t} \leftarrow \delta_{j,t} - \sum_{i\in H} \sum_{s\in T} q_{i,j,t}'$
\EndIf
\State $(M',\vec{q}\;') \leftarrow SingleTemperatureInterval(\vec{\sigma}_t,\vec{\delta}_t)$
{\Comment{Eqs (\ref{Eq:SingleMIPwithoutConservation_MaxBins} - \ref{Eq:SingleMIPwithoutConservation_Integrality}) or Alg IG, Sec \ref{Sec:SingleTemperatureIntervalProblem-Approximation}}}
\State $M \leftarrow M\cup M'$
\State $\vec{q} \leftarrow \vec{q} + \vec{q}\;'$
\If {$t\neq k$}
\For {$i\in H$}
\State $\vec{\sigma}_{i,t+1} \leftarrow \vec{\sigma}_{i,t+1} + (\vec{\sigma}_{i,t} - \sum_jq_{i,j,t})$ (excess heat descending)
\EndFor
\EndIf
\EndFor
\end{algorithmic}
\label{Alg:Water_Filling}
\end{algorithm}

%

%

\begin{theorem}
\label{Thm:WaterFillingRatio}
Algorithms WFG and WFM are $\Omega(k)$-approximate.
\end{theorem}
\begin{omitted_proof}
See \ref{App:Water_Filling}.
\end{omitted_proof}


\section{Greedy Packing Heuristics}
\label{sec:greedy_packing}

This section proposes greedy heuristics motivated by the packing nature of the minimum number of matches problem.
Each greedy packing heuristic starts from an infeasible solution with zero heat transferred between the streams and iterates towards feasibility by greedily selecting matches.
The two main ingredients of such a heuristic are: (i) a match selection policy and (ii) a heat exchange policy for transferring heat via the matches.
Section \ref{Subsection:MonotonicGreedyHeuristics} observes that a greedy heuristic has a poor worst-case performance if heat residual capacities are not considered.
Sections \ref{Subsection:Largest_Heat_Match} - \ref{Subsection:SS} define formally the greedy heuristics: (i) Largest Heat Match First, (ii) Largest Fraction Match First, and (iii) Smallest Stream First. 
{Figure \ref{Figure:Greedy_Packing_Ingredients} shows the main components of greedy packing heuristics.}

\begin{figure}
\begin{center}
\includegraphics{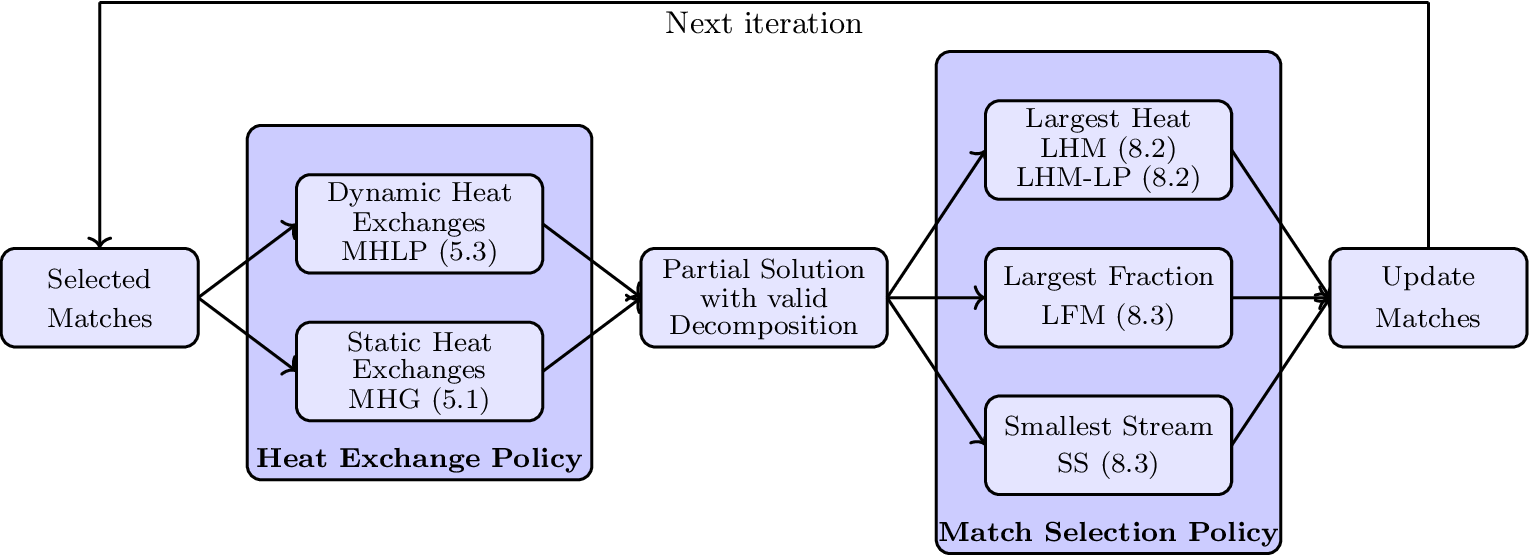}
\caption{
Greedy packing heuristics select matches iteratively one by one. 
The main components of greedy packing heuristics are (i) a heat exchange policy, and (ii) a match selection policy.
Greedy packing heuristics apply these policies with respect to all unmatched stream pairs, in each iteration.
Options for the heat exchange policy include \emph{dynamic heat exchange}, which solves the Section \ref{Sec:MaxHeat_MultipleIntervals} maximum heat LP, and \emph{static heat exchange}, which uses the Section \ref{Sec:MaxHeat_2Streams} greedy algorithm.
Once the heat exchange policy has been applied for every unmatched pair of streams, a match selection policy chooses the new match, e.g.\
(i) with the largest heat (LHM), (ii) with the largest fraction (LFM), or (iii) of the shortest stream (SS).}
\label{Figure:Greedy_Packing_Ingredients}
\end{center}
\end{figure}

\subsection{A Pathological Example and Heat Residual Capacities}
\label{Subsection:MonotonicGreedyHeuristics}

A greedy match selection heuristic is efficient if it performs a small number of iterations and chooses matches exchanging large heat load in each iteration.
Our greedy heuristics perform large moves towards feasibility by choosing good matches in terms of: (i) heat and (ii) stream fraction. 
An efficient greedy heuristic should also be monotonic in the sense that every chosen match achieves a strictly positive increase on the covered instance size.

The Figure \ref{Fig:non_monotonic_heuristic} example shows a pathological behavior of greedy non-monotonic heuristics.
The instance consists of 3 hot streams, 3 cold streams and 3 temperature intervals.
Hot stream $i\in H$ has heat supply $\sigma_{i,s}=1$ for $s=i$ and no supply in any other temperature interval.
Cold stream $j\in C$ has heat demand $\delta_{j,t}=1$ for $t=j$ and no demand in any other temperature interval.
Consider the heuristic which selects a match that may exchange the maximum amount of heat in each iteration.
The matches $(h_1,c_2)$ and $(h_2,c_3)$ consist the initial selections.
In the subsequent iteration, no match increases the heat that can be feasibly exchanged between the streams and the heuristic chooses unnecessary matches.

\begin{figure}[t]

\begin{center}
\includegraphics{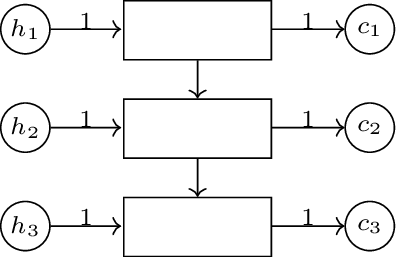}
\caption{A bad example of a non monotonic heuristic. If a heuristic begins by matching $h_1$ with $c_2$ and $h_2$ with $c_3$, then many unnecessary matches might be required to end up with a feasible solution.}
\label{Fig:non_monotonic_heuristic}
\end{center}

\end{figure}

A sufficient condition enforcing strictly monotonic behavior and avoiding the above pathology, is for each algorithm iteration to satisfy the heat residual capacities.
As depicted in Figure \ref{Fig:greedy_decomposition}, a greedy heuristic maintains a set $M$ of selected matches together with a decomposition of the original instance $I$ into two instances $I^A$ and $I^B$.
If $I=(H,C,T,\vec{\sigma},\vec{\delta})$, then it holds that $I^A=(H,C,T,\vec{\sigma}^A,\vec{\delta}^A)$ and $I^B=(H,C,T,\vec{\sigma}^B,\vec{\delta}^B)$, where $\mathbf{\sigma}=\vec{\sigma}^A+\vec{\sigma}^B$ and $\vec{\delta}=\vec{\delta}^A+\vec{\delta}^B$. 
The set $M$ corresponds to a feasible solution for $I^A$ and the instance $I^B$ remains to be solved.
In particular, $I^A$ is obtained by computing a maximal amount of heat exchanged by using the matches in $M$ and $I^B$ is the remaining part of $I$.
Initially, $I^A$ is empty and $I^B$ is exactly the original instance $I$.
A selection of a match increases the total heat exchanged in $I^A$ and reduces it in $I^B$.
\ref{App:Greedy_Packing} observes that a greedy heuristic is monotonic if $I^B$ is feasible in each iteration. 
Furthermore, $I^B$ is feasible if and only if $I^A$ satisfies the heat residual capacities $R_u = \sum_{i\in H}\sum_{s=1}^u\sigma_{i,s} - \sum_{j\in C}\sum_{t=1}^u\delta_{j,t}$, for $u\in T$.

\begin{figure}[t]

\begin{center}
\includegraphics{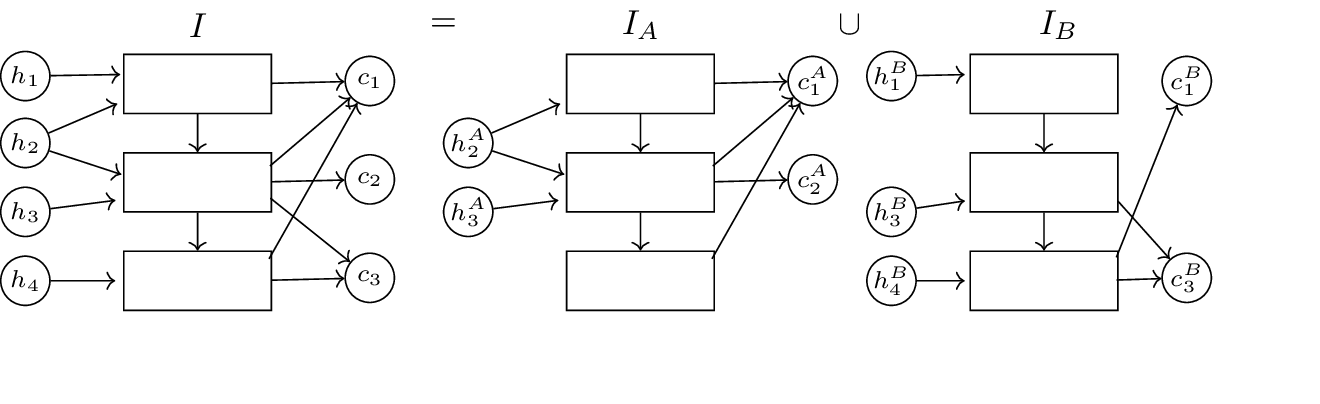}
\vspace{-20pt}
\caption{Decomposition of a greedy packing heuristic. The problem instance $I$ is the union of the instance $I_A$ already solved by the heuristic and the instance $I_B$ that remains to be solved.}
\label{Fig:greedy_decomposition}
\end{center}

\end{figure}

\subsection{Largest Heat Match First}
\label{Subsection:Largest_Heat_Match}

Our Largest Heat Match First heuristics arise from the idea that the matches should individually carry large amounts of heat in a near optimal solution.
Suppose that $Q_v$ is the maximum heat that may be transferred between the streams using only a number $v$ of matches.
Then, minimizing the number of matches is expressed as $\min\{v:Q_v\geq \sum_{i=1}^nh_i\}$.
This observation motivates the greedy packing heuristic which selects matches iteratively until it ends up with a feasible set $M$ of matches exchanging $\sum_{i=1}^nh_i$ units of heat.
In each iteration, the heuristic chooses a match maximizing the additional heat exchanged. 
Our two variants of largest heat matches heuristics are: (i) LP-based Largest Heat Match (LHM-LP) and (ii) Greedy Largest Heat Match (LHM).

Heuristic LHM-LP uses the $MHLP(M)$ (LP-based Maximum Heat) procedure to compute the maximum heat that can be transferred between the streams using only the matches in the set $M$.
This procedure is repeated $O(nm)$ times in each iteration, once for every candidate match, and solves an LP incorporating the proposed heat residual capacities.
Algorithm \ref{Alg:LHM-LP} is an LHM-LP heuristic using the LP in Section \ref{Sec:MaxHeat_MultipleIntervals}.
The algorithm maintains a set $M$ of chosen matches and selects a new match $(i',j')$ to maximize $MHLP(M\cup(i',j'))$.

\begin{algorithm}[t] \nonumber
\caption{Largest Heat Match First LP-based (LHM-LP)}
\begin{algorithmic}[1]
\State $M\leftarrow\emptyset$
\State $r\leftarrow \sum_{i\in H}h_i$
\While {$r>0$}
\State $(i',j')\leftarrow\arg\max_{(i,j)\in H\times C\setminus M} \{MHLP(M\cup\{(i,j)\})\}$
{\Comment{Eqs (\ref{Eq:MaxHeatLP_Objective} - \ref{Eq:MaxHeatLP_Positiveness}), Sec \ref{Sec:MaxHeat_MultipleIntervals}}}
\State $M\leftarrow M\cup \{(i',j')\}$
\State $r\leftarrow\sum_{i\in H}h_i-MHLP(M)$
\EndWhile
\State Return $M$
\end{algorithmic}
\label{Alg:LHM-LP}
\end{algorithm}

\begin{theorem}
\label{Thm:GreedyPackingRatio}
Algorithm LHM-LP is $O(\log n + \log \frac{h_{\max}}{\epsilon})$-approximate, where $\epsilon$ is the required precision.
\end{theorem}
\begin{omitted_proof}
See \ref{App:Greedy_Packing}.
\end{omitted_proof}

\medskip

LHM-LP heuristic is polynomial-time in the worst case.
The $i$-th iteration solves $nm-i+1$ LP instances which sums to solving a total of $\sum_{i=1}^{nm}(nm-i+1)=O(n^2m^2)$ LP instances in the worst case.
However, for large instances, the algorithm is time consuming because of this iterative LP solving.
So, we also propose an alternative, time-efficient greedy approach. 
The new heuristic version builds a solution by selecting matches and deciding the heat exchanges, without modifying them in subsequent iterations.

The new approach for implementing the heuristic, that we call LHM, requires the $MHG(\vec{\sigma},\vec{\delta},i,j)$ procedure. 
Given an instance $(\vec{\sigma},\vec{\delta})$ of the problem, it computes the maximum heat that can be feasibly exchanged between hot stream $i\in H$ and cold stream $j\in C$, as defined in Section \ref{Sec:MaxHeat_2Streams}.
The procedure also computes a corresponding value $q_{i,s,j,t}$ of heat exchanged between $i\in H$ in temperature interval $s\in T$ and $j\in C$ in temperature interval $t\in T$.
LHM maintains a set $M$ of currently chosen matches together with their respective vector $\vec{q}$ of heat exchanges.
In each iteration, it selects the match $(i',j')$ and heat exchanges $q'$ between $i'$ and $j'$ so that the value $MHG(\vec{\sigma},\vec{\delta},i',j')$ is maximum.
Algorithm \ref{Alg:LHM} is a pseudocode of this heuristic.

%

\begin{algorithm}[t] \nonumber
\caption{Largest Heat Match First Greedy (LHM)}
\begin{algorithmic}[1]
\State $M \leftarrow \emptyset$
\State $\vec{q} \leftarrow \vec{0}$
\State $r \leftarrow \sum_{i\in H}h_i$
\While {$r > 0$}
\State $(i', j', \vec{q}\;') \leftarrow \arg \max_{(i,j)\in H\times C\setminus M} \{MHG(\vec{\sigma},\vec{\delta},i,j)\}$
{\Comment{Algorithm MHG, Section \ref{Sec:MaxHeat_2Streams}}}
\State $M \leftarrow M \cup \{(i',j')\}$
\State $\vec{q} \leftarrow \vec{q} + \vec{q}\;'$
\State For each $s\in T$, set $\sigma_{i',s} \leftarrow \sigma_{i',s} - \sum_{t\in T} q_{i',s,j',t}'$
\State For each $t\in T$, set $\delta_{j',t} \leftarrow \delta_{j',t} - \sum_{s\in T} q_{i',s,j',t}'$
\State $r \leftarrow r - \sum_{s,t\in T}q_{i',s,j',t}'$
\EndWhile
\State Return $M$
\end{algorithmic}
\label{Alg:LHM}
\end{algorithm}

\subsection{Largest Fraction Match First}
\label{Subsection:LFM}

The heuristic \emph{Largest Fraction Match First} (LFM) exploits the bipartite nature of the problem by employing matches which exchange large fractions of the stream heats.
Consider a feasible solution with a set $M$ of matches.
Every match $(i,j)\in M$ covers a fraction $\sum_{s,t\in T}\frac{q_{i,s,j,t}}{h_i}$ of hot stream $i\in H$ and a fraction $\sum_{s,t\in T}\frac{q_{i,s,j,t}}{c_j}$ of cold stream $j\in C$.
The total covered fraction of all streams is equal to $\sum_{(i,j)\in M} \sum_{s,t\in T} \left( \frac{q_{i,s,j,t}}{h_i} + \frac{q_{i,s,j,t}}{c_j}\right)=n+m$.
Suppose that $F_v$ is the maximum amount of total stream fraction that can be covered using no more than $v$ matches.
Then, minimizing the number of matches is expressed as $\min\{v:F_v\geq n+m\}$.
Based on this observation, the main idea of LFM heuristic is to construct iteratively a feasible set of matches, by selecting the match covering the largest fraction of streams, in each iteration.
That is, LFM prioritizes proportional matches in a way that high heat hot streams are matched with high heat cold streams and low heat hot streams with low heat cold streams. 
In this sense, it generalizes the idea of Algorithm IG for the single temperature interval problem (see Section \ref{sec:single_temperature_interval}), according to which it is beneficial to match streams of (roughly) equal heat.

An alternative that would be similar to LHM-LP is an
LFM heuristic 
with an $MFLP(M)$ (LP-based Maximum Fraction) procedure computing the maximum fraction of streams that can be covered using only a given set $M$ of matches.
Like the LHM-LP heuristic, this procedure would be based on solving an LP (see \ref{App:Greedy_Packing}), except that the objective function maximizes the total stream fraction.
The LFM heuristic can be also modified to attain more efficient running times using Algorithm $MHG$, as defined in Section \ref{Sec:MaxHeat_2Streams}. 
In each iteration, the heuristic selects the match $(i,j)$ with the highest value $\frac{U_{i,j}'}{h_i}+\frac{U_{i,j}'}{c_j}$, where $U_{i,j}'$ is the maximum heat that can be feasibly exchanged between $i$ and $j$ in the remaining instance.


\subsection{Smallest Stream Heuristic}
\label{Subsection:SS}

Subsequently, we propose \emph{Smallest Stream First} (SS) heuristic based on greedy match selection, which also incorporates stream priorities so that a stream is involved in a small number of matches.
Let $\alpha_i$ and $\beta_j$ be the number of matches of hot stream $i\in H$ and cold stream $j\in C$, respectively.
Minimizing the number of matches problem is expressed as $\min\{\sum_{i\in H}\alpha_i\}$, or equivalently $\min\{\sum_{j\in C}\beta_j\}$.
Based on this observation, we investigate heuristics that specify a certain order of the hot streams and match them one by one, using individually a small number of matches.
Such a heuristic requires: (i) a stream ordering strategy and (ii) a match selection strategy.
To reduce the number of matches of small hot streams, heuristic SS uses the order $h_1\leq h_2\leq \ldots\leq h_n$.  

In each iteration, the next stream is matched with a low number of cold streams using a greedy match selection strategy; we use greedy LHM heuristic.
Observe that SS heuristic is more efficient in terms of running time compared to the other greedy packing heuristics, because it solves a subproblem with only one hot stream in each iteration. 
Algorithm \ref{Alg:ShortestStream} is a pseudocode of SS heuristic. 
Note that other variants of ordered stream heuristics may be obtained in a similar way.
The heuristic uses the $MHG$ algorithm in Section \ref{Sec:MaxHeat_2Streams}.

\begin{algorithm}[t] \nonumber
\caption{Smallest Steam First (SS)}
\begin{algorithmic}[1]
\State Sort the hot streams in non-decreasing order of their heat loads, i.e.\ $h_1\leq h_2\leq\ldots\leq h_n$.
\State $M \leftarrow \emptyset$
\State $\vec{q}\leftarrow \vec{0}$
\For {$i\in H$}
\State $r\leftarrow h_i$
\While {$r>0$}
\State $(i, j', \vec{q}\;') \leftarrow \arg \max_{j\in C} \{MHG(\vec{\sigma},\vec{\delta},i,j)\}$
{\Comment{Algorithm MHG, Section \ref{Sec:MaxHeat_2Streams}}}
\State $M\leftarrow M\cup \{(i,j')\}$
\State $\vec{q} \leftarrow \vec{q} + \vec{q}\;'$
\State For each $s\in T$, set $\sigma_{i,s} \leftarrow \sigma_{i,s} - \sum_{t\in T} q_{i,s,j',t}'$
\State For each $t\in T$, set $\delta_{j',t} \leftarrow \delta_{j',t} - \sum_{s\in T} q_{i,s,j',t}'$
\State $r \leftarrow r - \sum_{s,t\in T}q_{i',s,j',t}'$
\EndWhile
\EndFor
\State Return $M$
\end{algorithmic}
\label{Alg:ShortestStream}
\end{algorithm}



\section{Numerical Results}
\label{sec:results}

This section evaluates the proposed heuristics on three test sets. 
Section \ref{Section:Benchmark_Instances} provides information on system specifications and benchmark instances.
Section \ref{Section:Exact_Experiments} presents computational results of exact methods and shows that commercial, state-of-the-art approaches have difficult solving moderately-sized instances 
to global optimality.
Section \ref{Section:Heuristic_Experiments} evaluates experimentally the heuristic methods and compares the obtained results with those reported by \citet{furman:2004}.
All result tables are provided in \ref{App:Experimental_Results}.
{\cite{source_code} provide test cases and source code for the paper's computational experiments.}

\subsection{System Specification and Benchmark Instances}
\label{Section:Benchmark_Instances}

All computations are run on an Intel Core i7-4790 CPU 3.60GHz with 15.6 GB RAM running 64-bit Ubuntu 14.04.
CPLEX 12.6.3 and Gurobi 6.5.2 solve the minimum number of matches problem exactly. 
The mathematical optimization models and heuristics are implemented in Python 2.7.6 and Pyomo 4.4.1 \citep{hart:2011, hart:2012}. 

We use problem instances from two existing test sets \citep{furman:2004, chen:2015}. {We also generate two collections of larger test cases. The smaller of the two sets uses work of \citet{grossman:2017}. The larger of the two sets was created using our own random generation method.}
An instance of general heat exchanger network design consists of streams and utilities with inlet, outlet temperatures, flow rate heat capacities and other parameters.
\ref{App:Minimum_Utility_Cost} shows how a minimum number of matches instances arises from the original instance of general heat exchanger network design.
 
The \emph{\cite{furman:2000} test set} consists of test cases from the engineering literature. 
Table \ref{Table:Problem_Sizes} reports bibliographic information on the origin of these test cases. 
We manually digitize this data set and make it publicly available for the first time \citep{source_code}.
Table \ref{Table:Problem_Sizes} lists the 26 problem instance names and information on their sizes.
The total number streams and temperature intervals varies from 6 to 38 and from 5 to 32, respectively. 
Table \ref{Table:Problem_Sizes} also lists the number of binary and continuous variables as well as the number of constraints in the transshipment MILP formulation.

The \emph{\cite{minlp,chen:2015} test set} consists of 10 problem instances.
These instances are classified into two categories depending on whether they consist of balanced or unbalanced streams. 
Test cases with balanced streams have flowrate heat capacities in the same order of magnitude, 
while test cases with unbalanced streams have dissimilar flowrate heat capacities spanning several orders of magnitude. 
The sizes of these instances range from 10 to 42 streams and from 12 to 35 temperature intervals. 
Table \ref{Table:Problem_Sizes} reports more information on the size of each test case.

\emph{The \cite{grossman:2017} test set} is generated randomly.
The inlet, outlet temperatures of these instances are fixed while the values of flowrate heat capacities are generated randomly with fixed seeds.
This test set contains 12 moderately challenging problems (see Table \ref{Table:Problem_Sizes}) with a classification into balanced and unbalanced instances, similarly to the \cite{minlp,chen:2015} test set. 
The smallest problem involves 27 streams and 23 temperature intervals while the largest one consists of 43 streams and 37 temperature intervals.

{\emph{The Large Scale test set} is generated randomly. These instances have 80 hot streams, 80 cold streams, 1 hot utility and 1 cold utility.
For each hot stream $i\in HS$, the inlet temperature $T_{\text{in},i}^{HS}$ is chosen uniformly at random in the interval $(30,400]$. 
Then, the outlet temperature $T_{\text{out},i}^{HS}$ is selected uniformly at random in the interval $[30, T_{\text{in},i}^{HS})$.
Analogously, for each cold stream $j\in CS$, the outlet temperature $T_{\text{out},j}^{CS}$ is chosen uniformly at random in the interval $(20,400]$.
Next, the inlet temperature $T_{\text{in},j}^{CS}$ is chosen uniformly at random in the interval $[20,T_{\text{out},j}^{CS})$.
The flow rate heat capacities $FCp_i$ and $FCp_j$ of hot stream $i$ and cold stream $j$ are chosen as floating numbers with two decimal digits in the interval $[0,15]$.
The hot utility has inlet temperature $T_{\text{in}}^{HU}=500$, outlet temperature $T_{\text{out}}^{HS}=499$, and cost $\kappa^{HU}=80$.
The cold utility has inlet temperature $T_{\text{in}}^{CU}=20$, outlet temperature $T_{\text{out}}^{CU}=21$, and cost $\kappa^{CU}=20$.
The minimum heat recovery approach temperature is $\Delta T_{\min}=10$.}

\subsection{Exact Methods} 
\label{Section:Exact_Experiments}

We evaluate exact methods using state-of-the-art commercial approaches. 
For each problem instance, CPLEX and Gurobi solve the Section \ref{sec:preliminaries} transportation and transshipment models.
Based on the difficulty of each test set, we set a time limit for each solver run as follows: (i) 1800 seconds for the \cite{furman:2000} test set, (ii) 7200 seconds for the \cite{minlp,chen:2015} test set, and (iii) 14400 seconds for the \cite{grossman:2017} {and large scale} test sets. 
In each solver run, we set absolute gap 0.99, relative gap $4\%$, and maximum number of threads 1.

Table \ref{Table:Exact_Methods} reports the best found objective value, CPU time and relative gap, for each solver run.
Observe that state-of-the-art approaches cannot, in general, solve moderately-sized problems with 30-40 streams to global optimality.
For example, none of the test cases in the \cite{grossman:2017} {or large scale} test sets is solved to global optimality within the specified time limit. 
Table \ref{Table:Comparisons} contains the results reported by \citet{furman:2004} using CPLEX 7.0 with 7 hour time limit. 
CPLEX 7.0 fails to solve 4 instances to global optimality.
Interestingly, CPLEX 12.6.3 still cannot solve 3 of these 4 instances with a 1.5 hour timeout.

Theoretically, the transshipment MILP is better than the transportation MILP because the former has asymptotically fewer variables.
This observation is validated experimentally with the exception of very few instances, e.g.\ \texttt{balanced10}, in which the transportation model computes a better solution within the time limit.
CPLEX and Gurobi are comparable and neither dominates the other.
Instances with balanced streams are harder to solve, which highlights the difficulty introduced by symmetry, see \cite{kouyialis:2016}. 
{
The preceding numerical analysis refers to the extended transportation MILP.
Table \ref{Table:Transportation_Models_Comparison} compares solver performance to the reduced transportation MILP, i.e.\ a formulation removing redundant variables $q_{i,s,j,t}$ with $s> t$ and Equations (\ref{TransportationMIP_Eq:ThermoConstraint}). Note that modern versions of CPLEX and Gurobi show effectively no difference between the two formulations.
}
  
\subsection{Heuristic Methods}  
\label{Section:Heuristic_Experiments}

{
We implement the proposed heuristics using Python and develop the LP models with Pyomo \citep{hart:2011,hart:2012}.
We use CPLEX 12.6.3 with default settings to solve all LP models within the heuristic methods. \cite{source_code} make the source code available. The following discussion covers the 48 problems with 43 streams or fewer. Section \ref{Section:Large_Scale} discusses the 3 examples with 160 streams each.}

The difficulty of solving the minimum number of matches problem to global optimality motivates the design of heuristic methods and approximation algorithms with proven performance guarantees.
Tables \ref{Table:Heuristic_Upper_Bounds} and \ref{Table:Heuristic_CPU_Times} contain the computed objective value and CPU times, respectively, of the heuristics for all test cases. 
For the challenging \cite{minlp,chen:2015} and \cite{grossman:2017} test sets, heuristic LHM-LP always produces the best solution.
The LHM-LP running time is significantly higher compared to all heuristics due to the iterative LP solving, despite the fact that it is guaranteed to be polynomial in the worst case.
Alternatively, heuristic SS produces the second best heuristic result with very efficient running times in the \cite{minlp,chen:2015} and \cite{grossman:2017} test sets. 
Figure \ref{Fig:Boxplot_Performance_Ratios} depicts the performance ratio of the proposed heuristics using a box and whisker plot, where the computed objective value is normalized with the one found by CPLEX for the transshipment MILP.
Figure \ref{Fig:Boxplot_CPU_Times} shows a box and whisker plot of the CPU times of all heuristics in $\log_{10}$ scale normalized by the minimum CPU time for each test case.
Figure \ref{Fig:Line_Chart} shows a line chart verifying that our greedy packing approach produces better solutions than the relaxation rounding and water filling ones. 

\begin{figure}[!ht]
\begin{center}
\includegraphics[scale=0.6]{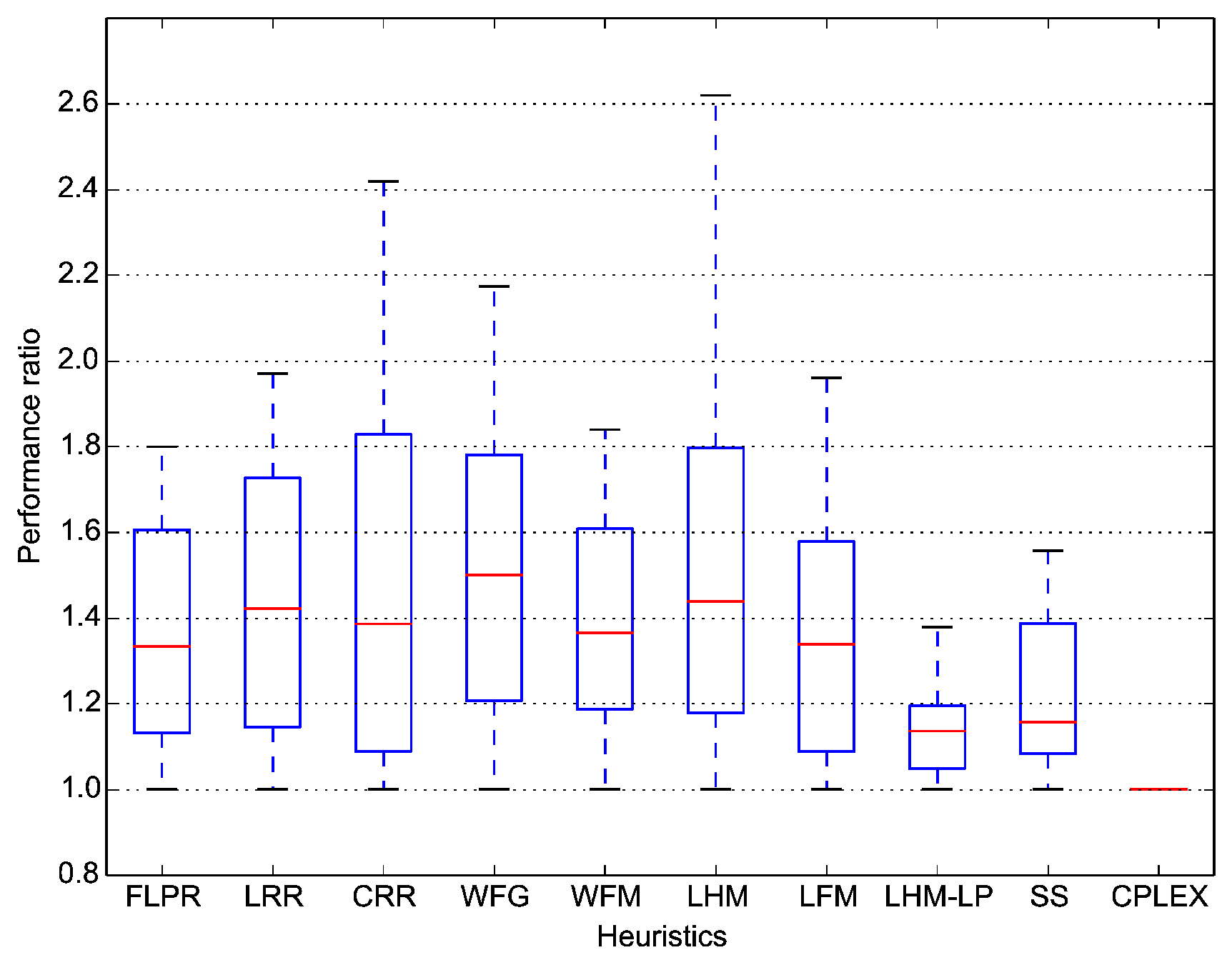}
\end{center}
\vspace{-20pt}
\caption{Box and whisker diagram of {48} heuristic performance ratios, i.e.\ computed solution / best known solution {for the problems with 43 streams or fewer.}}
\label{Fig:Boxplot_Performance_Ratios}
\end{figure}

\begin{figure}[!ht]
\begin{center}
\includegraphics[scale=0.6]{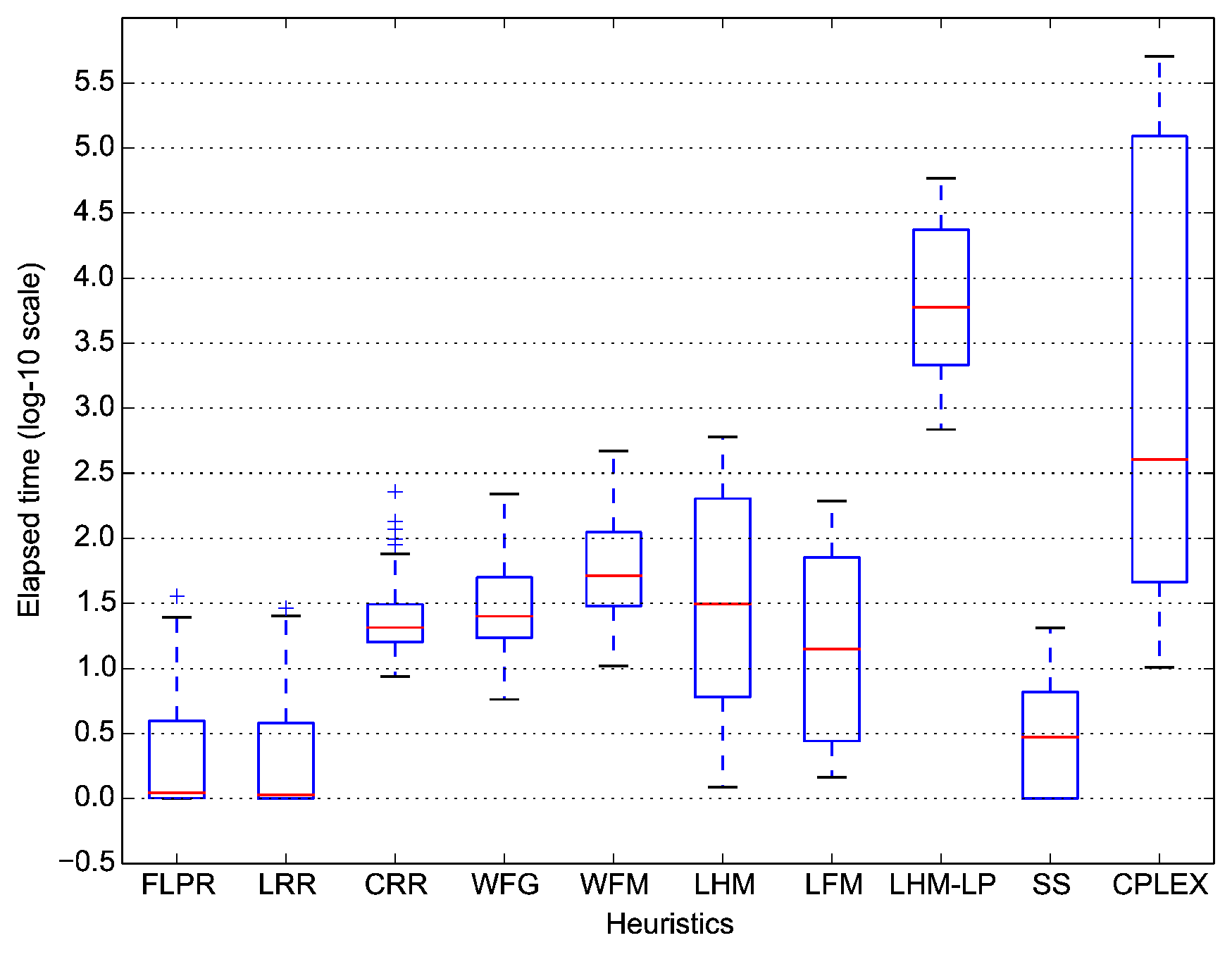}
\end{center}
\vspace{-20pt}
\caption{Box and whisker diagram of {48} CPU times ($\log_{10}$ scale) normalized by the minimum CPU time for each test case.}
\label{Fig:Boxplot_CPU_Times}
\end{figure}

\begin{figure}[!ht] 
\begin{center}
\includegraphics[scale=0.6]{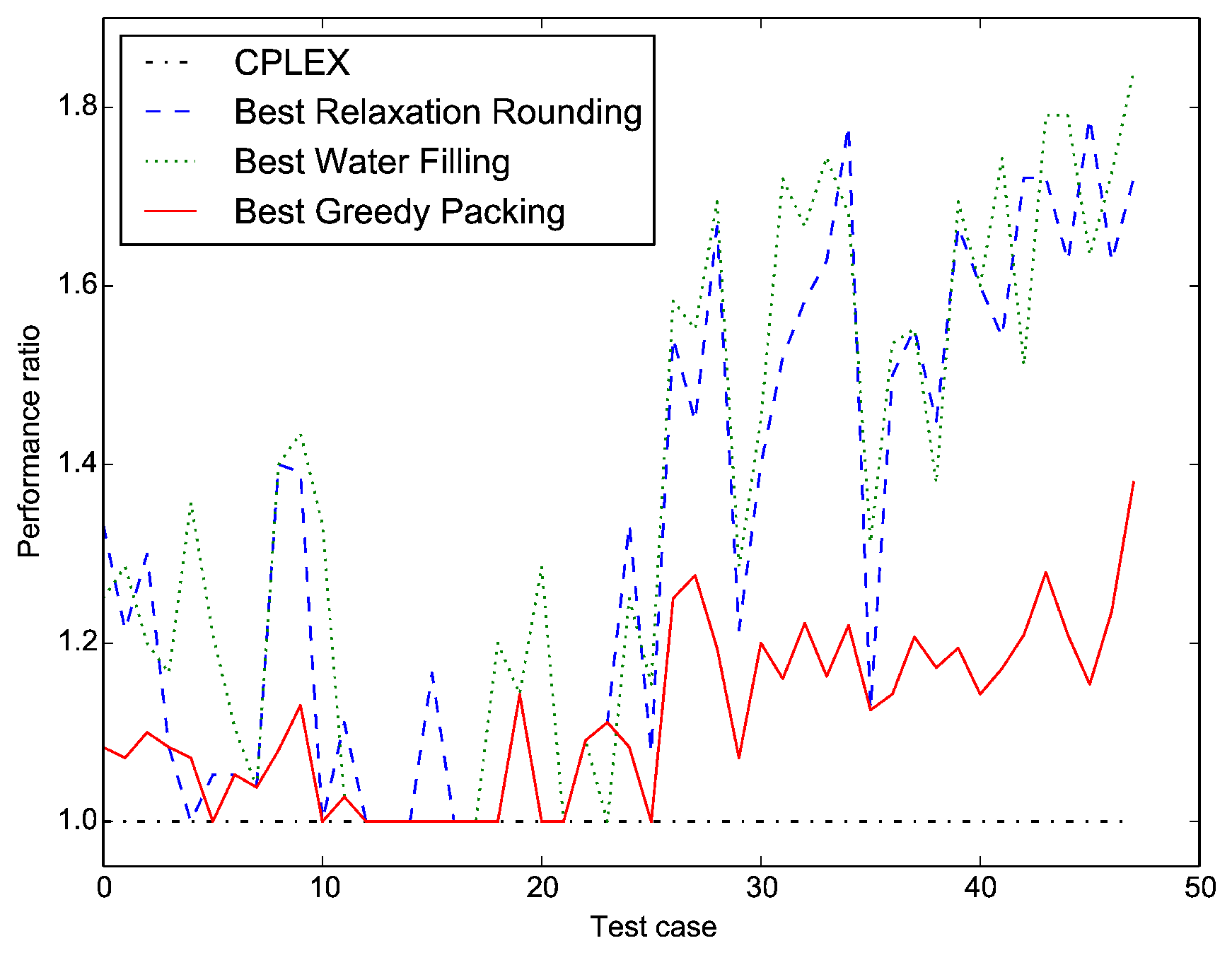}
\end{center}
\vspace{-20pt}
\caption{Line chart comparing the performance ratio, i.e.\ computed solution / best known solution, of the best computed result by heuristic methods: relaxation rounding, water filling, and greedy packing. {This graph applies to the 48 problems with 43 streams or fewer.}}
\label{Fig:Line_Chart}
\end{figure}

Table \ref{Table:Comparisons} contains the heuristic results reported by \citet{furman:2004} and the ones obtained with our improved version of the FLPR, LRR, and WFG heuristics of \cite{furman:2004}.
Our versions of FLPR, LRR, and WFG perform better for the \cite{furman:2004} test set because of our new Algorithm MHG for tightening the big-M parameters.
For example, out of the 26 instances, our version of FLPR performs strictly better than the \cite{furman:2004} version 20 times and worse only once (\texttt{10sp1}). To further explore the effect of the big-M parameter, Table \ref{Table:BigM_Comparisons} shows how different computations for the big-M parameter change the FLPR and LRR performance. Table \ref{Table:BigM_Comparisons} also demonstrates the importance of the big-M parameter on the transportation MILP fractional relaxation quality. 

In particular, Table \ref{Table:BigM_Comparisons} compares the three big-M computation methods discussed in Section \ref{Sec:MaxHeat_2Streams}: (i) the trivial bounds, (ii) the \citet{gundersen:1997} method, and (iii) our greedy Algorithm MHG.
Our greedy maximum heat algorithm dominates the other approaches for computing the big-M parameters. Algorithm MHG also outperforms the other two big-M computation methods by finding smaller feasible solutions via both Fractional LP Rounding and Lagrangian Relaxation Rounding.  In the 48 test cases, Algorithm MHG produces the best FLPR and LRR feasible solutions in 46 and 43 test cases, respectively. Algorithm MHG is strictly best for 33 FLPR and 32 LRR test cases. Finally, Algorithm MHG achieves the tightest fractional MILP relaxation for all test instances.

Figure \ref{Fig:Boxplot_Performance_Ratios} and Table \ref{Table:Heuristic_Upper_Bounds} show that our
new CRR heuristic is competitive with the other relaxation rounding heuristics, performing as well or better than FLPR or LRR in 19 of the 48 test cases and strictly outperforming both FLPR and LRR in 8 test cases.
%
%
Although CRR solves a sequence of MILPs, Figure \ref{Fig:Boxplot_CPU_Times} and Table \ref{Table:Heuristic_CPU_Times} show that its running time is efficient compared to the other relaxation rounding heuristics.

Our water filling heuristics are equivalent to or better than \citet{furman:2004} for 25 of their 26 test set instances (all except \texttt{7sp2}).  
In particular, our Algorithm WFG is strictly better than their WFG in 18 of 26 instances and is worse in just one. This improvement stems from the new 1.5-approximation algorithm for the single temperature interval problem (see Section \ref{Sec:SingleTemperatureIntervalProblem-Approximation}).
The novel Algorithm WFM is competitive with Algorithm WFG and produces equivalent or better feasible solutions for 37 of the 48 test cases. 
In particular, WFM has a better performance ratio than WFG (see Figure \ref{Fig:Boxplot_Performance_Ratios}) and WFM is strictly better than WFG in all but 1 of the \citet{grossman:2017} instances.
The strength of WFM highlights the importance of our new MILP formulation in Eqs.\ (\ref{Eq:SingleMIP_MaxBins})-(\ref{Eq:SingleMIP_Integrality}).
At each iteration, WFM solves an MILP without big-M constraints and therefore has a running time in the same order of magnitude as its greedy counterpart WFG (see Figure \ref{Fig:Boxplot_CPU_Times}).

In summary, our heuristics obtained via the relaxation rounding and water filling methods improve the corresponding ones proposed by \citet{furman:2004}.
Furthermore, greedy packing heuristics achieve even better results in more than $90\%$ of the test cases. 

{

\subsection{Larger Scale Instances}
\label{Section:Large_Scale}

Although CPLEX and Gurobi do not converge to global optimality for many of the 
\cite{furman:2000}, \cite{minlp,chen:2015}, and \cite{grossman:2017} instances,
the solvers produce the best heuristic solutions in all test cases.
But the literature instances are only moderately sized. We expect that the heuristic performance improves relative to the exact approaches as the problem sizes increase.
Towards a more complete numerical analysis, we randomly generate 3 larger scale instances with 160 streams each.

For larger problems, the running time may be important to a design engineer \citep{hindmarsh:1983}.
We apply the least time consuming heuristic of each type for solving the larger scale instances, i.e.\ apply relaxation rounding heuristic FLPR, water filling heuristic WFG, and greedy packing heuristic SS.
We also solve the transshipment model using CPLEX 12.6.3 with a 4h timeout.  The results are in Table \ref{Table:Large_Scale_Results}.

For these instances, greedy packing SS computes a better solution than the relaxation rounding FLPR heuristic or the water filling WFG heuristic, but SS has larger running time.
In instance \texttt{large-scale1}, greedy packing SS computes 218, a better solution than the CPLEX value 219. Moreover, the CPLEX heuristic spent the first 1hr of computation time at solution 257 (18\% worse than the solution SS obtains in 10 minutes) and the next 2hr of computation time at solution 235 (8\% worse than the solution SS obtains in 10 minutes). Any design engineer wishing to interact with the results would be frustrated by these times.

In instance \texttt{large-scale2}, CPLEX computes a slightly better solution (239) than the SS heuristic (242).
But the good CPLEX solution is computed slightly before the 4h timeout. For more than 3.5hr, the best CPLEX heuristic is 273 (13\% worse than the solution SS obtains in 10 minutes).
Finally, in instance \texttt{large-scale0}, CPLEX computes a significantly better solution (175) than the SS heuristic (233).
But CPLEX computes the good solution after 2h and the incumbent is similar to the greedy packing SS solution for the first 2 hours.
These findings demonstrate that greedy packing approaches are particularly useful when transitioning to larger scale instances.

Note that we could additionally study approaches to improve the heuristic performance of CPLEX, e.g.\ by changing CPLEX parameters or using a parallel version of CPLEX. But the point of this paper is to develop a deep understanding of a very important problem that consistently arises in process systems engineering \citep{floudas:2012}.
}

\section{Discussion of Manuscript Contributions}
\label{sec:discussion}

This section reflects on this paper's contributions and situates the work with respect to existing literature. We begin in Section \ref{sec:single_temperature_interval} by designing efficient heuristics for the minimum number of matches problem with the special case of a single temperature interval.
Initially, we show that the 2 performance guarantee by \citet{furman:2004} is tight.
Using graph theoretic properties, we propose a new MILP formulation for the single temperature interval problem which does not contain any big-M constraints. We also develop an improved, tight, greedy 1.5-approximation algorithm which prioritizes stream matches with equal heat loads. Apart from the its independent interest, solving the single temperature interval problem is a major ingredient of water filling heuristics.


The multiple temperature interval problem requires big-M parameters. We reduce these parameters in Section \ref{sec:max_heat} by
computing the maximum amount of heat transfer with match restrictions.
Initially, we present a greedy algorithm for exchanging the maximum amount of heat between two streams.
This algorithm computes tighter big-M parameters than \citet{gundersen:1997}.
We also propose LP-based ways for computing the maximum exchanged heat using only a subset of the available matches.
Maximum heat computations are fundamental ingredients of our heuristic methods and detect the overall problem feasibility. This paper emphasizes how tighter big-M parameters improve heuristics with performance guarantees, but notice that improving the big-M parameters will also tend to improve exact methods.

Section \ref{sec:relaxation_rounding} further investigates the \emph{relaxation rounding heuristics} of \cite{furman:2004}. 
\citet{furman:2004} propose a heuristic for the minimum number of matches problem based on rounding the LP relaxation of the transportation MILP formulation
(\emph{Fractional LP Rounding (FLPR)}).
Initially, we formulate the LP relaxation as a minimum cost flow problem showing that it can be solved with network flow techniques which are more efficient than generic linear programming.
We derive a negative performance guarantee showing that FLPR has poor performance in the worst case.
We also prove a new positive performance guarantee for FLPR indicating that its worst-case performance may be improved with tighter big-M parameters.
Experimental evaluation shows that the performance of FLPR improves with our tighter algorithm for computing big-M parameters.
Motivated by the method of Lagrangian Relaxation, \citet{furman:2004} proposed an approach generalizing FLPR by approximating the cost of the heat transferred via each match.
We revisit possible policies for approximating the cost of each match.
Interestingly, we show that this approach can be used as a generic method for potentially improving a solution of the minimum number of matches problem.
Heuristic \emph{Lagrangian Relaxation Rounding (LRR)} aims to improve the solution of FLPR in this way.
Finally, we propose a new heuristic, namely \emph{Covering Relaxation Rounding (CRR)}, that successively solves instances of a new covering relaxation which also requires big-M parameters.

Section \ref{sec:water_filling} defines \emph{water filling heuristics} as a class of heuristics solving the minimum number of matches problem in a top-down manner, i.e.\ from highest to lowest temperature interval.
\citet{cerdaetwesterberg:1983} and \citet{furman:2004} have solution methods based on water filling.
We improve these heuristics by developing novel, efficient ways for solving the single temperature interval problem.
For example, heuristics \emph{MILP-based Water Filling (WFM)} and \emph{Greedy Water Filling (WFG)} incorporate the new MILP formulation (Eqs.\ \ref{Eq:SingleMIP_MaxBins}-\ref{Eq:SingleMIP_Integrality}) and greedy Algorithm IG, respectively.
With appropriate LP, we further improve water filling heuristics by reusing in each iteration matches selected in previous iterations.
\citet{furman:2004} showed a performance guarantee scaling with the number of temperature intervals. 
We show that this performance guarantee is asymptotically tight for water filling heuristics. 

Section \ref{sec:greedy_packing} develops a new \emph{greedy packing approach} for designing efficient heuristics for the minimum the number of matches problem motivated by the packing nature of the problem.
Greedy packing requires feasibility conditions which may be interpreted as a decomposition method analogous to pinch point decomposition, see
\cite{hindmarsh:1983}.
Similarly to \cite{cerdaetwesterberg:1983}, stream ordering affects the efficiency of greedy packing heuristics.
Based on the feasibility conditions, the LP in Eqs.\ (\ref{Eq:MaxHeatLP_Objective})-(\ref{Eq:MaxHeatLP_Positiveness}) selects matches carrying a large amount of heat and incurring low unitary cost for exchanging heat.
Heuristic \emph{LP-based Largest Heat Match (LHM-LM)} selects matches greedily by solving instances of this LP.
Using a standard packing argument, we obtain a new logarithmic performance guarantee.
LHM-LP has a polynomial worst-case running time but is experimentally time-consuming due to the repeated LP solving.
We propose three other greedy packing heuristic variants which improve the running time at the price of solution quality.
These other variants are based on different time-efficient strategies for selecting good matches.
Heuristic \emph{Largest Heat Match (LHM)} selects matches exchanging high heat in a pairwise manner.
Heuristic \emph{Largest Fraction Match (LFM)} is inspired by the idea of our greedy approximation algorithm for the single temperature interval problem which prioritizes roughly equal matches.
Heuristic \emph{Smallest Stream First (SS)} is inspired by the idea of the tick-off heuristic \citep{hindmarsh:1983} and produces matches in a stream to stream basis, where a hot stream is ticked-off by being matched with a small number of cold streams. 

Finally, Section \ref{sec:results} 
shows numerically that our new way of computing the big-M parameters, our improved algorithms for the single temperature interval, and the other enhancements improve the performance of relaxation rounding and water-filling heuristics.
The numerical results also show that our novel greedy packing heuristics {typically find better feasible solutions than} relaxation rounding and water-filling ones. {But the tradeoff is that the relaxation rounding and water filling algorithms achieve very efficient run times.}

\section{Conclusion}
\label{sec:conclusion}

In his PhD thesis, Professor Floudas showed that, given a solution to the minimum number of matches problem, he could solve a nonlinear optimization problem designing effective heat recovery networks. But the sequential HENS method cannot guarantee that promising minimum number of matches solutions will be optimal (or even feasible!) to Professor Floudas' nonlinear optimization problem. Since the nonlinear optimization problem is relatively easy to solve, we propose generating many good candidate solutions to the minimum number of matches problem. This manuscript develops nine heuristics with performance guarantees to the minimum number of matches problem. Each of the nine heuristics is either novel or provably the best in its class. Beyond approximation algorithms, our work has interesting implications for solving the minimum number of matches problem exactly, e.g.\ the analysis into reducing big-M parameters {or the possibility of quickly generating good primal feasible solutions}. \\[-4pt]

\noindent
\textbf{Acknowledgments} \\[2pt]
\noindent
We gratefully acknowledge support from EPSRC EP/P008739/1, an EPSRC DTP to G.K., and a Royal Academy of Engineering Research Fellowship to R.M.


\bibliographystyle{ijocv081}
\bibliography{refs}


\newpage
\appendix

\section{$\mathcal{NP}$-hardness Reduction}
\label{App:NP_harndess}

\begingroup
\def\thetheorem{\ref{thm:NP_hardness}}
\begin{theorem}
There exists an $\mathcal{NP}$-hardness reduction from bin packing to the minimum number of matches problem with a single temperature interval.
\end{theorem}
\addtocounter{theorem}{-1}
\endgroup

\begin{proof}
Initially, define the decision version of bin packing.
A bin packing instance consists of a set $B=\{1,2,\ldots,m\}$ of bins, each bin of capacity $K$, and a set $O=\{1,2,\ldots,n\}$ of objects, where object $i\in O$ has size $s_i\in(0,K]$.
The goal is to determine whether there exist a feasible packing $O_1,O_2,\ldots,O_m$ of the objects into the bins, where $O_j\subseteq O$ is the subset of objects packed in bin $j\in B$. 
Each object is placed in exactly one bin, i.e.\ $\cup_{j=1}^mO_j=O$ and $O_j\cap O_{j'}=\emptyset$ for each $1\leq j<j'\leq m$, and the total size of the objects in a bin do not exceed its capacity, i.e.\ $\sum_{i\in O_j}s_i\leq K$, for $j\in B$.

Consider an instance $(O,n,B,m)$ of bin packing.
Construct an instance of the minimum number of matches problem with a single temperature interval by setting $H=O$, $h_i=s_i$ for $i=1,\ldots,n$, $C=B$ and $c_j=K$ for $j=1,\ldots,m$.
We claim that bin packing has a feasible solution if and only if the constructed minimum number of matches instance is feasible using exactly $n$ matches.

To the first direction, consider a feasible packing $O_1,\ldots,O_m$.
For each $i\in H$ and $j\in C$, we obtain a solution for the minimum number of matches instance by setting $q_{i,j}=h_i$ if $i\in O_j$, and $q_{i,j}=0$, otherwise. 
By the constraints $\cup_{j=1}^m O_j=O$ and $O_j\cap O_{j'}=\emptyset$ for each $1\leq j<j'\leq m$, there is exactly one $j\in B$ such that $i\in O_j$.
Hence, the number of matches is $|\{(i,j)\in H\times C:q_{i,j}>0\}|=n$ and $\sum_{j\in C}q_{i,j}=h_i$ for every $i\in H$.
Since the capacity of bin $j\in B$ is not exceeded, we have that $\sum_{i\in O_j}s_i\leq K$, or equivalently $\sum_{i\in H}q_{i,j}\leq c_j$ for all $j\in C$.
Thus, the obtained solution is feasible.

To the other direction, consider a feasible solution for the minimum number of matches instance.
Obtain a feasible packing by placing object $i\in O$ in the bin $j$ if and only if $q_{i,j}>0$.
Since the solution contains at most $n$ matches and $h_i>0$, for each $i\in H$, each hot stream $i\in H$ matches with exactly one cold stream $j\in C$ and it holds that $q_{i,j}=h_i$.
That is, each object is placed in exactly one bin.
Given that $\sum_{i\in  H}q_{i,j}\leq c_j=K$, the bin capacity constraints are also satisfied.
\end{proof}

\section{Single Temperature Interval Problem}
\label{App:Single_Temperature_Interval}

Lemma \ref{lem:acyclic} concerns the structure of an optimal solution for the single temperature interval problem.
It shows that the corresponding graph is acyclic and that the number of matches is related to the number of graph's connected components (trees), if arc directions are ignored.

\begin{lemma}
\label{lem:acyclic}
Consider an instance $H$, $C$ of the single temperature interval problem. 
For each optimal solution $(\vec{y}^*,\vec{q}^*)$, there exists an integer $\ell^*\in[1,\min\{n,m\}]$ s.t. 
\begin{itemize}
    \item if arc directions are ignored, the corresponding graph $G(\vec{y}^*,\vec{q}^*)$ is a forest consisting of $\ell^*$ trees, i.e.\ there are no cycles, and
    \item $(\vec{y}^*,\vec{q}^*)$ contains $v^*=m+n-\ell^*$ matches.
\end{itemize}
\end{lemma}

\begin{proof}
Assume that $G(\vec{y}^*,\vec{q}^*)$ contains a cycle, after removing arc directions.
Moreover, let $M=\{(i_1,j_1),(i_2,j_1),(i_2,j_2),(i_3,j_2),\ldots, (i_g,j_{g-1}),(i_g,j_g),(i_1,j_g)\}$ be a subset of matches forming a cycle.
Denote by $q_{\min}^*=\min\{q_{i,j}^*:(i,j)\in M\}$ the minimum amount of heat transferred via a match in $M$. 
Without loss of generality, assume that $q_{i_1,j_1}^*=q_{\min}^*$.
Starting from $(\vec{y}^*,\vec{q}^*)$, produce a feasible solution $(\vec{y},\vec{q})$ as follows.
Set $q_{i_1,j_1}=0$, $q_{i_e,j_e}=q_{i_e,j_e}^*-q_{\min}^*$ and $q_{i_e,j_{e-1}}=q_{i_e,j_{e-1}}^*+q_{\min}^*$, for $e=2,\ldots,g$, as well as $q_{i_1,j_g}=q_{i_1,j_g}^*+q_{\min}^*$.
The new solution $(\vec{y},\vec{q})$ is feasible and has a strictly smaller number of matches compared to $(\vec{y}^*,\vec{q}^*)$, which is a contradiction. 

%

Since $G(\vec{y}^*,\vec{q}^*)$ does not contain a cycle, it must be a forest consisting of $\ell^*$ trees (which we call \emph{bins} from a packing perspective).
Let $B=\{1,\ldots,\ell^*\}$ be the set of these trees and $M_b$ the subset of matches in tree $b\in B$.
By definition, tree $b\in B$ contains $|M_b|$ matches (edges) and, therefore, $|M_b|+1$ streams (nodes).
Furthermore, each stream appears in exactly one tree implying that $\sum_{b=1}^{\ell^*}|M_b|=n+m-\ell^*$.
Thus, it holds that the number of matches in $(\vec{y}^*,\vec{q}^*)$ is equal to:
\begin{equation*}
v =\sum_{b=1}^{\ell^*}|M_b| = n+m-\ell^*.
\vspace{-24pt}
\end{equation*}
\end{proof}

Theorem \ref{thm:greedy} that Algorithm SG, developed by \citet{furman:2004}, is tight.

\begingroup
\def\thetheorem{\ref{thm:greedy}}
\begin{theorem}
Algorithm SG achieves an approximation ratio of 2 for the single temperature interval problem and it is tight.
\end{theorem}
\addtocounter{theorem}{-1}
\endgroup

\begin{proof}
In the algorithm's solution, the number $v$ of matches is equal to the number of steps that the algorithm performs.
For each pair of streams $i\in H$ and $j\in C$ matched by the algorithm, at least one has zero remaining heat load exactly after they have been matched.
Therefore, the number of steps is at most $v\leq n+m-1$. 
The optimal solution contains at least $v^*\geq\max\{n,m\}$.
Hence, the algorithm is 2-approximate.

Consider a set of $n$ hot streams with heat loads $h_i=2n+1-i$ for $1\leq i\leq n$ and $m=n+1$ cold streams with  $c_j=2n-j$, for $1\leq j\leq m$. 
As shown in Figure \ref{Fig:SimpleGreedyApp} for the special case $n=5$, the algorithm uses $2n$ matches while the optimal solution has $n+1$ matches. 
Hence, the 2 approximation ratio of Algorithm SG is asymptotically tight.
%
%
%
\end{proof}

\begin{figure}
\begin{subfigure}[t]{0.45\textwidth}
\centering
\includegraphics{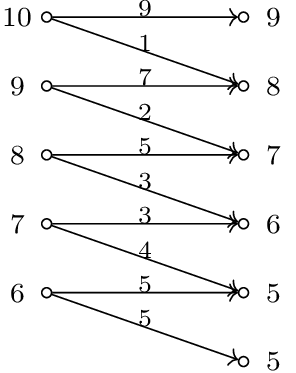}
\caption{Algorithm SG: 10}
\label{Fig:SimpleGreedyAlg}
\end{subfigure}
\begin{subfigure}[t]{0.45\textwidth}
\centering
\includegraphics{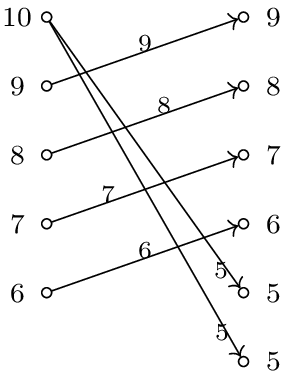}
\caption{Optimal solution: 6}
\label{Fig:SimpleGreedyOpt}
\end{subfigure}

\caption{An instance showing the tightness of the 2 performance guarantee for Algorithm SG.\label{Fig:SimpleGreedyApp}}
\end{figure}

Lemma \ref{Lem:EqualStreams} formalizes the benefit of matching stream pairs with equal heat loads and indicates the way of manipulating these matches in the analysis of Algorithm IG and the proof of Theorem \ref{thm:impgreedy}.

\begin{lemma}
\label{Lem:EqualStreams}
Consider an instance $(H,C)$ of the single temperature interval problem and suppose that there exists a pair of streams $i\in H$ and $j\in C$ such that $h_i=c_j$.
Then, 
\begin{itemize}
    \item there exists an optimal solution $(\vec{y}^*,\vec{q}^*)$ s.t. $q_{i,j}^*=h_i$, i.e.\ $i$ and $j$ are matched together,
    \item any $\rho$-approximate solution for $(H\setminus\{i\},C\setminus\{j\})$ is also $\rho$-approximate for $(H,C)$ with the addition of match $(i,j)$.
\end{itemize}
\end{lemma}

\begin{proof}
Consider an optimal solution $(\vec{y}^*,\vec{q}^*)$ in which $i$ and $j$ are not matched solely to each other.
Suppose that $i$ is matched with $j_1,j_2,\ldots,j_{m'}$ while $j$ is matched with $i_1,i_2,\ldots,i_{n'}$.
Without loss of generality, $q_{i,j}^*=0$; the case $0<q_{i,j}^*<h_i$ is treated similarly.
Starting from $(\vec{y}^*,\vec{q}^*)$, we obtain the slightly modified solution $(\vec{y},\vec{q})$ in which $i$ is matched only with $j$. 
The $c_j$ units of heat of $i_1,i_2,\ldots,i_{n'}$ originally transferred to $j$ are now exchanged with $j_1,j_2,\ldots,j_{m'}$, which are no longer matched with $i$.
The remaining solution is not modified.
Analogously to the proof of Theorem \ref{thm:greedy}, we show that there can be at most $n'+m'-1$ new matches between the $n'$ hot streams (i.e.\ $i_1,i_2,\ldots,i_{n'}$) and the $m'$ cold streams (i.e.\ $j_1,j_2,\ldots,j_{m'}$) in $(\vec{y},\vec{q})$.
By also taking into account the new match $(i,j)$, we conclude that there exists always a solution in which $i$ is only matched with $j$ and has no more matches than $(\vec{y}^*,\vec{q}^*)$.

%
%

Consider an optimal solution $(\vec{y}^*,\vec{q}^*)$ for $(H,C)$, in which there are $v^*$ matches and $i$ is matched only with $j$.
An optimal solution for $(H\setminus\{i\},C\setminus\{j\})$ contains $v^*-1$ matches.
Suppose that $(\vec{y},\vec{q})$ is the union of a $\rho$-approximate solution for $(H\setminus\{i\},C\setminus\{j\})$ and the match $(i,j)$.
Let $v$ be the number of matches in $(\vec{y},\vec{q})$.
Clearly, $v-1 \leq \rho\cdot(v^*-1)$ which implies that $v\leq \rho\cdot v^*$, as $\rho\geq 1$.
\end{proof}

The following theorem shows a tight analysis for Algorithm IG.

\begingroup
\def\thetheorem{\ref{thm:impgreedy}}
\begin{theorem}
Algorithm IG achieves an approximation ratio of 1.5 for the single temperature interval problem and it is tight.
\end{theorem}
\addtocounter{theorem}{-1}
\endgroup

\begin{proof}
By Theorem \ref{thm:greedy}, Algorithm IG produces a solution $(\vec{y},\vec{q})$ with $v\leq n+m$ matches.
Consider an optimal solution $(\vec{y}^*,\vec{q}^*)$.
By Lemma \ref{lem:acyclic}, $(\vec{y}^*,\vec{q}^*)$ consists of $\ell^*$ trees and has $v^* = n+m-\ell^*$ matches.
Due Lemma \ref{Lem:EqualStreams}, we may assume that instance does not contain a pair of equal hot and cold streams.
Hence, each tree in the optimal solution contains at least 3 streams, i.e.\ $\ell^*\leq (n+m)/3$. 
Thus, $v^*\geq (2/3)(n+m)$ and we conclude that $v\leq(3/2)v^*$.


For the tightness of our analysis, consider an instance of the problem with $n$ hot streams, where $h_i=4n-2i$ for $i=1,\ldots,n$, and $m=2n$ cold streams such that $c_j=4n-2j-1$ for $j=1,\ldots,n$ and $c_j=1$ for $j=n+1,\ldots,2n$.
Algorithm IG uses $3n$ matches, while the optimal solution uses $2n$ matches.
Hence the $3/2$ approximation ratio of the algorithm is tight.
Figures \ref{Fig:ImprovedGreedyAlg} and \ref{Fig:ImprovedGreedyOpt} show the special case with $n=4$.
\end{proof}

\begin{figure}
\begin{subfigure}[t]{0.45\textwidth}
\centering
\includegraphics{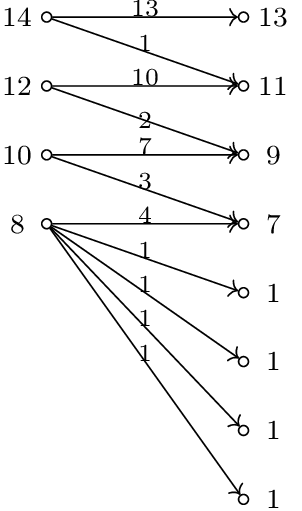}
\caption{Algorithm IG: 11}
\label{Fig:ImprovedGreedyAlg}
\end{subfigure}
\begin{subfigure}[t]{0.45\textwidth}
\centering
\includegraphics{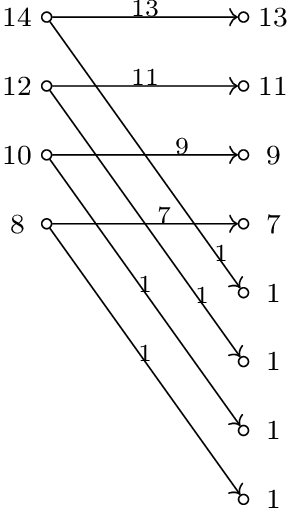}
\caption{Optimal solution: 8}
\label{Fig:ImprovedGreedyOpt}
\end{subfigure}

\caption{An instance showing the tightness of the 1.5 performance guarantee for Algorithm IG.}
\end{figure}

\section{Relaxation Rounding Heuristics}
\label{App:Relaxation_Rounding}

\begingroup
\def\thetheorem{\ref{Thm:FLPR_negative}}
\begin{theorem}
Algorithm FLPR is $\Omega(n)$-approximate.
\end{theorem}
\addtocounter{theorem}{-1}
\endgroup

\begin{proof}
We construct a minimum number of matches instance for which Algorithm FLPR produces a solution $\Omega(n)$ times far from the optimal solution.
This instance consists of a single temperature interval and an equal number of hot and cold streams, i.e.\ $n=m$, with the same heat load $h_i=n$ and $c_j=n$, for each $i\in H$ and $j\in C$.
Because of the single temperature interval, we ignore the temperature interval indices of the variables $\vec{q}$.
In the optimal solution, each hot stream is matched with exactly one cold stream and there are $v^*=n$ matches in total.
Given that there exist feasible solutions such that $q_{i,j}=n$, for every possible $i\in H$ and $j\in C$, the algorithm computes the upper bound $U_{i,j}=n$.
In an optimal fractional solution, it holds that $q_{i,j}^f=1$, for each $i\in H$ and $j\in C$.
In this case, Algorithm FLPR sets $y_{i,j}=1$ for each pair of streams $i\in H$, $j\in C$ and uses a total number of matches equal to $v = \sum_{i\in H} \sum_{j\in C} y_{i,j} = \Omega(n^2)$.
Therefore, it is $\Omega(n)$-approximate.
\end{proof}

\begingroup
\def\thetheorem{\ref{Thm:FLPR_positive}}
\begin{theorem}
Given an optimal fractional solution with a set $M$ of matches and filling ratio $\phi(M)$, FLPR produces a $\frac{1}{\phi(M)}$-approximate integral solution.
\end{theorem}
\addtocounter{theorem}{-1}
\endgroup

\begin{proof}
We denote Algorithm FLPR's solution and the optimal fractional solution by $(\vec{y},\vec{q})$ and $(\vec{y}^f,\vec{q}^f)$, respectively. 
Moreover, suppose that $(\vec{y}^*,\vec{q}^*)$ is an optimal integral solution.
Let $M\subseteq H\times C$ be the set of matched pairs of streams by the algorithm, i.e.\ $y_{i,j}=1$, if $(i,j)\in M$, and $y_{i,j}=0$, otherwise. 
Then, it holds that:
\begin{align*}
\sum_{(i,j)\in M} y_{i,j}
& = \sum_{(i,j)\in M} \frac{U_{i,j}}{L_{i,j}} \sum_{s,t\in T}\frac{q_{i,s,j,t}}{U_{i,j}} \\
&\leq \frac{1}{\phi(M)} \sum_{(i,j)\in M}\sum_{s,t\in T} \frac{q_{i,s,j,t}^f}{U_{i,j}} \\
& \leq \frac{1}{\phi(M)} \sum_{(i,j)\in M} y_{i,j}^f \\
&\leq \frac{1}{\phi(M)} \sum_{i\in H}\sum_{j\in C} y_{i,j}^*.
\end{align*} 
The first equality is obtained by using the fact that, for each $(i,j)\in M$, it holds that $y_{i,j}=\frac{U_{i,j}}{L_{i,j}}\frac{L_{i,j}}{U_{i,j}}$ and $L_{i,j}=\sum_{s,t\in T} q_{i,s,j,t}$.
The first inequality is true by the definition of the filling ratio $\phi(M)$ and the fact that $\vec{q} = \vec{q}^f$.
The second inequality holds by the big-M constraint of the fractional relaxation.
The final inequality is valid due to the fact that the optimal fractional solution is a lower bound on the optimal integral solution. 
\end{proof}

\section{Water Filling Heuristics}
\label{App:Water_Filling}

The reformulated MILP in Eqs.\ (\ref{Eq:SingleMIPwithoutConservation_MaxBins})-(\ref{Eq:SingleMIPwithoutConservation_Integrality}) solves the single temperature interval problem without heat conservation.
It is similar to the MILP in Eqs.\ (\ref{Eq:SingleMIP_MaxBins})-(\ref{Eq:SingleMIP_Integrality}) with heat conservation, except that it does not contain constraints (\ref{Eq:SingleMIP_HotBinUsage}) while Equalities (\ref{Eq:SingleMIP_HotAssignment}) and (\ref{Eq:SingleMIP_BinHeatConservation}) become the inequalities (\ref{Eq:SingleMIPwithoutConservation_HotAssignment}) and (\ref{Eq:SingleMIPwithoutConservation_BinHeatConservation}).
In the single temperature interval problem with (without) heat conservation, the total heat of hot streams is equal to (greater than or equal to) the demand of the cold streams.
Each water filling algorithm step solves the single temperature interval problem without heat conservation.
All heat demands of cold streams are satisfied and the excess heat supply of hot streams descends to the subsequent temperature interval.
\begin{align}
\text{max} & \sum_{b\in B} x_b \label{Eq:SingleMIPwithoutConservation_MaxBins}  \\
& x_b \geq \sum_{j\in C} z_{j,b} & b\in B \label{Eq:SingleMIPwithoutConservation_ColdBinUsage} \\ 
& \sum_{b\in B} w_{i,b} \leq 1 & i\in H \label{Eq:SingleMIPwithoutConservation_HotAssignment} \\
& \sum_{b\in B} z_{j,b} = 1 & j\in C \label{Eq:SingleMIPwithoutConservation_ColdAssignment} \\
& \sum_{i\in H} w_{i,b} \cdot h_i \geq \sum_{j\in C} z_{j,b} \cdot c_j & b\in B \label{Eq:SingleMIPwithoutConservation_BinHeatConservation} \\
& x_b, w_{i,b}, z_{j,b}\in\{0,1\} & b\in B, i\in H, j\in C \label{Eq:SingleMIPwithoutConservation_Integrality}
\end{align}


Theorem \ref{Thm:WaterFillingRatio} shows an asymptotically tight performance guarantee for water filling heuristics proportional to the number of temperature intervals.
The positive performance guarantee implies the proof of \citet{furman:2004}. 

\begingroup
\def\thetheorem{\ref{Thm:WaterFillingRatio}}
\begin{theorem}
Algorithms WFG and WFM are $\Theta(k)$-approximate (i.e.\ both $O(k)$-approximate and $\Omega(k)$-approximate).
\end{theorem}
\addtocounter{theorem}{-1}
\endgroup

\begin{proof}
A water filling algorithm solves an instance of the single temperature interval problem in each temperature interval $t=1,\ldots,k$.
This instance consists of at most $n$ hot streams and at most $m$ cold streams.
By Theorem \ref{thm:greedy}, algorithms WFG and WFM introduce at most $n+m$ new matches in each temperature interval and produce a solution with $v\leq k(n+m)$ matches.
In the optimal solution, each hot and cold stream appears in at least one match which means that $v^*\geq\max\{n,m\}$ matches are chosen in total.
So, $v\leq 2k\cdot v^*$.

On the negative side, we show a lower bound on the performance guarantee of algorithms WFG and WFM using the extension of the problem instance in Figure \ref{Fig:Greedy_Packing_Tightness} with an equal number of hot streams, cold streams and temperature intervals, i.e.\ $m=n=k$.
Each hot stream $i\in H$ has heat supply $\sigma_{i,s}\in\{0,1\}$ and each cold stream $j\in C$ has heat demand $\delta_{j,t}\in\{0,1\}$, for each $s,t\in T$. 
Hot stream $i$ has unit heat in temperature intervals $\{1,\ldots,k-i+1\}$ and no supply elsewhere.
Cold stream $j$ demands unit heat in temperature intervals $\{j,\ldots,k\}$ and no demand elsewhere.
In the optimal solution, hot stream $i$ is matched with cold stream $j=i$ and there are $v^*=k$ matches in total.
Algorithms WFG and WFM produce the same solution in which hot stream $i$ is matched with cold streams $\{1,2,\ldots,k-i+1\}$, where $j=i$, and there are $v=O(k^2)$ matches in total.
\end{proof}

\begin{figure}

\begin{center}
\includegraphics{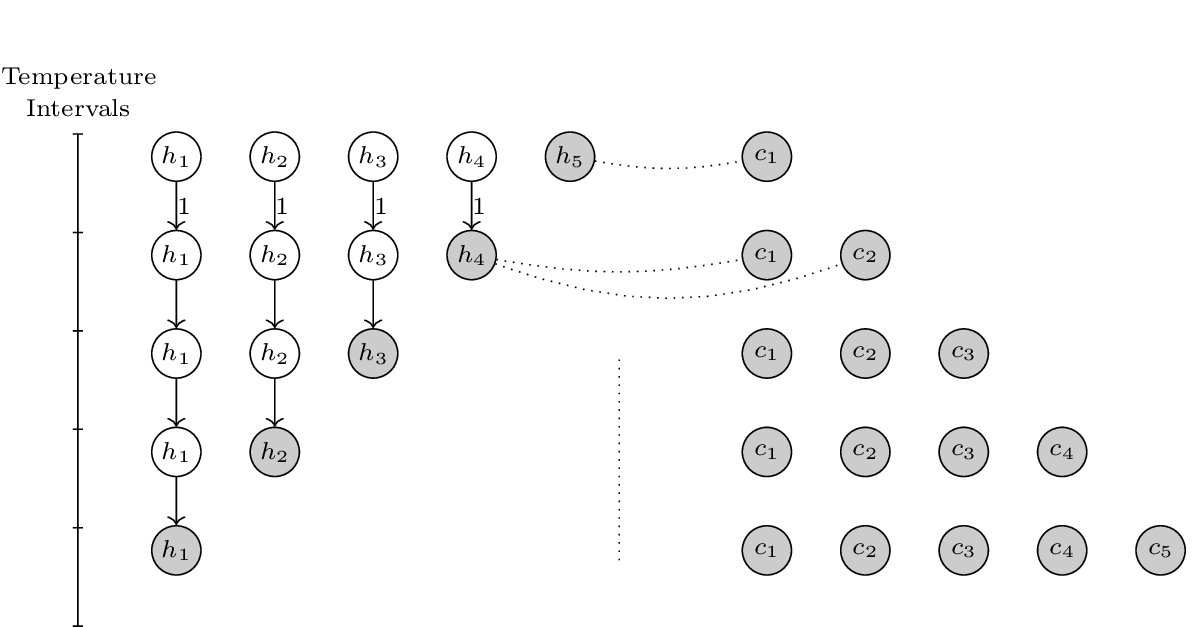}
\caption{An instance showing (asymptotically) the tightness of the $O(k)$ performance guarantee for greedy packing heuristics. In this instance, it holds that $n=m=k$ and every heat supply and heat demand belongs to $\{0,1\}$ in each temperature interval. In the optimal solution, hot stream $i$ is matched with cold stream $j=i$ and there are $n$ matches. In the algorithm's solution, hot stream $i$ is matched with cold streams $1,\ldots,m-i + 1$ and there are $\Omega(n^2)$ matches in total.}
\label{Fig:Greedy_Packing_Tightness}
\end{center}

\end{figure}

\section{Greedy Packing Heuristics}
\label{App:Greedy_Packing}

Lemma \ref{Lem:MonotonicDecomposition} shows a condition ensuring the strict monotonicity of a greedy heuristic which decomposes any instance $I$ into the instances $I^A$ (already solved) and $I^B$ (remaining to be solved) in each iteration (see Section \ref{Subsection:MonotonicGreedyHeuristics}).

\begin{lemma}
\label{Lem:MonotonicDecomposition}
A greedy heuristic is strictly monotonic if $I^B$ is feasible in each iteration.
\end{lemma}
\begin{proof}
Given that $I^A$ is of maximal heat (see Section \ref{Subsection:MonotonicGreedyHeuristics}), any match of $M$ is redundant in any feasible solution of $I^B$.
Since $I^B$ is feasible, there exists a match in $H\times C\setminus M$ whose selection increases the amount of heat exchanged in $I^A$.  
\end{proof}

Lemma \ref{Lem:FeasibilityConditions} states necessary and sufficient conditions for the feasibility of a minimum number of matches instance $I$. 
The first set of conditions ensures that heat always flows from the hot side to the same or lower temperature intervals on the cold side.
The last condition enforces heat conservation. 

\begin{lemma}
\label{Lem:FeasibilityConditions}
An instance $I$ of the minimum number of matches is feasible if and only if the following conditions hold.
\begin{itemize}
\item For each $u\in T\setminus\{k\}$, it is the case that $R_u\geq 0$, or equivalently 
\begin{equation*}
\sum_{i\in H}\sum_{s=1}^u\sigma_{i,s} \geq \sum_{j\in C}\sum_{t=1}^u\delta_{j,t}
\end{equation*} 
\item It holds that $R_k=0$, or equivalently 
\begin{equation*}
\sum_{i\in H}\sum_{s=1}^k\sigma_{i,s} = \sum_{j\in C}\sum_{t=1}^k\delta_{j,t}.
\end{equation*}
\end{itemize}
\end{lemma}
\begin{proof}
To the first direction, a violation of a condition makes the task of constructing a feasible solution impossible.
To the opposite direction, Algorithm MHG in Section \ref{Alg:MaximumHeat} constructs a feasible solution for every instance satisfying the conditions;
it suffices to consider all the hot and cold streams as one large hot and large cold stream, respectively.
The single hot stream has heat load $\sum_{i\in H}\sigma_{i,s}$ in temperature interval $s\in T$ and the single cold stream has heat load $\sum_{j,t}\delta_{j,t}$ in temperature interval $t\in T$.
\end{proof}

Given a decomposition of an instance $I$ into instances $I^A$ and $I^B$, Lemma \ref{Lem:FeasibleDecomposition} shows that a careful construction of $I^A$ respecting the proposed heat residual capacities in Section \ref{Subsection:MonotonicGreedyHeuristics} implies that $I^B$ is also feasible.

\begin{lemma}
\label{Lem:FeasibleDecomposition}
Consider a decomposition of a feasible instance $I$ into the instances $I^A$ and $I^B$. 
Let $R$, $\vec{R}^A$ and $\vec{R}^B$ be the corresponding heat residual capacities.
If $I^A$ is feasible and it holds that $R_u^A\leq R_u$ for each $u\in T$, then $I^B$ is also feasible.
\end{lemma}
\begin{proof}
To show that the Lemma is true, it suffices to show that $I^B$ satisfies the feasibility conditions of Lemma \ref{Lem:FeasibilityConditions}.  
Consider a temperature interval $u\in T\setminus\{k\}$.
Then,
\begin{align*}
R_u^A \leq R_u 
& \Leftrightarrow \sum_{i\in H} \sum_{s=1}^u \sigma_{i,s}^A - \sum_{j\in C} \sum_{t=1}^u \delta_{j,t}^A \leq \sum_{i\in H} \sum_{s=1}^u \sigma_{i,s} - \sum_{j\in C} \sum_{t=1}^u \delta_{j,t} \\
& \Leftrightarrow \sum_{i\in H} \sum_{s=1}^u \left( \sigma_{i,s} - \sigma_{i,s}^A \right) \geq \sum_{j\in C} \sum_{t=1}^u \left( \delta_{j,t} - \delta_{j,t}^A \right) \\
& \Leftrightarrow \sum_{i\in H} \sum_{s=1}^u \sigma_{i,s}^B \geq \sum_{j\in C} \sum_{t=1}^u \delta_{j,t}^B
\end{align*}
In the same fashion, the fact that $R_k=R_k^A=0$ implies that $R_k^B=0$.
Hence, $I^B$ is feasible.
\end{proof}

Given a set $M$ matches, the LP in Eqs.\ (\ref{Eq:MaxFractionLP_Objective})-(\ref{Eq:MaxFractionLP_Positiveness}) maximizes the total stream fraction that can be covered using only matches in $M$.
It is similar to the LP in Eqs.\ (\ref{Eq:MaxHeatLP_Objective})-(\ref{Eq:MaxHeatLP_Positiveness}) in Section \ref{Subsection:Largest_Heat_Match}, except that 
the maximum fraction objective function (\ref{Eq:MaxFractionLP_Objective}) replaces
the maximum heat objective function (\ref{Eq:MaxHeatLP_Objective}).

\begin{align}
\max & \sum_{(i,s,j,t) \in A(M)} \left( \frac{q_{i,s,j,t}}{h_i} + \frac{q_{i,s,j,t}}{c_j} \right) \label{Eq:MaxFractionLP_Objective} \\ 
& \sum_{(j,t)\in V_{i,s}^C(M)} q_{i,s,j,t} \leq \sigma_{i,s} & (i,s)\in V^H(M) \label{Eq:MaxFractionLP_HotStreamSupply}\\
& \sum_{(i,s)\in V_{j,t}^H(M)} q_{i,s,j,t} \leq \delta_{j,t} & (j,t)\in V^C(M) \label{Eq:MaxFractionLP_ColdStreamDemand}\\
& \sum_{(i,s,j,t)\in A_u(M)} q_{i,s,j,t}\leq R_u & u\in T \label{Eq:MaxFractionLP_ResidualCapacity}\\
& q_{i,s,j,t}\geq 0 & (i,s,j,t) \in A(M) \label{Eq:MaxFractionLP_Positiveness}
\end{align}

\medskip

The following theorem shows a performance guarantee for Algorithm LHM-LP using a standard packing argument.

\begingroup
\def\thetheorem{\ref{Thm:GreedyPackingRatio}}
\begin{theorem}
Algorithm LHM-LP is $O(\log n + \log \frac{h_{\max}}{\epsilon})$-approximate, where $\epsilon$ is the required precision.
\end{theorem}
\addtocounter{theorem}{-1}
\endgroup

\begin{proof}
Initialy, we show an approximation ratio of $O(\log n + \log h_{\max})$ for the special case of the problem with integer parameters.
Then, we generalize the result to decimal parameters.

We denote by $v$ the number of the algorithm's matches and by $v^*$ the number of matches in an optimal solution.
To upper bound $v$, we assign unitary costs to the transferred heat in the algorithm's solution as follows.
Algorithm LHM-LP constructs a feasible set $M$ of matches greedily.
At each iteration, LHM-LP selects a match whose addition in $M$ maximizes the heat that can be feasibly exchanged using the matches in $M$.
For $\ell=1,\ldots,v$, let $M_{\ell}$ be the set of selected matches at the end of the $\ell$-th iteration and $Q_{\ell}$ be the maximum amount of heat the can be feasibly exchanged between all streams using exactly the matches in $M_{\ell}$.
Before the algorithm begins, $M_0=\emptyset$ and $Q_0=0$.
The extra amount of transferable heat with the addition of the $\ell$-th chosen match is $E_{\ell}=Q_{\ell}-Q_{\ell-1}$, for $\ell=1,\ldots,v$.
We set the unitary cost to this part of the algorithm's total heat as $\kappa_{\ell}=\frac{1}{E_{\ell}}$.
Then, the algorithm's number of matches can be expressed:
\begin{equation}
\label{Eq:FractionalCost}
v = \sum_{\ell=1}^v \kappa_{\ell} \cdot E_{\ell}.
\end{equation}
Let $S_{\ell}$ be the total remaining heat to be transferred when the $\ell$-th iteration completes.
Then, $S_0 = Q$ and $S_{\ell}= Q - Q_{\ell}$, for $\ell=1,\ldots,v$, where $Q=\sum_{i=1}^nh_i$ is the total amount of heat.
Note that $S_v = 0$ because the algorithm produces a feasible solution.
Since the algorithm chooses the match that results in the highest increase of transferred heat in each iteration, it must be the case that $E_1\geq \ldots \geq E_v$ or equivalently $\kappa_1 \leq \ldots \leq \kappa_v$. 
At the end of the $\ell$-th iteration, the remaining heat can be transferred using at most $v^*$ additional matches by selecting the remaining matches of an optimal solution.
Using a simple average argument we get that $\kappa_{\ell} \leq \frac{v^*}{S_{\ell-1}}$, for each $\ell=1,\ldots,v$. 
Thus, Eq.\ (\ref{Eq:FractionalCost}) implies:
\begin{equation}
\label{Eq:GreedyPackingRatio}
v \leq \sum_{\ell=1}^v \left(\frac{v^*}{S_{\ell-1}}\right)\cdot E_{\ell} = \sum_{\ell=1}^v \left(\frac{E_{\ell}}{Q - Q_{\ell-1}}\right)\cdot v^*.
\end{equation}
By the integrality of the minimum cost network flow polytope, each value $E_{\ell}$ is an integer, for $\ell=1,\ldots,v$.
Hence,
\begin{equation*}
\frac{E_{\ell}}{Q-Q_{\ell-1}} = \sum_{e=1}^{E_{\ell}} \frac{1}{Q - Q_{\ell-1} } \leq \sum_{e=1}^{E_{\ell}} \frac{1}{Q - Q_{\ell-1} - e + 1 }.
\end{equation*}
Given that $Q_{\ell} = Q_{\ell-1} + E_{\ell}$,
\begin{equation}
\label{Eq:GreedyPackingHarmonic}
\frac{E_{\ell}}{Q-Q_{\ell-1}} \leq \sum_{e=Q_{\ell-1}}^{Q_{\ell} - 1} \frac{1}{Q-e}.
\end{equation}
Inequalities (\ref{Eq:GreedyPackingRatio}) and (\ref{Eq:GreedyPackingHarmonic}) imply:
\begin{equation*}
v \leq \left( \sum_{e=1}^Q \frac{1}{e} \right) \cdot v^*.
\end{equation*}
Using the asymptotic bound $\sum_{e=1}^Q \frac{1}{e}=O(\log Q)$ of harmonic series and the fact that $Q\leq n\cdot h_{max}$, we conclude that the algorithm is $O(\log n + \log h_{\max})$-approximate, where $h_{\max} = \max_{i\in H}\{h_i\}$ is the maximum heat of a hot stream.

Generalizing to decimal parameters, the algorithm is $O(\log n + \log\frac{h_{\max}}{\epsilon})$, where $\epsilon$ is the precision required for solving the problem instance.
The reasoning is the same except that, instead of considering integer units, we consider $\epsilon$ units to extend inequality (\ref{Eq:GreedyPackingHarmonic}).
\end{proof}

\section{Minimum Utility Cost Problem}
\label{App:Minimum_Utility_Cost}

This section shows how to obtain a minimum number of matches problem instance from a general heat exchanger network design problem instance via minimizing utility cost. We include this appendix for completeness, but this material is available elsewhere \citep{floudas}.
Table \ref{tbl:notation_min_utility} lists the notation.

\paragraph{General heat exchanger network design}
An instance of the general heat exchanger network design consists of a set $HS=\{1,2,\ldots,ns\}$ of hot streams, a set $CS=\{1,2,\ldots,ms\}$ of cold streams, a set $HU=\{1,2,\ldots,nu\}$ of hot utilities and a set $CU=\{1,2,\ldots,mu\}$ of cold utilities. 
Each hot stream $i \in HS$ (cold stream $j\in CS$) has initial inlet, outlet temperatures $T_{\text{in},i}^{HS}$, $T_{\text{out},i}^{HS}$ (resp. $T_{\text{in},j}^{CS}$, $T_{\text{out},j}^{CS}$) and flowrate heat capacity $FCp_{i}$ (resp. $FCp_{j}$).
Each hot utility $i\in HU$ (cold utility $j\in CU$) is associated with inlet, outlet temperatures $T_{\text{in},i}^{HU}$, $T_{\text{out},i}^{HU}$ (resp. $T_{\text{in},j}^{CU}$, $T_{\text{out},j}^{CU}$) and a cost $\kappa_{i}^{HU}$ (resp. $\kappa_j^{CU}$).

\singlespacing
\begin{longtable}{l l l}
\caption{Minimum Utility Cost Notation}\\
\toprule
Name & Description \\
\midrule
\multicolumn{2}{l}{\bf Cardinalities, Indices, Sets} \\
$ns,ms$ & Number of hot, cold streams \\
$nu,mu$ & Number of hot, cold utilities \\
$k$ & Number of temperature intervals \\
$i\in HS\cup HU$ & Hot stream, utility \\
$j\in CS\cup CU$ & Cold stream, utility \\
$t\in TI$ & Temperature interval \\
$HS,CS$ & Set of hot, cold streams \\
$HU,CU$ & Set of hot, cold utilities \\
$TI$ & Set of temperature intervals \\
\midrule
{\bf Parameters} & \\
$FCp_i,FCp_j$ & Flowrate heat capacity of hot stream $i$, cold stream $j$ \\
$T_{\text{in},i}^{HS},T_{\text{out},i}^{HS}$ & Inlet, outlet temperature of hot stream $i$ \\
$T_{\text{in},j}^{CS},T_{\text{out},j}^{CS}$ & Inlet, outlet temperature of cold stream $j$ \\
$T_{\text{in},i}^{HU},T_{\text{out},i}^{HU}$ & Inlet, outlet temperature of hot utility $i$ \\
$T_{\text{in},j}^{CU},T_{\text{out},j}^{CU}$ & Inlet, outlet temperature of cold utility $j$ \\
$\Delta T_{\min}$ & Minimum heat recovery approach temperature \\
$\kappa_i^{HU}, \kappa_j^{CU}$ & Unitary cost of hot utility $i$, cold utility $j$ \\
$\sigma_{i,t}^{HS}$ & Heat supply of hot stream $i$ in interval $t$ \\
$\delta_{j,t}^{CS}$ & Heat demand of cold stream $j$ in interval $t$ \\
\midrule
{\bf Variables} & \\
$\sigma_{i,t}^{HU}$ & Heat supply of hot utility $i$ in interval $t$ \\
$\delta_{j,t}^{CU}$ & Heat demand of cold utility $j$ in interval $t$ \\
$R_{t}$ & Residual heat exiting temperature interval $t$ \\
\bottomrule
\label{tbl:notation_min_utility}
\end{longtable}
\onehalfspacing

\paragraph{Temperature intervals}
The sequential method begins by computing a set $TI=\{1,2,\ldots,k\}$ of $k$ temperature intervals \citep{linnhoff:1978, ciric:1989}. 
A minimum heat recovery approach temperature $\Delta T_{\min}$ specifies the minimum temperature difference between two streams exchanging heat. 
In order to incorporate $\Delta T_{\min}$ in the problem's setting, we enforce that each temperature interval corresponds to a temperature range on the hot stream side shifted up by $\Delta T_{\min}$ with respect to to its corresponding temperature range on the cold stream side.
Let $TI^H$ and $TI^C$ be the temperature intervals on the hot and cold side, respectively.
Consider, on the hot side, all $k+1$ discrete temperature values $T_1>T_2>\ldots>T_{k+1}$ belonging to the set $\{T_{\text{in},i}^{HS}:i\in HS\}\cup\{T_{\text{in},i}^{HU}:i\in HU\}\cup\{T_{\text{in},j}^{CS}+\Delta T_{\min}:j\in CS\}\cup\{T_{\text{in},j}^{CU}+\Delta T_{\min}:j\in CU\}$. 
Then, we define $TI^H=\cup_{t=1}^k\{[T_t,T_{t+1}]\}$ and $TI^C=\cup_{t=1}^k\{[T_t-\Delta T_{\min},T_{t+1}-\Delta T_{\min}]\}$.
We set $TI=TI^H$ and we observe that $TI^C$ contains exactly the same temperature intervals with $TI$ shifted by $\Delta T_{\min}$.
Moreover, we set $\Delta T_t=T_t-T_{t+1}$, for $t\in TI$.

For each temperature interval $t\in TI$, we are now able to compute the quantity $\sigma_{i,t}^{HS}$ of heat load exported by hot stream $i\in HS$ as well as the amount $\delta_{j,t}^{CS}$ of heat load received by cold stream $j\in CS$ in $t\in TI$.
Specifically, for each $i\in HS$ and $t\in TI$, we set 
\begin{equation*}
\sigma_{i,t}^{HS}=
\left\{
	\begin{array}{ll}
		FCp_i\cdot\Delta T_t,  & \mbox{if $T_i^{\text{in}}\geq T_t$ and $T_i^{\text{out}}\leq T_{t+1}$} \\
		FCp_i\cdot(T_t-T_i^{\text{out}}), & \mbox{if $T_i^{\text{in}}\geq T_t$ and $T_i^{\text{out}}>T_{t+1}$} \\
		0, & \mbox{if $T_i^{\text{in}}<T_t$}
	\end{array}
\right.
\end{equation*}
Similarly, for each $j\in CS$ and $t\in TI$,
\begin{equation*}
\delta_{j,t}^{CS}=
\left\{
	\begin{array}{ll}
		FCp_j\cdot\Delta T_t,  & \mbox{if $T_j^{\text{in}}\leq T_{t+1}-\Delta T_{\min}$ and $T_j^{\text{out}}\geq T_t-\Delta T_{\min}$} \\
		FCp_j\cdot(T_j^{\text{out}}-(T_{t+1}-\Delta T_{\min})), & \mbox{if $T_j^{\text{in}}\leq T_{t+1}-\Delta T_{\min}$ and $T_j^{\text{out}}< T_t-\Delta T_{\min}$} \\
		0, & \mbox{if $T_j^{\text{in}}>T_{t+1}-\Delta T_{\min}$}
	\end{array}
\right.
\end{equation*}

\paragraph{Minimum utility cost} 

This problem is solved in order to compute the minimum amount of utility heat so that there is heat balance in the network. 
For each hot utility $i\in HU$ and cold utility $j\in CU$ the continuous variables $\sigma_{i,t}^{HU}$ and $\delta_{j,t}^{CU}$ correspond to the amount of heat of $i$ and $j$, respectively, in temperature interval $t$. 
The LP uses a heat residual variable $R_t$, for each $t\in TI$. 
Let $TI_i$ be the set of temperature intervals to which hot utility $i\in HU$ can transfer heat, feasibly. 
Similarly, let $TI_j$ be the set of temperature intervals from which cold utility $j\in CU$ can receive heat.
The minimum utility cost problem can be solved by using the following LP model (see \citet{cerdaetwesterberg:1983, papoulias:1983}).


{\small

\begin{alignat}{2}
\min & \sum_{i\in HU}\sum_{t\in TI} \kappa_i^{HU}\cdot \sigma_{i,t}^{HU} + \sum_{j\in CU}\sum_{t\in TI} \kappa_j^{CU}\cdot \delta_{j,t}^{CU} \label{minut:ob} \\
& \sum_{i\in HS} \sigma_{i,t}^{HS} + \sum_{i\in HU} \sigma_{i,t}^{HU} +R_t = \sum_{j\in CS} \delta_{j,t}^{CS} + \sum_{j\in CU} \delta_{j,t}^{CU} + R_{t+1} \;\;\;\; && t\in TI \label{minut:eb} \\
& R_1, R_{k+1} = 0 && \label{minut:res} \\
& \sigma_{i,t}^{HU} = 0 && i\in HU, t\in TI\setminus TI_i \label{minut:hot} \\
& \delta_{j,t}^{CU} = 0 && j\in CU, t\in TI\setminus TI_j \label{minut:cold} \\
& \sigma_{i,t}^{HU}, \delta_{j,t}^{CU}, R_t \geq 0 && i\in HU, j\in CU, t\in TI
\end{alignat}

}

Expression \ref{minut:ob} minimizes the utility cost. 
Constraints \ref{minut:eb} and \ref{minut:res} ensure energy balance. 
Constraints \ref{minut:hot} and \ref{minut:cold} ensure that heat flows from a temperature interval to the same or a lower temperature interval. 

\paragraph{Minimum number of matches} 

Given an optimal solution of the minimum utility cost problem, we obtain an instance of the minimum number of matches problem as follows.
All utilities are considered as streams, i.e.\  $H=HS\cup HU$, $C=CS\cup CU$, $n=ns+nu$ and $m=ms+mu$. 
Furthermore, $T=TI$.
Finally, for each $i\in H$ and $t\in T$ the parameter $\sigma_{i,t}$ is equal to $\sigma_{i,t}^{HS}$ or $\sigma_{i,t}^{HU}$ depending on whether $i$ was originally a hot stream or utility.
The parameters $\delta_{j,t}$ are obtained similarly, for each $j\in C$ and $t\in T$.



\newpage

\section{Experimental Results}
\label{App:Experimental_Results}

\begin{table}[ht!] 
\scriptsize 
\begin{adjustbox}{center} 
\begin{tabular}{|lrrrrrrr|}
\hline 
\multirow{2}{*}{\textbf{Test Case}} & \multirow{2}{1cm}{\textbf{Hot Streams}} & \multirow{2}{1cm}{\textbf{Cold Streams}} & \multirow{2}{1.3cm}{\textbf{Temp.\ Intervals}} & \multirow{2}{1.1cm}{\textbf{Binary Vars}} & \multirow{2}{1.4cm}{\textbf{Continuous Vars}} & \multirow{2}{*}{{\textbf{Constraints}}} & \multirow{2}{*} {\textbf{Ref}} \\ 
& & & & & & & \\
\hline 
\multicolumn{8}{|l|} {\textbf{\cite{furman:2004} Test Set}} \\ 
\texttt{10sp-la1 } & 5 & 6 & 9 & 30 & 315 & 134 & \cite{linhoff:1989} \\ 
\texttt{10sp-ol1 } & 5 & 7 & 8 & 35 & 320 & 136 & \cite{shenoy:1995} \\ 
\texttt{10sp1    } & 5 & 6 & 9 & 30 & 315 & 134 & \cite{pho:1973} \\ 
\texttt{12sp1    } & 10 & 3 & 13 & 30 & 520 & 209 & \cite{sargent:1978} \\ 
\texttt{14sp1    } & 7 & 8 & 14 & 56 & 882 & 273 & \cite{sargent:1978} \\ 
\texttt{15sp-tkm } & 10 & 7 & 15 & 70 & 1200 & 335 & \cite{kokossis:2000} \\ 
\texttt{20sp1    } & 10 & 11 & 20 & 110 & 2400 & 540 & \cite{sargent:1978} \\ 
\texttt{22sp-ph  } & 12 & 12 & 18 & 144 & 2808 & 588 & \cite{polley:1999} \\ 
\texttt{22sp1    } & 12 & 12 & 17 & 144 & 2652 & 564 & \cite{miguel:1998} \\ 
\texttt{23sp1    } & 11 & 13 & 19 & 143 & 2926 & 610 & \cite{mocsny:1984} \\ 
\texttt{28sp-as1 } & 17 & 13 & 15 & 221 & 3570 & 688 & \cite{smith:1989} \\ 
\texttt{37sp-yfyv} & 21 & 17 & 32 & 357 & 12096 & 1594 & \cite{yu:2000} \\ 
\texttt{4sp1     } & 3 & 3 & 5 & 9 & 60 & 42 & \cite{lee:1970} \\ 
\texttt{6sp-cf1  } & 3 & 4 & 5 & 12 & 75 & 50 & \cite{ciric:1989} \\ 
\texttt{6sp-gg1  } & 3 & 3 & 5 & 9 & 60 & 42 & \cite{gundersen:1990} \\ 
\texttt{6sp1     } & 3 & 4 & 6 & 12 & 90 & 57 & \cite{lee:1970} \\ 
\texttt{7sp-cm1  } & 4 & 5 & 8 & 20 & 192 & 96 & \cite{colberg:1990} \\ 
\texttt{7sp-s1   } & 7 & 2 & 8 & 14 & 168 & 93 & \cite{shenoy:1995} \\ 
\texttt{7sp-torw1} & 5 & 4 & 7 & 20 & 175 & 88 & \cite{trivedi:1990} \\ 
\texttt{7sp1     } & 3 & 5 & 6 & 15 & 108 & 66 & \cite{masso:1969} \\ 
\texttt{7sp2     } & 4 & 4 & 7 & 16 & 140 & 76 & \cite{masso:1969} \\ 
\texttt{7sp4     } & 7 & 2 & 8 & 14 & 168 & 93 & \cite{dolan:1990} \\ 
\texttt{8sp-fs1  } & 6 & 4 & 8 & 24 & 240 & 110 & \cite{farhanieh:1990} \\ 
\texttt{8sp1     } & 5 & 5 & 8 & 25 & 240 & 110 & \cite{sargent:1978} \\ 
\texttt{9sp-al1  } & 5 & 6 & 9 & 30 & 315 & 134 & \cite{ahmad:1989} \\ 
\texttt{9sp-has1 } & 6 & 5 & 9 & 30 & 324 & 135 & \cite{hall:1990} \\ 
\hline 
\multicolumn{8}{|l|} {\textbf{\cite{minlp,chen:2015} Test Set}} \\ 
\texttt{balanced10  } & 12 & 11 & 20 & 132 & 2880 & 604 &  \\ 
\texttt{balanced12  } & 14 & 13 & 23 & 182 & 4508 & 817 &  \\ 
\texttt{balanced15  } & 17 & 16 & 28 & 272 & 8092 & 1213 &  \\ 
\texttt{balanced5   } & 7 & 6 & 12 & 42 & 588 & 205 &  \\ 
\texttt{balanced8   } & 10 & 9 & 16 & 90 & 1600 & 404 &  \\ 
\texttt{unbalanced10} & 12 & 11 & 20 & 132 & 2880 & 604 &  \\ 
\texttt{unbalanced15} & 17 & 16 & 28 & 272 & 8092 & 1213 &  \\ 
\texttt{unbalanced17} & 19 & 18 & 32 & 342 & 11552 & 1545 &  \\ 
\texttt{unbalanced20} & 22 & 21 & 36 & 462 & 17424 & 2032 &  \\ 
\texttt{unbalanced5 } & 7 & 6 & 12 & 42 & 588 & 205 &  \\ 
\hline
\multicolumn{8}{|l|} {\textbf{\cite{grossman:2017} Test Set}} \\ 
\texttt{balanced12\_random0  } & 14 & 13 & 23 & 182 & 4508 & 817 &  \\ 
\texttt{balanced12\_random1  } & 14 & 13 & 23 & 182 & 4508 & 817 &  \\ 
\texttt{balanced12\_random2  } & 14 & 13 & 23 & 182 & 4508 & 817 &  \\ 
\texttt{balanced15\_random0  } & 17 & 16 & 28 & 272 & 8092 & 1213 &  \\ 
\texttt{balanced15\_random1  } & 17 & 16 & 28 & 272 & 8092 & 1213 &  \\ 
\texttt{balanced15\_random2  } & 17 & 16 & 28 & 272 & 8092 & 1213 &  \\ 
\texttt{unbalanced17\_random0} & 19 & 18 & 32 & 342 & 11552 & 1545 &  \\ 
\texttt{unbalanced17\_random1} & 19 & 18 & 32 & 342 & 11552 & 1545 &  \\ 
\texttt{unbalanced17\_random2} & 19 & 18 & 32 & 342 & 11552 & 1545 &  \\ 
\texttt{unbalanced20\_random0} & 22 & 21 & 36 & 462 & 17424 & 2032 &  \\ 
\texttt{unbalanced20\_random1} & 22 & 21 & 36 & 462 & 17424 & 2032 &  \\ 
\texttt{unbalanced20\_random2} & 22 & 21 & 36 & 462 & 17424 & 2032 &  \\ 
\hline 
\end{tabular} 
\end{adjustbox} 
\vspace*{-0.2cm} 
\caption{Problem sizes of the test cases. The number of variables and constraints are computed with respect to the transshipment model. All test cases are available online \citep{source_code}.} 
\label{Table:Problem_Sizes} 
\end{table}

\begin{table} 
\scriptsize 
\begin{adjustbox}{center} 
\begin{tabular}{|lrrrrrrrrrrrr|}
\hline 
\multirow{2}{*} {\textbf{Test Case}} & \multicolumn{3}{|c|} {\textbf{CPLEX Transportation}} & \multicolumn{3}{c|} {\textbf{CPLEX Transshipment}} & \multicolumn{3}{c|} {\textbf{Gurobi Transportation}} & \multicolumn{3}{c|} {\textbf{Gurobi Transshipment}} \\ 
\cline{2-13} & \multicolumn{1}{|c} {Value} & Time (s) & Gap & \multicolumn{1}{|c} {Value} & Time (s) & Gap & \multicolumn{1}{|c} {Value} & Time (s) & Gap & \multicolumn{1}{|c} {Value} & Time (s) & Gap \\ 
\hline 
\multicolumn{13}{|l|} {\textbf{\cite{furman:2004} Test Set (30min time limit)}} \\ 
\texttt{10sp-la1 } &         12  &           0.04  &      & \textbf{12} & \textbf{ 0.03} &               & 12 &   0.10 &      &         12  &           0.09  &               \\ 
\texttt{10sp-ol1 } &         14  &           0.06  &      & \textbf{14} & \textbf{ 0.03} &               & 14 &   0.13 &      &         14  &           0.11  &               \\ 
\texttt{10sp1    } &         10  &           0.51  &      & \textbf{10} & \textbf{ 0.05} &               & 10 &   0.24 &      &         10  &           0.13  &               \\ 
\texttt{12sp1    } &         12  &           0.08  &      & \textbf{12} & \textbf{ 0.05} &               & 12 &   0.16 &      &         12  &           0.11  &               \\ 
\texttt{14sp1    } &         14  &         145.50  &      & \textbf{14} & \textbf{41.23} &               & 14 & 170.45 &      &         14  &         126.71  &               \\ 
\texttt{15sp-tkm } &         19  &           0.17  &      & \textbf{19} & \textbf{ 0.07} &               & 19 &   0.28 &      &         19  &           0.14  &               \\ 
\texttt{20sp1    } &         19  &              *  & 19\% &         19  &             *  &         19\%  & 19 &      * & 21\% & \textbf{19} & \textbf{     *} & \textbf{15\%} \\ 
\texttt{22sp-ph  } &         26  &           0.25  &      & \textbf{26} & \textbf{ 0.04} &               & 26 &   0.44 &      &         26  &           0.13  &               \\ 
\texttt{22sp1    } &         25  &              *  & 12\% & \textbf{25} & \textbf{    *} & \textbf{11\%} & 25 &      * & 12\% &         25  &              *  &         12\%  \\ 
\texttt{23sp1    } &         23  &              *  & 28\% &         23  &             *  &         28\%  & 23 &      * & 30\% & \textbf{23} & \textbf{     *} & \textbf{26\%} \\ 
\texttt{28sp-as1 } &         30  &           0.19  &      & \textbf{30} & \textbf{ 0.05} &               & 30 &   0.44 &      &         30  &           0.12  &               \\ 
\texttt{37sp-yfyv} &         36  &          54.80  &      &         36  &          7.36  &               & 36 &  20.40 &      & \textbf{36} & \textbf{  6.02} &               \\ 
\texttt{4sp1     } &          5  &           0.03  &      & \textbf{ 5} & \textbf{ 0.02} &               &  5 &   0.10 &      &          5  &           0.08  &               \\ 
\texttt{6sp-cf1  } & \textbf{ 6} &  \textbf{ 0.03} &      & \textbf{ 6} & \textbf{ 0.03} &               &  6 &   0.09 &      &          6  &           0.09  &               \\ 
\texttt{6sp-gg1  } & \textbf{ 3} &  \textbf{ 0.02} &      & \textbf{ 3} & \textbf{ 0.02} &               &  3 &   0.09 &      &          3  &           0.08  &               \\ 
\texttt{6sp1     } &          6  &           0.03  &      & \textbf{ 6} & \textbf{ 0.02} &               &  6 &   0.30 &      &          6  &           0.09  &               \\ 
\texttt{7sp-cm1  } & \textbf{10} &  \textbf{ 0.02} &      & \textbf{10} & \textbf{ 0.02} &               & 10 &   0.09 &      &         10  &           0.08  &               \\ 
\texttt{7sp-s1   } & \textbf{10} &  \textbf{ 0.02} &      & \textbf{10} & \textbf{ 0.02} &               & 10 &   0.09 &      &         10  &           0.08  &               \\ 
\texttt{7sp-torw1} &         10  &            0.03 &      & \textbf{10} & \textbf{ 0.02} &               & 10 &   0.10 &      &         10  &           0.09  &               \\ 
\texttt{7sp1     } & \textbf{ 7} &  \textbf{ 0.03} &      &          7  &          0.04  &               &  7 &   0.09 &      &          7  &           0.08  &               \\ 
\texttt{7sp2     } &          7  &           0.05  &      & \textbf{ 7} & \textbf{ 0.03} &               &  7 &   0.09 &      &          7  &           0.09  &               \\ 
\texttt{7sp4     } &          8  &           0.03  &      & \textbf{ 8} & \textbf{ 0.02} &               &  8 &   0.10 &      &          8  &           0.08  &               \\ 
\texttt{8sp-fs1  } &         11  &           0.03  &      & \textbf{11} & \textbf{ 0.02} &               & 11 &   0.10 &      &         11  &           0.08  &               \\ 
\texttt{8sp1     } &          9  &           0.04  &      & \textbf{ 9} & \textbf{ 0.03} &               &  9 &   0.14 &      &          9  &           0.10  &               \\ 
\texttt{9sp-al1  } &          12 &           0.04  &      & \textbf{12} & \textbf{ 0.03} &               & 12 &   0.11 &      &         12  &           0.09  &               \\ 
\texttt{9sp-has1 } & \textbf{13} &  \textbf{ 0.04} &      & \textbf{13} & \textbf{ 0.04} &               & 13 &   0.12 &      &         13  &           0.09  &               \\ 
\hline                                                                          
\multicolumn{13}{|l|} {\textbf{\cite{minlp,chen:2015} Test Set (2h time limit)}} \\ 
\texttt{balanced10  } & 25 &      * &  6\% &         24  &         1607.14  &               & 25 &      * &  4\% & \textbf{24} & \textbf{358.18} &               \\ 
\texttt{balanced12  } & 30 &      * & 16\% & \textbf{28} & \textbf{      *} & \textbf{ 7\%} & 29 &      * & 13\% &         29  &              *  &         10\%  \\ 
\texttt{balanced15  } & 36 &      * & 19\% &         37  &               *  &         17\%  & \textbf{35} & \textbf{     *} & \textbf{17\%} &         36  &              *  &         16\%  \\ 
\texttt{balanced5   } & 14 &   0.27 &      &         14  &            0.20  &               & 14 &   0.43 &      &         14  &           0.23  &               \\ 
\texttt{balanced8   } & 20 & 180.84 &      &         20  &           69.16  &               & 20 & 997.01 &      &         20  &         248.08  &               \\ 
\texttt{unbalanced10} & 25 &  36.24 &      &         25  &            7.45  &               & 25 &  46.81 &      & \textbf{25} & \textbf{ 15.97} &               \\ 
\texttt{unbalanced15} & 36 &      * &  8\% &         36  &               *  &          4\%  & 36 &      * &  8\% & \textbf{36} & \textbf{     *} & \textbf{ 4\%} \\ 
\texttt{unbalanced17} & 43 &      * & 15\% &         43  &               *  &         11\%  & 43 &      * & 13\% & \textbf{43} & \textbf{     *} & \textbf{ 9\%} \\ 
\texttt{unbalanced20} & 55 &      * & 22\% &         51  &               *  &         13\%  & 51 &      * & 17\% & \textbf{50} & \textbf{     *} & \textbf{10\%} \\ 
\texttt{unbalanced5 } & 16 &   0.09 &      & \textbf{16} & \textbf{   0.05} &               & 16 &   0.26 &      &         16  &           0.13  &               \\ 
\hline 
\multicolumn{13}{|l|} {\bf {\cite{grossman:2017} Test Set (4h time limit)}} \\ 
\texttt{balanced12\_random0}   & 29 & * & 13\% & \textbf{28} & \textbf{*} & \textbf{ 7\%} & 29 & * & 13\% & \textbf{28} & \textbf{*} & \textbf{ 7\%} \\ 
\texttt{balanced12\_random1}   & 29 & * & 13\% & \textbf{29} & \textbf{*} & \textbf{ 9\%} & 30 & * & 13\% &         29  &         *  &         10\%  \\ 
\texttt{balanced12\_random2}   & 30 & * & 16\% & \textbf{29} & \textbf{*} & \textbf{10\%} & 29 & * & 10\% & \textbf{29} & \textbf{*} & \textbf{10\%} \\ 
\texttt{balanced15\_random0}   & 36 & * & 18\% &         36  &         *  &         15\%  & 35 & * & 14\% & \textbf{36} & \textbf{*} & \textbf{13\%} \\ 
\texttt{balanced15\_random1}   & 36 & * & 18\% &         36  &         *  &         15\%  & 35 & * & 17\% & \textbf{35} & \textbf{*} & \textbf{11\%} \\ 
\texttt{balanced15\_random2}   & 36 & * & 17\% &         35  &         *  &         12\%  & 36 & * & 16\% & \textbf{35} & \textbf{*} & \textbf{11\%} \\ 
\texttt{unbalanced17\_random0} & 44 & * & 16\% & \textbf{43} & \textbf{*} & \textbf{ 9\%} & 43 & * & 13\% & \textbf{43} & \textbf{*} & \textbf{ 9\%} \\ 
\texttt{unbalanced17\_random1} & 44 & * & 16\% &         44  &         *  &         10\%  & 44 & * & 15\% & \textbf{43} & \textbf{*} & \textbf{ 6\%} \\ 
\texttt{unbalanced17\_random2} & 43 & * & 13\% & \textbf{43} & \textbf{*} & \textbf{ 9\%} & 43 & * & 13\% & \textbf{43} & \textbf{*} & \textbf{ 9\%} \\ 
\texttt{unbalanced20\_random0} & 51 & * & 16\% & \textbf{51} & \textbf{*} & \textbf{12\%} & 52 & * & 19\% &         52  &         *  &         13\%  \\ 
\texttt{unbalanced20\_random1} & 52 & * & 18\% &         52  &         *  &         15\%  & 52 & * & 19\% & \textbf{51} & \textbf{*} & \textbf{11\%} \\ 
\texttt{unbalanced20\_random2} & 51 & * & 16\% &         52  &         *  &         14\%  & 52 & * & 19\% & \textbf{50} & \textbf{*} & \textbf{10\%} \\ 
\hline 
\end{tabular} 
\end{adjustbox} 
\vspace*{-0.2cm} 
\caption{Computational results using exact solvers CPLEX 12.6.3 and Gurobi 6.5.2 with relative gap 4\%. Relative gap is defined (best incumbent - best lower bound) / best incumbent and  * indicates timeout. {\textbf{Bold} values mark the best solver result.} The transshipment formulation performs better than the transportation model: the transshipment model solves one additional problem (\texttt{balanced10}) and performs as well or better than the transportation model on 46 of the 48 test cases (with respect to time or gap closed). CPLEX solves the small models slightly faster than Gurobi while Gurobi closes more of the optimality gap for large problems. All exact method results are available online \citep{source_code}.} 
\label{Table:Exact_Methods} 
\end{table}

\begin{table} 
\scriptsize 
\color{black}
\begin{adjustbox}{center} 
\begin{tabular}{|ccccccccccccc|}
\hline 
\multirow{2}{*} {\textbf{Test Case}} & \multicolumn{3}{|c|} {\textbf{CPLEX Transportation}} & \multicolumn{3}{c|} {\textbf{CPLEX Reduced Transportation}} & \multicolumn{3}{c|} {\textbf{Gurobi Transportation}} & \multicolumn{3}{c|} {\textbf{Gurobi Reduced Transportation}} \\ 
\cline{2-13} & \multicolumn{1}{|c} {\textbf{Value}} & \textbf{Time} & \textbf{Gap} & \multicolumn{1}{|c} {\textbf{Value}} & \textbf{Time} & \textbf{Gap} & \multicolumn{1}{|c} {\textbf{Value}} & \textbf{Time} & \textbf{Gap} & \multicolumn{1}{|c} {\textbf{Value}} & \textbf{Time} & \textbf{Gap} \\ 
\hline 
\multicolumn{13}{|l|} {\textbf{\cite{furman:2004} Test Set(30min time limit)}} \\ 
\texttt{10sp-la1} & 12 & 0.04 &  & 12 & 0.05 &  & 12 & 0.10 &  & 12 & 0.13 &  \\ 
\texttt{10sp-ol1} & 14 & 0.06 &  & 14 & 0.04 &  & 14 & 0.13 &  & 14 & 0.11 &  \\ 
\texttt{10sp1} & 10 & 0.51 &  & 10 & 0.65 &  & 10 & 0.24 &  & 10 & 0.13 &  \\ 
\texttt{12sp1} & 12 & 0.08 &  & 12 & 0.08 &  & 12 & 0.16 &  & 12 & 0.13 &  \\ 
\texttt{14sp1} & 14 & 145.50 &  & 14 & 144.87 &  & 14 & 170.45 &  & 14 & 172.62 &  \\ 
\texttt{15sp-tkm} & 19 & 0.17 &  & 19 & 0.14 &  & 19 & 0.28 &  & 19 & 0.20 &  \\ 
\texttt{20sp1} & 19 & * & 19\% & 19 & * & 19\% & 19 & * & 21\% & 19 & * & 21\% \\ 
\texttt{22sp-ph} & 26 & 0.25 &  & 26 & 0.13 &  & 26 & 0.44 &  & 26 & 0.24 &  \\ 
\texttt{22sp1} & 25 & * & 12\% & 25 & * & 12\% & 25 & * & 12\% & 25 & * & 12\% \\ 
\texttt{23sp1} & 23 & * & 28\% & 23 & * & 28\% & 23 & * & 30\% & 23 & * & 30\% \\ 
\texttt{28sp-as1} & 30 & 0.19 &  & 30 & 0.09 &  & 30 & 0.44 &  & 30 & 0.23 &  \\ 
\texttt{37sp-yfyv} & 36 & 54.80 &  & 36 & 32.86 &  & 36 & 20.40 &  & 36 & 89.50 &  \\ 
\texttt{4sp1} & 5 & 0.03 &  & 5 & 0.02 &  & 5 & 0.10 &  & 5 & 0.08 &  \\ 
\texttt{6sp-cf1} & 6 & 0.03 &  & 6 & 0.02 &  & 6 & 0.09 &  & 6 & 0.08 &  \\ 
\texttt{6sp-gg1} & 3 & 0.02 &  & 3 & 0.02 &  & 3 & 0.09 &  & 3 & 0.08 &  \\ 
\texttt{6sp1} & 6 & 0.03 &  & 6 & 0.02 &  & 6 & 0.30 &  & 6 & 0.08 &  \\ 
\texttt{7sp-cm1} & 10 & 0.02 &  & 10 & 0.02 &  & 10 & 0.09 &  & 10 & 0.09 &  \\ 
\texttt{7sp-s1} & 10 & 0.02 &  & 10 & 0.02 &  & 10 & 0.09 &  & 10 & 0.12 &  \\ 
\texttt{7sp-torw1} & 10 & 0.03 &  & 10 & 0.03 &  & 10 & 0.10 &  & 10 & 0.09 &  \\ 
\texttt{7sp1} & 7 & 0.03 &  & 7 & 0.23 &  & 7 & 0.09 &  & 7 & 0.08 &  \\ 
\texttt{7sp2} & 7 & 0.05 &  & 7 & 0.04 &  & 7 & 0.09 &  & 7 & 0.09 &  \\ 
\texttt{7sp4} & 8 & 0.03 &  & 8 & 0.02 &  & 8 & 0.10 &  & 8 & 0.10 &  \\ 
\texttt{8sp-fs1} & 11 & 0.03 &  & 11 & 0.05 &  & 11 & 0.10 &  & 11 & 0.10 &  \\ 
\texttt{8sp1} & 9 & 0.04 &  & 9 & 0.07 &  & 9 & 0.14 &  & 9 & 0.13 &  \\ 
\texttt{9sp-al1} & 12 & 0.04 &  & 12 & 0.03 &  & 12 & 0.11 &  & 12 & 0.10 &  \\ 
\texttt{9sp-has1} & 13 & 0.04 &  & 13 & 0.04 &  & 13 & 0.12 &  & 13 & 0.10 &  \\ 
\hline 
\multicolumn{13}{|l|} {\textbf{\cite{minlp,chen:2015} Test Set (2h time limit)}} \\ 
\texttt{balanced10} & 25 & * & 6\% & 25 & * & 6\% & 25 & * & 4\% & 25 & * & 8\% \\ 
\texttt{balanced12} & 30 & * & 16\% & 30 & * & 16\% & 29 & * & 13\% & 29 & * & 13\% \\ 
\texttt{balanced15} & 36 & * & 19\% & 37 & * & 21\% & 35 & * & 17\% & 36 & * & 19\% \\ 
\texttt{balanced5} & 14 & 0.27 &  & 14 & 0.26 &  & 14 & 0.43 &  & 14 & 0.33 &  \\ 
\texttt{balanced8} & 20 & 180.84 &  & 20 & 885.66 &  & 20 & 997.01 &  & 20 & 214.00 &  \\ 
\texttt{unbalanced10} & 25 & 36.24 &  & 25 & 64.92 &  & 25 & 46.81 &  & 25 & 61.19 &  \\ 
\texttt{unbalanced15} & 36 & * & 8\% & 36 & * & 9\% & 36 & * & 8\% & 36 & * & 8\% \\ 
\texttt{unbalanced17} & 43 & * & 15\% & 43 & * & 15\% & 43 & * & 13\% & 43 & * & 13\% \\ 
\texttt{unbalanced20} & 55 & * & 22\% & 53 & * & 21\% & 51 & * & 17\% & 53 & * & 20\% \\ 
\texttt{unbalanced5} & 16 & 0.09 &  & 16 & 0.08 &  & 16 & 0.26 &  & 16 & 0.22 &  \\ 
\hline 
\multicolumn{13}{|l|} {\textbf{\cite{grossman:2017} Test Set (4h time limit)}} \\ 
\texttt{balanced12\_random0} & 29 & * & 13\% & 28 & * & 8\% & 29 & * & 13\% & 29 & * & 13\% \\ 
\texttt{balanced12\_random1} & 29 & * & 13\% & 29 & * & 11\% & 30 & * & 13\% & 29 & * & 10\% \\ 
\texttt{balanced12\_random2} & 30 & * & 16\% & 29 & * & 13\% & 29 & * & 10\% & 28 & * & 7\% \\ 
\texttt{balanced15\_random0} & 36 & * & 18\% & 36 & * & 18\% & 35 & * & 14\% & 36 & * & 19\% \\ 
\texttt{balanced15\_random1} & 36 & * & 18\% & 36 & * & 18\% & 35 & * & 17\% & 35 & * & 17\% \\ 
\texttt{balanced15\_random2} & 36 & * & 17\% & 37 & * & 19\% & 36 & * & 16\% & 36 & * & 19\% \\ 
\texttt{unbalanced17\_random0} & 44 & * & 16\% & 43 & * & 14\% & 43 & * & 13\% & 43 & * & 13\% \\ 
\texttt{unbalanced17\_random1} & 44 & * & 16\% & 44 & * & 15\% & 44 & * & 15\% & 44 & * & 15\% \\ 
\texttt{unbalanced17\_random2} & 43 & * & 13\% & 43 & * & 14\% & 43 & * & 13\% & 44 & * & 13\% \\ 
\texttt{unbalanced20\_random0} & 51 & * & 16\% & 51 & * & 16\% & 52 & * & 19\% & 52 & * & 19\% \\ 
\texttt{unbalanced20\_random1} & 52 & * & 18\% & 52 & * & 18\% & 52 & * & 19\% & 52 & * & 17\% \\ 
\texttt{unbalanced20\_random2} & 51 & * & 16\% & 51 & * & 17\% & 52 & * & 19\% & 53 & * & 20\% \\ 
\hline 
\end{tabular} 
\end{adjustbox} 
\vspace*{-0.2cm} 
\caption{\color{black} Computational results using exact solvers CPLEX 12.6.3 and Gurobi 6.5.2 with relative gap 4\% for solving the transportation and reduced transportation MILP models. The relative gap is (best incumbent - best lower bound) / best incumbent and  * indicates timeout. All exact method results are available online in \cite{source_code}.} 
\label{Table:Transportation_Models_Comparison} 
\end{table}

\begin{table} 
\scriptsize 
\begin{adjustbox}{center} 
\begin{tabular}{|lrrrrrrrrrr|}
\hline 
\multirow{2}{*} {\textbf{Test Case}} & \multicolumn{3}{|c|} {\textbf{Relaxation Rounding}} & \multicolumn{2}{c|} {\textbf{Water Filling}} & \multicolumn{4}{c|} {\textbf{Greedy Packing}} & \multirow{2}{*} {\textbf{CPLEX}} \\ 
\cline{2-10} & \multicolumn{1}{|c} {\textbf{FLPR}} & \textbf{LRR} & \multicolumn{1}{c|} {\textbf{CRR}} & \textbf{WFG} & \multicolumn{1}{c|} {\textbf{WFM}} & \textbf{LHM} & \textbf{LFM} & \textbf{LHM-LP} & \multicolumn{1}{c|} {\textbf{SS}} &  \\ 
\hline 
\multicolumn{11}{|l|} {\textbf{\cite{furman:2004} Test Set}} \\ 
\texttt{10sp-la1} & 16 & 17 & 16 & 16 & 15 & 15 & \emph{13} & 14 & \emph{13} & \textbf{12\phantom{*}} \\ 
\texttt{10sp-ol1} & 18 & 18 & 17 & 18 & 18 & 19 & 17 & \emph{15} & 16 & \textbf{14\phantom{*}} \\ 
\texttt{10sp1   } & 18 & 15 & 13 & 15 & 12 & 17 & 15 & \emph{11} & 12 & \textbf{10\phantom{*}} \\ 
\texttt{12sp1   } & 16 & 17 & \emph{13} & 14 & 15 & 18 & 17 & \emph{13} & \emph{13} & \textbf{12\phantom{*}} \\ 
\texttt{14sp1   } & \textbf{14} & 20 & 18 & 21 & 19 & 16 & 16 & 15 & 16 & \textbf{14\phantom{*}} \\ 
\texttt{15sp-tkm} & 20 & 22 & 20 & 23 & 25 & 21 & 22 & \textbf{19} & 21 & \textbf{19\phantom{*}} \\ 
\texttt{20sp1    } & 22 & \emph{20} & 23 & 24 & 21 & \emph{20} & 21 & \emph{20} & 21 & \textbf{19*} \\ 
\texttt{22sp-ph  } & \emph{27} & 28 & 28 & \emph{27} & 28 & 37 & \emph{27} & \emph{27} & \emph{27} & \textbf{26\phantom{*}} \\ 
\texttt{22sp1    } & 35 & 37 & 36 & 42 & 35 & 34 & 31 & \emph{27} & 29 & \textbf{25*} \\ 
\texttt{23sp1    } & 32 & 32 & 40 & 50 & 33 & 32 & 32 & \emph{26} & \emph{26} & \textbf{23*} \\ 
\texttt{28sp-as1 } & \textbf{30} & \textbf{30} & \textbf{30} & 40 & 45 & 50 & 50 & \textbf{30} & 40 & \textbf{30\phantom{*}} \\ 
\texttt{37sp-yfyv} & 44 & 40 & 45 & \emph{37} & 43 & 55 & 46 & 40 & \emph{37} & \textbf{36\phantom{*}} \\ 
\texttt{4sp1   } & \textbf{5} & \textbf{5} & \textbf{5} & \textbf{5} & \textbf{5} & \textbf{5} & \textbf{5} & \textbf{5} & \textbf{5} & \textbf{5\phantom{*}} \\ 
\texttt{6sp-cf1} & \textbf{6} & \textbf{6} & \textbf{6} & \textbf{6} & 7 & \textbf{6} & \textbf{6} & \textbf{6} & \textbf{6} & \textbf{6\phantom{*}} \\ 
\texttt{6sp-gg1} & \textbf{3} & \textbf{3} & \textbf{3} & \textbf{3} & \textbf{3} & \textbf{3} & \textbf{3} & \textbf{3} & \textbf{3} & \textbf{3\phantom{*}} \\ 
\texttt{6sp1   } &  8 &  7 &  9 &  9 & \textbf{ 6} & \textbf{ 6} & \textbf{ 6} & \textbf{ 6} & \textbf{ 6} & \textbf{6\phantom{*}} \\ 
\texttt{7sp-cm1} & \textbf{10} & \textbf{10} & \textbf{10} & \textbf{10} & \textbf{10} & \textbf{10} & \textbf{10} & \textbf{10} & \textbf{10} & \textbf{10\phantom{*}} \\ 
\texttt{7sp-s1 } & \textbf{10} & \textbf{10} & \textbf{10} & \textbf{10} & \textbf{10} & \textbf{10} & \textbf{10} & \textbf{10} & \textbf{10} & \textbf{10\phantom{*}} \\ 
\texttt{7sp-torw1} & 11 & 12 & 10 & 12 & 12 & 12 & 11 & 11 & \textbf{10} & \textbf{10\phantom{*}} \\ 
\texttt{7sp1     } & \emph{8} & 10 & 11 & 10 & \emph{8} & 9 & \emph{8} & \emph{8} & \emph{8} & \textbf{7\phantom{*}} \\ 
\texttt{7sp2     } & \textbf{7} & \textbf{7} & \textbf{7} & 9 & 9 &  \textbf{7} & \textbf{7} & \textbf{7} & \textbf{7} & \textbf{7\phantom{*}} \\ 
\texttt{7sp4     } & \textbf{8} & \textbf{8} & \textbf{8} & \textbf{8} & \textbf{8} & 10 & \textbf{8} & \textbf{8} & \textbf{8} & \textbf{8\phantom{*}} \\ 
\texttt{8sp-fs1  } & 13 & 13 & \emph{12} & \emph{12} & \emph{12} & 15 & \emph{12} & 14 & \emph{12} & \textbf{11\phantom{*}} \\ 
\texttt{8sp1     } & 11 & 11 & 10 & 14 & \textbf{9} & 10 & 10 & 10 & 10 & \textbf{9\phantom{*}} \\ 
\texttt{9sp-al1} & 16 & 17 & 16 & 16 & 15 & 15 & \emph{13} & 14 & \emph{13} & \textbf{12\phantom{*}} \\ 
\texttt{9sp-has1} & 15 & 14 & 15 & 15 & 16 & 15 & 14 & \textbf{13} & 15 & \textbf{13\phantom{*}} \\ 
\hline 
\multicolumn{11}{|l|} {\textbf{\cite{minlp,chen:2015} Test Set}} \\ 
\texttt{balanced10  } & 39 & 42 & 37 & 42 & 38 & 40 & 42 & \emph{30} & 35 & \textbf{24\phantom{*}} \\ 
\texttt{balanced12  } & 42 & 48 & 53 & 48 & 45 & 48 & 41 & \emph{37} & 41 & \textbf{28*} \\ 
\texttt{balanced15  } & 60 & 69 & 71 & 63 & 61 & 82 & 62 & \emph{43} & 51 & \textbf{37*} \\ 
\texttt{balanced5   } & 18 & 17 & 18 & 18 & 19 & 20 & 18 & \emph{15} & 19 & \textbf{14\phantom{*}} \\ 
\texttt{balanced8   } & 28 & 33 & 35 & 29 & 32 & 29 & 30 & \emph{24} & 30 & \textbf{20\phantom{*}} \\ 
\texttt{unbalanced10} & 38 & 46 & 43 & 46 & 43 & 42 & 35 & \emph{29} & 33 & \textbf{25\phantom{*}} \\ 
\texttt{unbalanced15} & 57 & 64 & 63 & 64 & 60 & 85 & 55 & \emph{44} & 49 & \textbf{36*} \\ 
\texttt{unbalanced17} & 70 & 78 & 73 & 79 & 75 & 86 & 67 & \emph{50} & 57 & \textbf{43*} \\ 
\texttt{unbalanced20} & 89 & 89 & 104 & 84 & 90 & 106 & 80 & \emph{61} & 68 & \textbf{51*} \\ 
\texttt{unbalanced5 } & 19 & 20 & \emph{18} & 21 & 22 & 19 & \emph{18} & \emph{18} & \emph{18} & \textbf{16\phantom{*}} \\ 
\hline      
\multicolumn{11}{|l|} {\textbf{\cite{grossman:2017} Test Set}} \\ 
\texttt{balanced12\_random0}   & 42 & 48 &  52 &  44 & 43 &  44 &  45 & \emph{32} & 42 & \textbf{28*} \\ 
\texttt{balanced12\_random1}   & 45 & 49 &  53 &  50 & 45 &  47 &  43 & \emph{35} & 40 & \textbf{29*} \\ 
\texttt{balanced12\_random2}   & 42 & 49 &  57 &  49 & 40 &  46 &  43 & \emph{34} & 42 & \textbf{29*} \\ 
\texttt{balanced15\_random0}   & 60 & 61 &  66 &  67 & 61 &  64 &  63 & \emph{43} & 53 & \textbf{36*} \\ 
\texttt{balanced15\_random1}   & 56 & 65 &  71 &  66 & 56 &  65 &  55 & \emph{40} & 52 & \textbf{36*} \\ 
\texttt{balanced15\_random2}   & 54 & 69 &  63 &  63 & 61 &  67 &  55 & \emph{41} & 54 & \textbf{35*} \\ 
\texttt{unbalanced17\_random0} & 74 & 80 &  86 &  81 & 65 & 102 &  72 & \emph{52} & 67 & \textbf{43*} \\ 
\texttt{unbalanced17\_random1} & 74 & 74 & 104 &  84 & 77 & 100 &  70 & \emph{55} & 56 & \textbf{44*} \\ 
\texttt{unbalanced17\_random2} & 70 & 79 &  95 &  77 & 77 & 111 &  76 & \emph{52} & 59 & \textbf{43*} \\ 
\texttt{unbalanced20\_random0} & 93 & 93 & 109 & 100 & 85 & 115 &  86 & \emph{60} & 64 & \textbf{51*} \\ 
\texttt{unbalanced20\_random1} & 83 & 89 & 117 &  92 & 88 & 114 & 100 & \emph{63} & 75 & \textbf{52*} \\ 
\texttt{unbalanced20\_random2} & 87 & 86 & 111 & 102 & 92 & 131 &  96 & \emph{69} & 74 & \textbf{52*} \\ 
\hline 
\end{tabular} 
\end{adjustbox} 
\vspace*{-0.2cm} 
\caption{Upper bounds, i.e.\ feasible solutions, computed by our heuristics and CPLEX 12.6.3 with time limit (i) 30min for the \cite{furman:2004} Test Set, (ii) 2h for the \cite{minlp,chen:2015} test set, and (iii) 4h for the \cite{grossman:2017} test set. Symbol * indicates timeout. {\textbf{Bold} values indicate the best computed value. \emph{Italic} values indicate the best heuristic result.} The proposed heuristics produce feasible as good as the exact solver for 13 of the 48 test cases. All heuristic results are available online \citep{source_code}.} 
\label{Table:Heuristic_Upper_Bounds} 
\end{table}

\begin{table} 
\scriptsize 
\begin{adjustbox}{center} 
\begin{tabular}{|lrrrrrrrrrr|}
\hline 
\multirow{2}{*} {\textbf{Test Case}} & \multicolumn{3}{|c|} {\textbf{Relaxation Rounding}} & \multicolumn{2}{c|} {\textbf{Water Filling}} & \multicolumn{4}{c|} {\textbf{Greedy Packing}} & \multirow{2}{*} {\textbf{CPLEX}} \\ 
\cline{2-10} & \multicolumn{1}{|c} {\textbf{FLPR}} & \textbf{LRR} & \multicolumn{1}{c|} {\textbf{CRR}} & \textbf{WFG} & \multicolumn{1}{c|} {\textbf{WFM}} & \textbf{LHM} & \textbf{LFM} & \textbf{LHM-LP} & \multicolumn{1}{c|} {\textbf{SS}} &  \\ 
\hline 
\multicolumn{11}{|l|} {\textbf{\cite{furman:2004} Test Set}} \\ 
\texttt{10sp-la1 } & 0.01 & 0.01 & 0.10 & 0.14 & 0.28 & 0.02 & 0.01 & 7.76 & $<0.01$ & \textbf{0.03} \\ 
\texttt{10sp-ol1 } & 0.01 & 0.01 & 0.07 & 0.10 & 0.21 & 0.02 & 0.01 & 9.83 & $<0.01$ & \textbf{0.03} \\ 
\texttt{10sp1    } & 0.01 & 0.01 & 0.10 & 0.13 & 0.26 & 0.02 & 0.01 & 7.39 & $<0.01$ & \textbf{0.05} \\ 
\texttt{12sp1    } & 0.01 & 0.01 & 0.10 & 0.25 & 0.40 & 0.04 & 0.02 & 9.79 & $<0.01$ & \textbf{0.05} \\ 
\texttt{14sp1    } & 0.01 & 0.01 & 0.13 & 0.22 & 0.42 & 0.09 & 0.04 & 24.49 & 0.02 & \textbf{41.23} \\ 
\texttt{15sp-tkm } & 0.01 & 0.01 & 0.12 & 0.23 & 0.52 & 0.17 & 0.09 & 29.46 & 0.02 & \textbf{0.07} \\ 
\texttt{20sp1    } & 0.01 & 0.01 & 0.29 & 0.37 & 0.60 & 0.52 & 0.24 & 64.63 & 0.05 & \textbf{*} \\ 
\texttt{22sp-ph  } & 0.01 & 0.02 & 0.20 & 0.46 & 0.64 & 0.92 & 0.30 & 156.76 & 0.05 & \textbf{0.04} \\ 
\texttt{22sp1    } & 0.02 & 0.02 & 0.32 & 0.39 & 0.60 & 0.76 & 0.34 & 144.77 & 0.06 & \textbf{*} \\ 
\texttt{23sp1    } & 0.04 & 0.04 & 0.34 & 0.29 & 0.52 & 0.91 & 0.40 & 239.94 & 0.07 & \textbf{*} \\ 
\texttt{28sp-as1 } & 0.01 & 0.01 & 0.09 & 0.31 & 0.57 & 1.26 & 0.38 & 227.06 & 0.05 & \textbf{0.05} \\ 
\texttt{37sp-yfyv} & 0.02 & 0.03 & 0.92 & 0.82 & 1.45 & 14.68 & 4.74 & 1435.94 & 0.50 & \textbf{7.36} \\ 
\texttt{4sp1     } & 0.01 & 0.01 & 0.04 & 0.08 & 0.16 & $<0.01$ & $<0.01$ & 0.75 & $<0.01$ & \textbf{0.02} \\ 
\texttt{6sp-cf1  } & 0.01 & 0.01 & 0.09 & 0.08 & 0.18 & $<0.01$ & $<0.01$ & 1.25 & $<0.01$ & \textbf{0.03} \\ 
\texttt{6sp-gg1  } & 0.01 & 0.01 & 0.05 & 0.07 & 0.12 & $<0.01$ & $<0.01$ & 0.42 & $<0.01$ & \textbf{0.02} \\ 
\texttt{6sp1     } & 0.01 & 0.01 & 0.07 & 0.07 & 0.14 & $<0.01$ & $<0.01$ & 1.36 & $<0.01$ & \textbf{0.02} \\ 
\texttt{7sp-cm1  } & 0.01 & 0.01 & 0.05 & 0.10 & 0.24 & 0.01 & $<0.01$ & 3.51 & $<0.01$ & \textbf{0.02} \\ 
\texttt{7sp-s1   } & 0.01 & 0.01 & 0.05 & 0.13 & 0.23 & 0.01 & $<0.01$ & 2.18 & $<0.01$ & \textbf{0.02} \\ 
\texttt{7sp-torw1} & 0.01 & 0.01 & 0.09 & 0.09 & 0.21 & 0.01 & 0.01 & 3.82 & $<0.01$ & \textbf{0.02} \\ 
\texttt{7sp1     } & 0.01 & 0.01 & 0.09 & 0.08 & 0.17 & $<0.01$ & $<0.01$ & 2.10 & $<0.01$ & \textbf{0.04} \\ 
\texttt{7sp2     } & 0.01 & 0.01 & 0.09 & 0.06 & 0.16 & $<0.01$ & $<0.01$ & 2.15 & $<0.01$ & \textbf{0.03} \\ 
\texttt{7sp4     } & 0.01 & 0.01 & 0.04 & 0.13 & 0.27 & $<0.01$ & $<0.01$ & 1.61 & $<0.01$ & \textbf{0.02} \\ 
\texttt{8sp-fs1  } & 0.01 & 0.01 & 0.09 & 0.15 & 0.27 & 0.01 & $<0.01$ & 5.84 & $<0.01$ & \textbf{0.02} \\ 
\texttt{8sp1     } & 0.01 & 0.01 & 0.10 & 0.13 & 0.25 & 0.01 & $<0.01$ & 4.75 & $<0.01$ & \textbf{0.03} \\ 
\texttt{9sp-al1  } & 0.01 & 0.01 & 0.09 & 0.13 & 0.28 & 0.02 & 0.01 & 8.18 & $<0.01$ & \textbf{0.03} \\ 
\texttt{9sp-has1 } & 0.01 & 0.01 & 0.09 & 0.12 & 0.34 & 0.02 & 0.01 & 6.99 & $<0.01$ & \textbf{0.04} \\ 
\hline           
\multicolumn{11}{|l|} {\textbf{\cite{minlp,chen:2015} Test Set}} \\ 
\texttt{balanced10  } & 0.02 & 0.02 & 0.32 & 0.43 & 0.84 & 1.15 & 0.50 & 181.13 & 0.09 & \textbf{1607.14} \\ 
\texttt{balanced12  } & 0.03 & 0.03 & 0.62 & 0.57 & 1.05 & 2.74 & 1.00 & 397.37 & 0.16 & \textbf{*} \\ 
\texttt{balanced15  } & 0.05 & 0.05 & 0.83 & 0.76 & 1.23 & 10.96 & 3.45 & 1147.96 & 0.41 & \textbf{*} \\ 
\texttt{balanced5   } & 0.01 & 0.01 & 0.11 & 0.25 & 0.50 & 0.05 & 0.03 & 13.64 & 0.01 & \textbf{0.20} \\ 
\texttt{balanced8   } & 0.02 & 0.01 & 0.21 & 0.31 & 0.65 & 0.33 & 0.15 & 68.85 & 0.04 & \textbf{69.16} \\ 
\texttt{unbalanced10} & 0.03 & 0.02 & 0.30 & 0.47 & 0.72 & 1.19 & 0.50 & 173.88 & 0.08 & \textbf{7.45} \\ 
\texttt{unbalanced15} & 0.05 & 0.04 & 0.90 & 0.74 & 1.34 & 11.28 & 3.78 & 1185.72 & 0.39 & \textbf{*} \\ 
\texttt{unbalanced17} & 0.07 & 0.07 & 1.45 & 0.95 & 1.76 & 21.25 & 8.31 & 2742.15 & 0.71 & \textbf{*} \\ 
\texttt{unbalanced20} & 0.13 & 0.13 & 3.08 & 1.25 & 2.41 & 47.55 & 15.94 & 7154.64 & 1.34 & \textbf{*} \\ 
\texttt{unbalanced5 } & 0.01 & 0.01 & 0.11 & 0.26 & 0.45 & 0.05 & 0.04 & 16.40 & 0.01 & \textbf{0.05} \\ 
\hline 
\multicolumn{11}{|l|} {\textbf{\cite{grossman:2017} Test Set}} \\ 
\texttt{balanced12\_random0  } & 0.03 & 0.03 & 0.46 & 0.59 & 1.00 & 2.51 & 1.15 & 351.10 & 0.17 & \textbf{*} \\ 
\texttt{balanced12\_random1  } & 0.03 & 0.03 & 0.46 & 0.59 & 1.11 & 2.74 & 0.99 & 398.26 & 0.16 & \textbf{*} \\ 
\texttt{balanced12\_random2  } & 0.03 & 0.03 & 0.46 & 0.60 & 0.88 & 2.61 & 0.95 & 382.57 & 0.17 & \textbf{*} \\ 
\texttt{balanced15\_random0  } & 0.04 & 0.05 & 0.83 & 0.76 & 1.49 & 8.87 & 4.13 & 1241.33 & 0.43 & \textbf{*} \\ 
\texttt{balanced15\_random1  } & 0.05 & 0.04 & 0.90 & 0.83 & 1.36 & 9.01 & 3.22 & 1041.37 & 0.42 & \textbf{*} \\ 
\texttt{balanced15\_random2  } & 0.05 & 0.05 & 0.90 & 0.82 & 1.53 & 9.26 & 3.43 & 1104.94 & 0.43 & \textbf{*} \\ 
\texttt{unbalanced17\_random0} & 0.12 & 0.11 & 1.85 & 0.95 & 1.72 & 24.25 & 8.65 & 3689.80 & 0.80 & \textbf{*} \\ 
\texttt{unbalanced17\_random1} & 0.12 & 0.12 & 1.82 & 0.79 & 1.54 & 24.08 & 8.32 & 4052.52 & 1.53 & \textbf{*} \\ 
\texttt{unbalanced17\_random2} & 0.12 & 0.12 & 1.89 & 1.00 & 1.60 & 25.60 & 9.72 & 3471.80 & 0.73 & \textbf{*} \\ 
\texttt{unbalanced20\_random0} & 0.18 & 0.17 & 3.23 & 1.22 & 2.18 & 50.67 & 18.37 & 8820.55 & 1.31 & \textbf{*} \\ 
\texttt{unbalanced20\_random1} & 0.21 & 0.19 & 3.48 & 1.24 & 2.21 & 51.19 & 19.35 & 9613.90 & 1.40 & \textbf{*} \\ 
\texttt{unbalanced20\_random2} & 0.23 & 0.22 & 3.17 & 1.28 & 2.33 & 56.47 & 17.93 & 11854.82 & 1.40 & \textbf{*} \\ 
\hline 
\end{tabular} 
\end{adjustbox} 
\vspace*{-0.2cm} 
\caption{CPU times of the heuristics and CPLEX 12.6.3 with time limit (i) 30min for the \cite{furman:2004} test set, (ii) 2h for the \cite{minlp,chen:2015} test set, and (iii) 4h for the \cite{grossman:2017} test set. An * indicates timeout. All heuristic results are available online \citep{source_code}.} 
\label{Table:Heuristic_CPU_Times} 
\end{table}

\begin{table} 
\scriptsize 
\begin{adjustbox}{center} 
\begin{tabular}{|lrrrrrrrrrrrr|}
\hline 
\multirow{3}{*} {\textbf{Test Case}} & \multicolumn{5}{|c|} {\textbf{Relaxation Rounding}} & \multicolumn{3}{c|} {\textbf{Water Filling}} & \multicolumn{4}{c|} {\textbf{CPLEX}} \\ 
\cline{2-13} & \multicolumn{2}{|c|} {\textbf{FS04}} & \multicolumn{3}{c|} {\textbf{LKM17}} & \multicolumn{1}{c|} {\textbf{FS04}} & \multicolumn{2}{c|} {\textbf{LKM17}} & \multicolumn{2}{|c|} {\textbf{FS04}} & \multicolumn{2}{c|} {\textbf{LKM17}} \\ 
\cline{2-13} & \multicolumn{1}{|c} {\textbf{FLPR}} & \textbf{LRR} & \multicolumn{1}{|c} {\textbf{FLPR}} & \textbf{LRR} & \multicolumn{1}{c|} {\textbf{CRR}} & \multicolumn{1}{c|} {\textbf{WFG}} & \textbf{WFG} & \multicolumn{1}{c|} {\textbf{WFM}} & \textbf{Value} & \textbf{Time} & \multicolumn{1}{|c} {\textbf{Value}} & \textbf{Time} \\ 
\hline 
\texttt{10sp-la1 } & 21 & 19 & 16 & 17 & 16 & 22 & 16 & 15 & 12 & 0.07 & 12 & 0.03 \\ 
\texttt{10sp-ol1 } & 22 & 17 & 18 & 18 & 17 & 23 & 18 & 18 & 14 & 0.09 & 14 & 0.03 \\ 
\texttt{10sp1    } & 14 & 14 & 18 & 15 & 13 & 21 & 15 & 12 & 10 & 2.20 & 10 & 0.05 \\ 
\texttt{12sp1    } & 17 & 18 & 16 & 17 & 13 & 18 & 14 & 15 & 12 & 0.04 & 12 & 0.05 \\ 
\texttt{14sp1    } & 27 & 21 & 14 & 20 & 18 & 27 & 21 & 19 & 14 & 33.76 & 14 & 41.23 \\ 
\texttt{15sp-tkm } & 29 & 27 & 20 & 22 & 20 & 29 & 23 & 25 & 19 & 0.70 & 19 & 0.07 \\ 
\texttt{20sp1    } & 24 & 25 & 22 & 20 & 23 & 25 & 24 & 21 & 19 & ** & 19 & * \\ 
\texttt{22sp-ph  } & 34 & 40 & 27 & 28 & 28 & 35 & 27 & 28 & 26 & 1.84 & 26 & 0.04 \\ 
\texttt{22sp1    } & 41 & 42 & 35 & 37 & 36 & 54 & 42 & 35 & 25 & ** & 25 & * \\ 
\texttt{23sp1    } & 38 & 32 & 32 & 32 & 40 & 60 & 50 & 33 & 23 & ** & 23 & * \\ 
\texttt{28sp-as1 } & 41 & 45 & 30 & 30 & 30 & 43 & 40 & 45 & 30 & 0.03 & 30 & 0.05 \\ 
\texttt{37sp-yfyv} & 67 & 59 & 44 & 40 & 45 & 61 & 37 & 43 & 36 & ** & 36 & 7.36 \\ 
\texttt{4sp1     } & 5 & 6 & 5 & 5 & 5 & 5 & 5 & 5 & 5 & 0.00 & 5 & 0.02 \\ 
\texttt{6sp-cf1  } & 6 & 6 & 6 & 6 & 6 & 7 & 6 & 7 & 6 & 0.01 & 6 & 0.03 \\ 
\texttt{6sp-gg1  } & 3 & 3 & 3 & 3 & 3 & 3 & 3 & 3 & 3 & 0.00 & 3 & 0.02 \\ 
\texttt{6sp1     } & 9 & 10 & 8 & 7 & 9 & 9 & 9 & 6 & 6 & 0.00 & 6 & 0.02 \\ 
\texttt{7sp-cm1  } & 11 & 10 & 10 & 10 & 10 & 10 & 10 & 10 & 10 & 0.00 & 10 & 0.02 \\ 
\texttt{7sp-s1   } & 10 & 10 & 10 & 10 & 10 & 10 & 10 & 10 & 10 & 0.00 & 10 & 0.02 \\ 
\texttt{7sp-torw1} & 14 & 15 & 11 & 12 & 10 & 13 & 12 & 12 & 10 & 0.03 & 10 & 0.02 \\ 
\texttt{7sp1     } & 10 & 13 & 8 & 10 & 11 & 10 & 10 & 8 & 7 & 0.01 & 7 & 0.04 \\ 
\texttt{7sp2     } & 8 & 7 & 7 & 7 & 7 & 8 & 9 & 9 & 7 & 0.04 & 7 & 0.03 \\ 
\texttt{7sp4     } & 11 & 9 & 8 & 8 & 8 & 8 & 8 & 8 & 8 & 0.00 & 8 & 0.02 \\ 
\texttt{8sp-fs1  } & 14 & 14 & 13 & 13 & 12 & 14 & 12 & 12 & 11 & 0.01 & 11 & 0.02 \\ 
\texttt{8sp1     } & 11 & 13 & 11 & 11 & 10 & 14 & 14 & 9 & 9 & 0.03 & 9 & 0.03 \\ 
\texttt{9sp-al1  } & 17 & 19 & 16 & 17 & 16 & 20 & 16 & 15 & 12 & 0.03 & 12 & 0.03 \\ 
\texttt{9sp-has1 } & 16 & 14 & 15 & 14 & 15 & 18 & 15 & 16 & 13 & 0.03 & 13 & 0.04 \\ 
\hline 
\end{tabular} 
\end{adjustbox} 
\vspace*{-0.2cm} 
\caption{Comparison of our results (labelled LKM17) with the ones reported by \citet{furman:2004} (labelled FS04). The LKM17 heuristics FLPR, LRR, and WFG perform better than their FS04 counterparts because of our improved calculation of the big-M parameter $U_{ij}$. The CPLEX comparison basically confirms that CPLEX has improved in the past 13 years: LKM17 use CPLEX 12.6.3 while FS04 use CPLEX 7.0. An * indicates 30min timeout while ** corresponds to a 7h timeout. All results are available online \citep{source_code}.} 
\label{Table:Comparisons} 
\end{table}

\begin{table} 
\scriptsize 
\begin{adjustbox}{center} 
\begin{tabular}{|lrrrrrrrrr|}
\hline 
\multirow{2}{*} {\textbf{Test Case}} & \multicolumn{3}{|c|} {\textbf{Fractional LP Rounding}} & \multicolumn{3}{c|} {\textbf{Lagrangian Relaxation Rounding}} & \multicolumn{3}{c|} {\textbf{Fractional Relaxation}} \\ 
\cline{2-10} & \multicolumn{1}{|c} {\textbf{Simple}} & \textbf{GTA97} & \textbf{LKM17} & \multicolumn{1}{|c} {\textbf{Simple}} & \textbf{GTA97} & \textbf{LKM17} & \multicolumn{1}{|c} {\textbf{Simple}} & \textbf{GTA97} & \textbf{LKM17} \\ 
\hline 
\multicolumn{10}{|l|} {\textbf{\cite{furman:2004} Test Set}} \\ 
\texttt{10sp-la1 } &         17  & \textbf{15} &         16  &         17  & \textbf{16} &         17  &          7.04  &          7.57  & \textbf{ 8.35} \\ 
\texttt{10sp-ol1 } &         20  &         19  & \textbf{18} &         22  &         23  & \textbf{18} &          8.29  &          8.86  & \textbf{ 9.94} \\ 
\texttt{10sp1    } & \textbf{18} & \textbf{18} & \textbf{18} & \textbf{14} &         17  &         15  &          7.11  &          7.24  & \textbf{ 7.39} \\ 
\texttt{12sp1    } &         17  &         17  & \textbf{16} & \textbf{17} & \textbf{17} & \textbf{17} &         10.06  &         10.12  & \textbf{10.26} \\ 
\texttt{14sp1    } &         27  &         20  & \textbf{14} &         22  & \textbf{17} &         20  &          8.79  &          8.92  & \textbf{ 9.06} \\ 
\texttt{15sp-tkm } &         27  &         27  & \textbf{20} &         25  &         25  & \textbf{22} &         11.01  &         11.47  & \textbf{14.31} \\ 
\texttt{20sp1    } & \textbf{22} & \textbf{22} & \textbf{22} &         27  &         24  & \textbf{20} & \textbf{11.75} & \textbf{11.75} & \textbf{11.75} \\ 
\texttt{22sp-ph  } &         31  &         30  & \textbf{27} &         32  &         30  & \textbf{28} &         20.15  &         20.89  & \textbf{22.23} \\ 
\texttt{22sp1    } &         45  &         37  & \textbf{35} &         40  &         44  & \textbf{37} &         13.66  &         14.04  & \textbf{15.86} \\ 
\texttt{23sp1    } &         40  &         37  & \textbf{32} &         35  &         33  & \textbf{32} &         13.31  & \textbf{13.40} & \textbf{13.40} \\ 
\texttt{28sp-as1 } &         41  & \textbf{30} & \textbf{30} &         45  & \textbf{30} & \textbf{30} &         27.51  &         27.96  & \textbf{28.45} \\ 
\texttt{37sp-yfyv} &         56  &         53  & \textbf{44} &         42  &         42  & \textbf{40} &         31.96  &         31.93  & \textbf{32.28} \\ 
\texttt{4sp1     } & \textbf{ 5} & \textbf{ 5} & \textbf{ 5} & \textbf{ 5} & \textbf{ 5} & \textbf{ 5} &          4.03  &          4.06  & \textbf{ 4.25} \\ 
\texttt{6sp-cf1  } & \textbf{ 6} & \textbf{ 6} & \textbf{ 6} & \textbf{ 6} & \textbf{ 6} & \textbf{ 6} &          4.10  &          4.10  & \textbf{ 4.18} \\ 
\texttt{6sp-gg1  } & \textbf{ 3} & \textbf{ 3} & \textbf{ 3} & \textbf{ 3} & \textbf{ 3} & \textbf{ 3} & \textbf{ 3.00} & \textbf{ 3.00} & \textbf{ 3.00} \\ 
\texttt{6sp1     } &          8  & \textbf{ 7} &          8  &          8  & \textbf{ 7} & \textbf{ 7} & \textbf{ 4.00} & \textbf{ 4.00} & \textbf{ 4.00} \\ 
\texttt{7sp-cm1  } &         11  & \textbf{10} & \textbf{10} & \textbf{10} & \textbf{10} & \textbf{10} &          6.61  &          7.15  & \textbf{ 8.40} \\ 
\texttt{7sp-s1   } & \textbf{10} & \textbf{10} & \textbf{10} & \textbf{10} & \textbf{10} & \textbf{10} &          7.83  &          7.83  & \textbf{10.00} \\ 
\texttt{7sp-torw1} &         15  &         15  & \textbf{11} &         14  &         13  & \textbf{12} &          5.68  &          5.84  & \textbf{ 6.56} \\ 
\texttt{7sp1     } &         10  &         10  & \textbf{ 8} &          9  &          9  & \textbf{10} &          5.00  &          5.00  & \textbf{ 5.01} \\ 
\texttt{7sp2     } &          9  & \textbf{ 7} & \textbf{ 7} & \textbf{ 7} & \textbf{ 7} & \textbf{ 7} & \textbf{ 4.37} & \textbf{ 4.37} & \textbf{ 4.37} \\ 
\texttt{7sp4     } &         11  &         11  & \textbf{ 8} &          9  &          9  & \textbf{ 8} &          7.01  &          7.01  & \textbf{ 7.11} \\ 
\texttt{8sp-fs1  } &         15  & \textbf{13} & \textbf{13} &         14  & \textbf{13} & \textbf{13} &          6.89  &          7.50  & \textbf{ 8.69} \\ 
\texttt{8sp1     } & \textbf{11} & \textbf{11} & \textbf{11} &         12  &         12  & \textbf{11} &          6.15  &          6.22  & \textbf{ 6.30} \\ 
\texttt{9sp-al1  } &         17  & \textbf{15} &         16  &         17  & \textbf{16} &         17  &          7.04  &          7.57  & \textbf{ 8.35} \\ 
\texttt{9sp-has1 } &         16  & \textbf{15} & \textbf{15} & \textbf{14} &         15  & \textbf{14} &          6.91  &          7.14  & \textbf{ 9.98} \\ 
\hline                                                    
\multicolumn{10}{|l|} {\textbf{\cite{minlp,chen:2015} Test Set}} \\ 
\texttt{balanced10  } &  61 & 46 & \textbf{39} &  46 &         43  & \textbf{42} & 13.51 & 13.92 & \textbf{15.29} \\ 
\texttt{balanced12  } &  71 & 52 & \textbf{42} &  57 &         56  & \textbf{48} & 15.69 & 16.16 & \textbf{17.48} \\ 
\texttt{balanced15  } &  91 & 63 & \textbf{60} &  91 & \textbf{69} & \textbf{69} & 18.84 & 19.31 & \textbf{21.56} \\ 
\texttt{balanced5   } &  26 & 19 & \textbf{18} &  20 &         22  & \textbf{17} &  8.09 &  8.40 & \textbf{ 8.95} \\ 
\texttt{balanced8   } &  42 & 33 & \textbf{28} &  38 &         35  & \textbf{33} & 11.54 & 11.88 & \textbf{12.76} \\ 
\texttt{unbalanced10} &  52 & 49 & \textbf{38} &  53 &         54  & \textbf{46} & 14.31 & 15.05 & \textbf{16.96} \\ 
\texttt{unbalanced15} &  73 & 60 & \textbf{57} &  76 & \textbf{62} &         64  & 19.62 & 20.59 & \textbf{23.17} \\ 
\texttt{unbalanced17} &  96 & 78 & \textbf{70} & 101 &         84  & \textbf{78} & 21.90 & 23.53 & \textbf{27.48} \\ 
\texttt{unbalanced20} & 132 & 95 & \textbf{89} & 137 &         99  & \textbf{89} & 25.89 & 27.72 & \textbf{32.43} \\ 
\texttt{unbalanced5 } &  23 & 22 & \textbf{19} &  21 &         23  & \textbf{20} &  8.34 &  8.82 & \textbf{10.93} \\ 
\hline              
\multicolumn{10}{|l|} {\textbf{\cite{grossman:2017} Test Set}} \\ 
\texttt{balanced12\_random0  } &  73 &  56 & \textbf{42} &  61 &  52 & \textbf{48} & 15.76 & 16.21 & \textbf{17.51} \\ 
\texttt{balanced12\_random1  } &  62 &  56 & \textbf{45} &  60 &  54 & \textbf{49} & 15.67 & 16.06 & \textbf{17.37} \\ 
\texttt{balanced12\_random2  } &  66 &  51 & \textbf{42} &  53 &  57 & \textbf{49} & 15.67 & 16.14 & \textbf{17.40} \\ 
\texttt{balanced15\_random0  } &  93 &  68 & \textbf{60} &  75 &  73 & \textbf{61} & 18.59 & 19.19 & \textbf{21.47} \\ 
\texttt{balanced15\_random1  } &  96 &  68 & \textbf{56} &  79 &  73 & \textbf{65} & 18.86 & 19.38 & \textbf{21.59} \\ 
\texttt{balanced15\_random2  } & 102 &  64 & \textbf{54} &  86 &  76 & \textbf{69} & 18.73 & 19.41 & \textbf{21.95} \\ 
\texttt{unbalanced17\_random0} & 106 &  77 & \textbf{74} & 108 &  95 & \textbf{80} & 22.48 & 23.97 & \textbf{27.64} \\ 
\texttt{unbalanced17\_random1} & 116 &  82 & \textbf{74} &  99 &  91 & \textbf{74} & 22.43 & 23.89 & \textbf{27.66} \\ 
\texttt{unbalanced17\_random2} & 101 &  84 & \textbf{70} &  94 &  92 & \textbf{79} & 21.99 & 23.61 & \textbf{27.74} \\ 
\texttt{unbalanced20\_random0} & 131 &  95 & \textbf{93} & 136 & 103 & \textbf{93} & 26.02 & 27.91 & \textbf{32.49} \\ 
\texttt{unbalanced20\_random1} & 138 &  91 & \textbf{83} & 139 & 104 & \textbf{89} & 26.01 & 27.74 & \textbf{32.64} \\ 
\texttt{unbalanced20\_random2} & 138 & 102 & \textbf{87} & 131 & 100 & \textbf{86} & 25.68 & 27.67 & \textbf{32.60} \\ 
\hline 
\end{tabular} 
\end{adjustbox} 
\vspace*{-0.2cm} 
\caption{This table compares the effect of three different methods for computing the big-M parameter $U_{ij}$: (i) simple greedy, (ii) the \citet{gundersen:1997} (GTA97) method and (iii) our greedy Algorithm MHG (LKM17).
{\color{black} \textbf{Bold} values mark the best result for each of the heuristics.}
LKM17 outperforms the other two big-M computation methods by finding smaller feasible solutions via both Fractional LP Rounding and Lagrangian Relaxation Rounding. It also acheives the tightest fractional relaxation for all test instances.
} 
\label{Table:BigM_Comparisons} 
\end{table}

\begin{table} 
\scriptsize 
\begin{adjustbox}{center} 
\begin{tabular}{|ccccccccc|}
\hline 
\multirow{3}{*} {\textbf{Test Case}} & \multicolumn{2}{|c|} {\textbf{Relaxation Rounding}} & \multicolumn{2}{c|} {\textbf{Water Filling}} & \multicolumn{2}{c|} {\textbf{Greedy Packing}} & \multicolumn{2}{c|} {\textbf{CPLEX}} \\ 
 & \multicolumn{2}{|c|} {\textbf{FLPR}} & \multicolumn{2}{|c|} {\textbf{WFG}} & \multicolumn{2}{|c|} {\textbf{SS}} & \multicolumn{2}{|c|} {\textbf{Transshipment}} \\ 
\cline{2-9} & \multicolumn{1}{|c} {\textbf{Value}} & \textbf{Time} & \multicolumn{1}{|c} {\textbf{Value}} & \textbf{Time} & \multicolumn{1}{|c} {\textbf{Value}} & \textbf{Time} & \multicolumn{1}{|c} {\textbf{Value}} & \textbf{Time} \\ 
\hline 
\texttt{large\_scale0} & 233 & 8.84 & 306 & 58.52 & \emph{233} & \emph{642.94} & \textbf{175} & \textbf{*} \\ 
\texttt{large\_scale1} & 273 & 15.59 & 432 & 54.53 & \emph{\textbf{218}} & \emph{\textbf{652.00}} & 219 & * \\ 
\texttt{large\_scale2} & 279 & 41.83 & 497 & 54.46 & \emph{242} & \emph{670.32} & \textbf{239} & \textbf{*} \\ 
\hline 
\end{tabular} 
\end{adjustbox} 
\vspace*{-0.2cm} 
\caption{Upper bounds, i.e.\ feasible solutions, for large-scale instances computed by the least time consuming heuristics of each type and CPLEX 12.6.3 transshipment model with 4h timeout. Symbol * indicates timeout. {\textbf{Bold} marks the best upper bound. \emph{Italic} marks the best heuristic result. In instance \texttt{large\_scale1}, heuristic LFM computes the best heuristic result.}} 
\label{Table:Large_Scale_Results} 
\end{table}

\end{document}